%% file: main-arxiv.tex
\documentclass[a4paper]{amsart}

\usepackage[english]{babel}
\usepackage{amsmath}
\usepackage{amssymb}
\usepackage{amsthm}
\usepackage{cite}
\usepackage{tikz}
\usepackage{url}
\usepackage{graphicx}
\usepackage{multirow}
\usepackage{pstricks}
\usepackage{enumerate}
\usepackage{longtable}

\newcommand{\Q}{\mathbb{Q}}
\newcommand{\Z}{\mathbb{Z}}
\newcommand{\R}{\mathbb{R}}
\newcommand{\N}{\mathbb{N}}

\newcommand{\frako}{\mathfrak{o}}
\newcommand{\frakp}{\mathfrak{p}}
\newcommand{\Aut}{\text{Aut}}
\newcommand{\Gal}{\text{Gal}}

\newcommand{\rank}{\text{rank}}

\newcommand{\tr}{\text{tr}}
\newcommand{\hite}{\text{ht}}
\newcommand{\lz}[1]{\left\langle #1 \right\rangle}
\newcommand{\pz}[1]{\left( #1 \right)}

\newtheorem{lemma}{Lemma}
\newtheorem{theorem}{Theorem}
\newtheorem{corollary}{Corollary}
\newtheorem{algorithm}{Algorithm}

\theoremstyle{definition}

\title{The classification of rank 3 reflective hyperbolic lattices over $\Z(\sqrt{2})$}

\author{Alice Mark}

\date{}

\begin{document}

\begin{abstract}

There are 432 strongly squarefree symmetric bilinear forms of signature $(2,1)$ defined over $\Z[\sqrt{2}]$ whose integral isometry groups are generated up to finite index by finitely many reflections. We adapted Allcock's method (based on Nikulin's) of analysis for the $2$-dimensional Weyl chamber to the real quadratic setting, and used it to produce a finite list of quadratic forms which contains all of the ones of interest to us as a sub-list.  The standard method for determining whether a hyperbolic reflection group is generated up to finite index by reflections is an algorithm of Vinberg.  However, for a large number of our quadratic forms the computation time required by Vinberg's algorithm was too long.  We invented some alternatives, which we present here.

\end{abstract}

\maketitle

\section{Introduction, statement of the main theorem}\label{introduction}

Hyperbolic reflection groups arise as discrete subgroups of automorphism groups of quadratic forms of signature $(n,1)$.  The integral automorphism group of a quadratic form has a subgroup generated by all of its reflections.  If the quadratic form has signature $(n,1)$, those reflections act on hyperbolic space by hyperbolic reflections.  The fundamental domain for this action is a hyperbolic Coxeter polyhedron.  
Arithmetic subgroups of algebraic groups are a particularly well studied class of discrete subgroups of algebraic groups that have finite co-area.  The quadratic form $Q$ defined over the totally real field $F$ is arithmetic if for all nontrivial embeddings $\sigma$ of $F$ into $\R$, the quadratic form $\sigma Q$ is definite.

Let $Q$ be an arithmetic quadratic form of signature $(n,1)$ defined over the ring of integers in a totally real finite extension of $\Q$.  The integral automorphism group of $Q$ is the group of isometries of $Q$ that preserve an integral lattice $L$ in $\R^{n,1}$.  The automorphism group of $L$ is a discrete subgroup of $O(Q)\cong O(n,1)$.  It acts on the model of hyperbolic space obtained as the projectivization of the future cone in $\R^{n,1}$.  The subgroup that acts by reflections is a hyperbolic reflection group.  We say that $Q$ (or $L$) is reflective if its integral automorphism group is generated up to finite index by finitely many reflections.

This paper is devoted to proving the following:
\begin{theorem}\label{main}
There are 432 rank $3$ strongly squarefree reflective arithmetic hyperbolic lattices defined over $\Z[\sqrt{2}]$.
\end{theorem}

The structure of our proof is based on Allcock's in \cite{allcock2012reflective}.   We modify several of his lemmas so that we can apply them to lattices with ground fields that are real quadratic extensions of $\Q$.  Our modifications are inspired by the modifications Bugaenko made to Vinberg's algorithm in \cite{bugaenko1984groups} \cite{bugaenko1990reflective} and \cite{bugaenko1992arithmetic}, where he applied the algorithm to lattices defined over totally real finite extensions of $\Q$.  What is new in this paper is that we do not use Vinberg's algorithm to determine reflectivity of our lattices.  Instead we ``walk around'' the edges of the fundamental chamber in order.  We will be more precise about what this means later on.

There are finitely many discrete maximal hyperbolic reflection groups with finite covolume that are arithmetic in $O(n,1)$. In 1985, Vinberg proved that there are no reflective arithmetic quadratic forms in dimension $\geq 30$ \cite{vinberg1985hyperbolic}.  In a series of papers from the early 1980s, Nikulin showed that there are only finitely many in any dimension greater than 9 \cite{nikulin1981arithmetic} \cite{nikulin1982classification}.  In their 2006 paper about genus 0 fuchsian groups, Long, Maclachlan and Reid used covolume bounds to prove finiteness in dimension 2 \cite{long2006arithmetic}.  Using similar methods Agol proved finiteness in dimension 3 \cite{agol2006kleinian}.  In two papers published at around the same time using different methods, Nikulin \cite{nikulin2006finiteness}, and Agol, Belolipetsky, Storm, and Whyte \cite{agol2006finiteness} both finished the proof of finiteness by proving it for dimensions 4 through 9.  Belolipetsky's recent survey paper \cite{belolipetsky2015arithmetic} gives a thorough overview the subject of hyperbolic reflection groups, including the history of this finitenesss theorem and many related results and problems.

Allcock's classification of the rank 3 hyperbolic reflective lattices over $\Z$ is most similar to our own.  He restricted his classification to the reflective ones because of their role the study of both Kac-Moody algebras and K3 surfaces. Nikulin's classification of hyperbolic root systems of rank 3 is also similar.  In a paper in 3 parts, he classifies all the hyperbolic quadratic forms from a wider class of lattices which are ``almost reflective.'' Part I contains the strongly squarefree reflective lattices, which are all of the ``essential versions'' of the lattices on Allcock's list.  Allcock's list only contains reflective lattices, but it contains the ones that are not strongly-squarefree as well.

In Section
2 we give the necessary background about lattices and quadratic forms over number fields.  In particular we highlight the things that are different from \cite{allcock2012reflective} due to the fact that we are working over a quadratic extension of $\Q$.  Some of these differences are quite trivial while others are fairly substantial.  In his very thorough book \cite{omeara1999introduction} on quadratic forms, O'Meara fully develops the theory of quadratic forms over Dedekind domains of arithmetic type, and we refer the reader there for any further details.  

In Section
3 we prove versions of the lemmas from \cite{allcock2012reflective} that have been modified so that they now apply to lattices defined over real quadratic extensions of $\Q$.  The fundamental chamber of the reflection part of the automorphism group of a reflective lattices is a hyperbolic polygon with finite volume and finitely many sides.  All such polygons have ``thin parts,'' which Allcock made precise by introducing three types of configurations called ``short edge,'' ``short pair,'' and ``close pair.''   One of these must occur in any finite sided polygon with finite volume.  We used these lemmas to generate a finite list of lattices defined over $\Q[\sqrt{2}]$, which we will pare down into our classification in 
Section
 4.  

As a consequence of our modifications of the lemmas, we get an upper bound on the discriminant of a real quadratic field over which a reflective arithmetic lattice can be defined.  There are no reflective arithmetic lattices with real quadratic ground fields of discriminant larger than 27935.  We prove this, and in fact we get even better bounds for the short pair and close pair cases, in Theorem \ref{discrBounds}. 

In Section 
 4 we explain our method for determining which of the lattices from Section 
3 are reflective.  We introduce a method for finding an element of the automorphism group of a lattice that takes the fundamental chamber to its nearest translate along a line in hyperbolic space containing an edge of the chamber.  We use this in a fast algorithm for finding all the edges of a boundary component of that fundamental domain.  This algorithm is what we call ``walking,'' and it is used to determine whether a lattice is reflective.  We also describe how we resolved the few cases for which walking by itself was not fast enough.

The classification itself is in the appendix.  We have organized the lattices into tables by the number of sides of the fundamental polygon.  The smallest are triangles, of which there are 3.  We recognize them from Takeuchi's list of arithmetic triangle fuchsian groups \cite{takeuchi1977arithmetic}.  There are 8 triangles on Takeuchi's list with ground field $\Q(\sqrt{2})$.  Among these eight, 3 are maximal, and these are the 3 that appear on our list.  The largest polygons on our list are $24$-gons with order 2 rotation.  The entries in our table sorted by the norm in $\Q(\sqrt{2})/\Q$ of the determinant of the quadratic form.  For each entry, we give the determinant, the shape of the fundamental domain for the action on hyperbolic space, the quadratic form, and a list of simple roots.

All of the computations were done using the PARI/GP library \cite{PARI2} in C/C++.

\section{Background}\label{background}

Throughout this paper $F$ will be a totally real quadratic extension of $\Q$ whose ring of integers $\frako$ (or $\frako_F$ if it's ambiguous) is a PID.  We fix an embedding of $F$ into $\R$.

\subsection{Lattices, quadratic forms, Minkowski space}

A lattice $L$ is a projective $\frako$-module\footnote{Since $\frako$ is a PID, being projective is the same as being free.} with an $F$ valued symmetric bilinear form.  The rank of $L$ is the rank of its associated vector space $V = L\otimes F$.  We denote the norm of a vector $v$ with respect to the bilinear form on $V$ by $v^2$.  We say that $L$ is integral if the bilinear form is $\frako$-valued.  The scale of $L$ is the ideal generated by all of its inner products.  We say that $L$ is unscaled if its scale is $\frako$.  All of our lattices will be unscaled unless otherwise specified.  If the bilinear form on $L$ has signature $(n,1)$, then $V\otimes \R$ is Minkowski space, which means it is a real $n+1$ dimensional vector space with a bilinear that is equivalent over $\R$ to the standard quadratic form of signature $(n,1)$
\begin{equation}\label{qf21}
f_0 = -x_0^2+x_1^2+\ldots+x_n^2
\end{equation}

The vectors with norm $0$ form a cone in Minkowski space and are called light-like.  The vectors of positive norm lie outside the cone and are called space-like.  The vectors of negative norm lie inside the cone and are called time-like.  The set of negative norm vectors 
$$\mathfrak{C} = \{v\in V:v^2 < 0\}$$
has two components.  We fix a vector $p$ of negative norm, and declare the future cone $\mathfrak{C}^+$ to be those vectors in $\mathfrak{C}$ whose inner product with $p$ is negative, that is
$$\mathfrak{C}^+ = \{v\in\mathfrak{C}:v\cdot p < 0\}$$

\begin{figure}[ht]
\includegraphics[width=200pt]{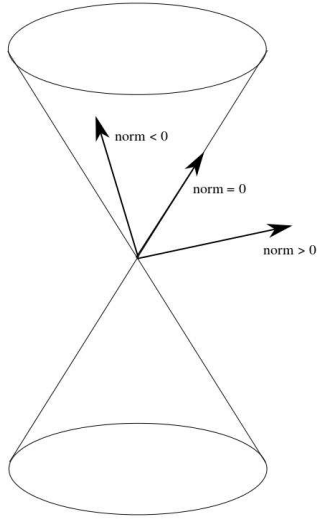}
\caption{Light cone in Minkowski space}
\label{lightCone}
\end{figure}

From $\mathfrak{C}^+$ we obtain a model of $n$-dimensional hyperbolic space $\Lambda^n$ by taking the quotient of $\mathfrak{C}^+$ by the equivalence relation $v\sim\lambda v$ for all  real $\lambda > 0$.
$$\Lambda^n = \mathfrak{C}^+/\sim$$

If $v\in V$ is a space-like vector, then the restriction of the bilinear form to the hyperplane $v^{\perp}$ has signature $(n-1,1)$.  It cuts through $\mathfrak{C}^+$, and its image in $\Lambda^n$ is a geodesic hyperplane.

\subsection{Units in number fields}
Denote the group of units in $\frako$ by $U$ (or by $U(F)$ if we need to clarify which field it is in).  Let $\alpha_0$ be the fundamental unit in $U$.  By convention we take $\alpha_0 >0$.   We have $U = \lz{\alpha_0} \times \{\pm 1\}$.  We define two subroups $U^+$ and $U^+_1$ to be the group of positive units, and the group of positive units with norm 1 respectively.
$$U^+ = \{\alpha\in U:\alpha >0\} = \lz{\alpha_0}$$
$$U^+_1 = \{\alpha\in U^+: N_{K/F}(\alpha) = 1\}$$

In a quadratic field $F$ with norm $N_{F/\Q}:F\rightarrow\Q$, one of the following is true.
\begin{enumerate}
\item[(U1)] $N_{F/\Q}(\alpha_0) = 1$, and $U^+_1$ is the union of two square classes in $U^+$, 
$$\{\alpha_0^{2k}:n\in \Z\}\text{ and }\{\alpha_0^{2k+1}:k\in \Z\}$$
\item[(U2)] $N_{F/\Q}(\alpha_0) = -1$, and $U^+_1$ is a single square-class in $U^+$
$$\{\alpha_0^{2k}:k\in\Z\}$$
\end{enumerate}

We wish to draw attention to the fact that units may be positive or negative, and independently have positive or negative norm.  We point this out because of what scaling by units does to vectors in lattices and to their norms.  Let $v\in L$ be a lattice vector, and $\alpha\in U$ a unit.  Then $v$ and $\alpha v$ generate the same sublattice of $L$,  $\lz{\alpha v} = \lz{v}$.  This may mean that we want to consider $\alpha v$ equivalent to $v$.  However, if direction matters then we only want to consider them equivalent if $\alpha > 0$  since if $\alpha < 0$ then $\alpha v $ points the opposite direction to $v$.  For us, $v$ and $v'$ will be equivalent if $v' = \alpha v$ for some $\alpha >0$ that is a unit.

The norm of $\alpha v$ is $\alpha^2 v^2$.  We would like equivalent vectors to have equivalent norms.  Since $\alpha v$ and $v$ are equivalent vectors, we say that two norms are equivalent if they differ by a factor of a unit squared.  The arithmeticity condition, which we will discuss later, will imply that $N_{F/\Q}(v^2)$ must always have the same sign as $v^2$, so $\alpha v^2$ could only be a if $\alpha\in U^+_1$.  Let $n\in\frako_F$ be a number that is a norm of a vector in $L$.  There are two possibilities, corresponding to (U1) and (U2) above.
\begin{enumerate}
\item[(N1)] There are two equivalence classes of numbers in $\frako$ associated to the norm $n$, namely 
$$\{\alpha n: \alpha\in U^+_1\text{ is not a square} \}\text{ and }\{\alpha n: \alpha\in U^+_1\text{ is a square} \}$$
\item[(N2)] There is one equivalence class of numbers in $\frako$ associated to the norm the norm $n$, namely 
$$\{\alpha n:\alpha\in U^+_1\}$$
\end{enumerate}
We note that when $F = \Q(\sqrt{2})$, (U2) and (N2) hold.  

\subsection{Roots}
A vector $v\in L$ is called primitive if whenever we have
$$v = \alpha v'$$
for some $v'\in L$ and $\alpha\in \frako$, then $\alpha$ is a unit.  A \emph{root} is a primitive space-like vector $r\in L$ such that the reflection negating $r$ and fixing $r^{\perp}$ is an automorphism of $L$. Reflection with respect to $r$ is denoted $R_r$ and is given by the formula
\begin{equation}\label{reflection}
R_r(v) = v-\frac{2r\cdot v}{r^2}r
\end{equation}
thus the condition for $r$ to be a root can be stated as 
\begin{equation}\label{rootCondition}
r\cdot v \in \frac{r^2}{2}\frako\text{ for all }v\in L
\end{equation}

\subsection{Arithmeticity}

Let $Q$ be a quadratic form defined over $\frako_F$.  The full real isometry group $O(Q)(\R)$ is an algebraic group defined over $\frako$.  The integral isometry group $\Gamma$ is a discrete subgroup preserving a lattice.  Let $\sigma\in\Gal(F/\Q)$ be a nontrivial element of the Galois group.  Since $Q$, $O(Q)(\R)$ are defined over $F$, we may apply $\sigma$ to $Q$ and to any element of $O(Q)(\R)$.  If $Q$ has signature $(2,1)$, $O(Q)(\R)\cong O(2,1)$.  We say that $\Gamma$ is arithmetic if for all non-identity elements $\sigma \in \Gal(F/\Q)$, the isometry group $O(\sigma Q)(\R)$ is isomorphic to $O(3)(\R)$.  We say that $Q$ is arithmetic (or that the matrix for $Q$ is an arithmetic matrix) if $Q$ has signature $(n,1)$, and $\sigma Q$ is definite for all nontrivial $\sigma\in\Gal(F/\Q)$.

In particular, if $F$ is a real quadratic field, then there is only one nontrivial element of the Galois group of $F/\Q$.  Since we can scale $Q$ by a positive element of $\frako_F$ with negative galois conjugate, we may assume that $\sigma Q$ is positive definite.  In particular, if $U(F)$ contains units with norm $-1$, we may assume that $Q$ is unscaled and $\sigma Q$ is positive definite. We siummarize the necessary facts about arithmeticity in the following lemma.

\begin{lemma}\label{arithCor}
Let $L$ be a lattice of signature $(n,1)$ defined over $F$ whose reflection group is arithmetic.
\begin{enumerate}[(i)]
\item $L$ contains no nontrivial vectors of norm $0$.
\item The fundamental domain $P$ for the action of $L$ on $\Lambda^n$ has no ideal vertices.
\item If $r$ is a root, then $N_{F/\Q}(r^2) > 0$.  Equivalently, its norm $r^2$ has positive Galois conjugate $\overline{r}^2$.
\item If $r$ and $s$ are any vectors of $L$ then there is a bound on the Galois conjugate of $r\cdot s$,
$$0 \leq |\overline{r\cdot s}| < \sqrt{\overline{r^2s^2}}$$
\end{enumerate}
\end{lemma}
The notation $\overline{v}^2$ and $\overline{r\cdot s}$ may seem sloppy.  However, since Galois conjugation is a homomorphism, it commutes with taking inner products.  So the notation is fine as long as the inner products of conjugate vectors are taken with respect to the conjugate quadratic form.
\begin{proof} Let $v\in L$.  
\begin{enumerate}[(i)]
\item If $v^2 = 0$, $\overline{v}^2 = 0$.  Since the conjugate of the quadratic form is positive definite, this is only possible if $v$ is the zero vector.
\item At any vertex of $P$ there is a lattice vector $p\in\lz{r_1,\ldots r_n}^{\perp}$ where the $r_i$ are roots for the stabilizer of $p$.  If $P$ were an ideal vertex, then $p$  would live in $\partial \mathfrak{C}$.  But by (i), $L$ has no nontrivial vectors of norm 0.
\item Since $\sigma Q$ is positive definite, $\overline{r}^2 >0$.  Since $r$ is a root, $r^2 >0$. Thus
$$N_{F/\Q}(r^2) = r^2\overline{r}^2 >0$$
\item This is a property of positive definite quadratic forms.
\end{enumerate}
\end{proof}

\subsection{Sublattices and glue}
Recall that $\frako$ is a PID.  If the bilinear form on $L$ is non-degenerate, then $L$ has a generating set of size $\rank(L)$ which we call a basis for $L$.

A sublattice $M\subset L$ means an $\frako$-submodule of $L$.  We say that $M$ is saturated if we have
$$M = (M\otimes F)\cap L$$

If $\rank(M) = \rank(L) = n$, then $M$ has finite index in $L$, and there is a basis $\{v_1,\ldots, v_n\}$ for $L$ and $a_1,\ldots a_n\in\frako$ such that $\{a_1v_1,\ldots,a_nv_n\}$ is a basis for $M$ and $(a_n)\supset\ldots\supset (a_1)$ (Theorem 81:11 in \cite{omeara1999introduction}).  The $a_i$ are the invariant factors of a matrix taking a basis for $L$ to a basis for $M$.

Let $L$ be a lattice with saturated sublattices $M_1$ and $M_2$, each the orthogonal complement of the other.  Let $\pi_i$ denote projection onto the vector space spanned by $M_i$.  We call $v\in L$ a glue vector for $M_1$ and $M_2$ if it is not contained in the lattice $M_1\oplus M_2$.  If $L$ is generated by $M_1$, $M_2$ and a glue vector $g$, we write
$$L = M_1\oplus_g M_2$$

\subsection{Duals}
Every non-degenerate lattice $L$ has a dual lattice 
$$L^* = \{v\in L\otimes F: v\cdot w\in\frako \text{ for all } w\in L\}$$
Given a basis $v_0,\ldots v_n$ for $L$, the corresponding dual basis for $L^*$ is $\hat{v}_0,\ldots,\hat{v}_n$, defined by 
\begin{equation}\label{dualBasis}
\hat{v}_i\cdot v_j = \left\{\begin{array}{ll}
1 & \text{if } i = j\\ 
0 & \text{if } i \neq j
\end{array}\right.
\end{equation}

If $L$ is an integral lattice, then $L\subset L^*$.  In this case, we define the discriminant $\frako$-module to be  
$$\Delta(L):=L^*/L$$
If $L$ is non-degenerate and finite dimensional, then $\Delta(L)$ is a torsion $\frako$-module, and so it has an in invariant factor decomposition
$$L^*/L \cong \bigoplus_{i=1}^n\frako/a_i\frako$$
where the $a_i\in \frako$ satisfy $(a_n) \supset (a_{n-1}) \supset \ldots \supset (a_1)$.  The inner product matrix $Q$ of the basis $v_i$ for $L$ is the matrix that writes the vectors $v_i$ in terms of the dual basis vectors $\hat{v}_i$.  Thus all of the information about the structure of $\Delta(L)$ is encoded in $Q$.  In particular, the ideal generated by $\det(Q)$ is the same as the ideal generated by the product of the $a_i$'s.  This ideal is independent of our choice of basis, and so we refer to it as the determinant ideal of $L$, denoted $\det(L)$.

\begin{lemma}\label{subDeterm}
Let $M\subset L$ be a sublattice of finite index such that
$$L/M = \bigoplus_{i=1}^n\frako/b_i\frako$$
with $(b_n)\supset \ldots \supset (b_1)$.  Let 
$$b = \prod_{i=1}^nb_i$$
Then $\det(M) = b^2\det(L)$.
\end{lemma}
\begin{proof}
There is a basis $\{v_1,\ldots,v_n\}$ for $L$ such that $\{b_1v_1,\ldots b_nv_n\}$ is a basis for $M$.   We have 
$$M\subset L\subset L^*\subset M^*$$
The corresponding dual bases for $L^*$ and $M^*$ are 
$$\{\hat{v}_1,\ldots, \hat{v}_n\}\text{ and }\{b_1^{-1}\hat{v}_1,\ldots,b_n^{-1}\hat{v}_n\}$$
respectively.  Thus $M^*/L^*$ has the same elementary divisor decomposition as $L/M$, so the product of the elementary divisors of $M^*/L^*$ is $b$.  Thus we have $\det(M) = b^2\det(L)$.
\end{proof}

\subsection{Strongly squarefree lattices}
A lattice $L$ is called strongly squarefree (SSF) if the smallest generating set for $\Delta(L)$ as an $\frako$-module has at most $\frac{1}{2}\rank(L)$ elements, and every invariant factor of $\Delta(L)$ is squarefree.  The automorphism group of any non-SSF lattice is contained in the automorphism group of an SSF one.  For this reason, we restrict our classification to SSF lattices, as Nikulin did in \cite{nikulin2000classification}.  Allcock applies Vinberg's algorithm to get only the SSF lattices, and obtains the non-SSF ones from the SSF ones in other ways \cite{allcock2012reflective}. For rank 3, being SSF is equivalent to $\det(L)$ being squarefree. 

Though we do not use the $\frakp$-filling or $\frakp$-duality operations in the rest of this paper, we describe them briefly here, since they are the operations by which we would turn a non-SSF lattice into an SSF one whose automorphism group contains $\Aut(L)$. They serve the same purpose as the $p$-duality and $p$-filling operations used by both Nikulin and Allcock in their respective classifications.  The only difference is we replace prime numbers $p$ with prime ideals $\frakp$, and $p$-elementary abelian groups with $\frakp$-elementary $\frako$-modules.  

For a prime ideal $\frakp$ of $\frako$, the $\frakp$-power part of $\Delta(L)$ is the submodule annihilated by some power of the ideal $\frakp$.  The $\frakp$-dual of a lattice is the sublattice of $L^*$ corresponding to the $\frakp$-power part of $\Delta(L)$, which is a direct sum $\oplus_{i=0}^n\frako/\frakp^{a_i}$.  Since $L$ is unscaled, at least one of the summands is trivial.  To get the rescaled $\frakp$-dual of $L$ we scale $\frakp\text{-dual}(L)$ by $\frakp^a$ where $a$ is the largest of the $a_i$.  The rescaled $\frakp$-dual is an unscaled integral lattice.  Like other kinds of duals, taking the rescaled $\frakp$-dual twice gets back the original lattice, so the lattice and its rescaled $\frakp$-dual have the same automorphism group.  All the unscaled lattices related by rescaled $\frakp$-dualities for various primes $\frakp$ form a duality class, all of whose members have the same automorphism groups.  

If $\Delta(L)$ has any elements annihilated by $\frakp^2\setminus\frakp$, then we can do an operation known as $\frakp$-filling.  Let $A$ be the $\frakp$-power part of $\Delta(L)$, and suppose $\frakp^a$ is its annihilator.  The $\frakp$-filling of $L$ is if the sublattice of $L^*$ whose image in $\Delta(L)$ is $\frakp^{a-1}A$.  We have 
$$L \subsetneq \frakp\text{-fill}(L) \subseteq \frakp\text{-fill}(L) \subsetneq L^*$$
so $\frakp\text{-fill}(L)$ is integral and its discriminant $\frako$-module is strictly smaller than that of $L$ and $\Aut(L)\subseteq\Aut(\frakp\text{-fill}(L))$.  A finite number of $\frakp$-filling operations followed by a finite number of $\frakp$-duality operations turns $L$ into an SSF lattice.

\subsection{The reflection part of $\Aut(L)$, simple roots, and chains of roots}
In \cite{vinberg1972groups}, Vinberg gives a descrption of the action of the reflection part of $\Aut(L)$ on $n$-dimensional hyperbolic space.  We repeat it here because it is important for what comes later.

The reflecting plane $V_r=r^{\perp}$ of the root $r\in L$ divides $V = L\otimes F$ into two halfspaces, 
$$V_r^+ = \{v\in V: v\cdot r > 0 \} \text{ and }V_r^- = \{v\in V: v\cdot r <0\}$$
 with $r\in V_r^+$.  Let $\{r_i\}_i$ be a collection of roots of $L$.  If we have an indexed set of roots like this we abbreviate
$$V_{r_i} = V_i$$ $$V_{r_i}^{\pm} = V_{i}^{\pm}$$
Let
$$P = \bigcap_{i}V_{i}^-\subset V$$
 If every compact subset of $\mathfrak{C}^+$ intersects only finitely many $V_{i}$'s, then we will call $P$ a polygonal cone.  Any polygonal cone can be written as the intersection of halfspaces in such a way that no $V_{i}^-$ contains the intersection of all the others.  When we write it this way, we say the roots whose halfspaces we intersect to get $P$ are the roots defining $P$.  If $P\cap \mathfrak{C}^+$ is nonempty, then the image in $\Lambda^n$ of $P\cap\mathfrak{C}^+$ is nonempty and is called a polygonal cell.

The full automorphism group of $L$ contains transformations that swap $\mathfrak{C}^+$ and $\mathfrak{C}^-$.  We will not be concerned with these automorphisms, so \textbf{for us $\Aut(L)$ will the group of isometries of $L$ that preserve $\mathfrak{C}^+$}.  Let $\Gamma\subseteq \Aut(L)$ be the subgroup of $\Aut(L)$ generated by the reflections in all of the roots of $L$.  If $r\in L$ is a root, then the image in $\Lambda^n$ of $V_r\cap\mathfrak{C}^+$ is a geodesic hyperplane.  The hyperplanes in $\Lambda^n$ corresponding to the roots of $L$ carve up $\Lambda^n$ into polygonal cells that tile $\Lambda^n$.  Each of these cells is a copy of the fundamental domain for the action of $\Gamma$ on $\Lambda^n$, called the Weyl chamber of $\Gamma$.  Fix a copy $C$ of the chamber. 
The set of roots $\{r_i\}$ defining the polygonal cone whose image in $\Lambda $ is $C$ are a set of simple roots for $L$.  Their reflections generate $\Gamma$, and $\Aut(L)$ can be written as a semidirect product 
$$\Aut(L) = \Gamma\rtimes H$$
where $H$ is the group of symmetries of $C$.

Because $\Gamma$ is discrete, the angles between the codimension $1$ faces of $C$ are all of the form $\frac{\pi}{m}$ for some $m\in\N$.  We say that $L$ is reflective if $\Gamma$ is finitely generated, or equivalently, if $C$ has finitely many faces.

When $n=2$, $C$ is $2$-dimensional, and its boundary is $1$-dimensional.  We can think of each boundary component of $C$ as a chain of consecutive edges.  More generally, we will define a chain of edges in such a way that it is not necessarily part of the boundary component of a single copy of the chamber.  We start by defining a \emph{chain of roots}.

A chain of roots is a set of roots of $L$ that can be indexed by a set $I$ of consecutive integers such that the following are satisfied
\begin{enumerate}

\item For $i, j\in I$ with $i < j$, we have $r_i \cdot r_j \leq 0$, and
\begin{equation*}
\frac{r_i\cdot r_j}{\sqrt{r_i^2r_j^2}}
\left\{
\begin{array}{cc}
 \geq -1 &\text{ if } j = i+1\\
 < -1&\text{ otherwise}
\end{array}\right.
\end{equation*}

\item The intersection of the negative halfspaces 
$$\bigcap_{i\in I}V_{i}^-$$
is nonempty.

\item  The chain is ``locally simple,'' that is, for each consecutive pair $i,i+1\in I$, the roots $r_i,r_{i+1}$ are part of a system of simple roots for $L$.  This means there is a copy $C$ of the chamber where $V_i^-$ and $V_{i+1}^-$ are two of the halfspaces in the intersection defining $C$.
\end{enumerate}

A chain of roots defines a polygonal cell in $\Lambda^2$, which may have infinitely many sides but is locally finite.  The reflection group generated by reflections in the roots in the chain is a subgroup $\Gamma'$ of $\Gamma$.  Each edge of the polygonal cell is contained in a hyperplane orthogonal to a root in the chain of roots.  We call the sequence of edges of the polygonal cell for a chain of roots a \emph{chain of edges}.  Each pair of consecutive roots forms a corner of a copy of the chamber for $\Gamma$ that is contained in the chamber for $\Gamma'$.

A closed chain of roots is the same as a chain of roots, except it is cyclically indexed.  For example, if $L$ is a reflective lattice, and $C$ is a copy of the chamber for $\Gamma$, a system of simple roots for $L$ are a closed chain of roots.

Let $\Phi = \{r_i\}_{i\in I}$ be a non-closed chain of roots.  If $I$ is bounded above with $k = \max I$, we call $r_{k}$ the highest root in the chain.  If there exists a root $r_{k+1}$ such that $\Phi\cup\{r_{k+1}\}$ is a chain of roots, then we say that $\Phi$ can be extended.  In particular, if $\Phi$ can be extended then there is a root $r_{k+1}$ such that $r_{k-1},r_k,\text{ and }r_{k+1}$ all bound a single copy of the chamber.  We call this root the \emph{next root} of $\Phi$.  If $I$ is bounded below, then $\Phi$ may have a \emph{previous root}, which we define similarly.  In all of our applications, every bounded chain of roots has a next root and a previous root.

\subsection{Reflective hulls and enlargements}\label{rrer}
Let $L$ be a lattice generated by roots.  The reflective hull, denoted by $L^{rh}$, of $L$ is 

$$L^{rh} = \left\{v\in L\otimes F: r\cdot v\in \frac{r^2}{2}\frako\text{ for all roots }r\in L\right\}$$

An $\frako$-lattice $L$ sits inside its reflective hull $L \subset L^{rh}$.  A lattice $M$ with 
$$L\subset M\subset L^{rh}$$
 is called a reflection-stable enlargement of $L$.  
Reflection-stable enlargements of $L$ correspond to submodules of $L^{rh}/L$.  If we take the invariant factor decomposition
$$L^{rh}/L = \bigoplus_{i=1}^n\frako/a_i\frako$$
with $(a_n)\supset \ldots \supset (a_1)$, that means there is a basis $x_1,\ldots x_n$ for $L^{rh}$ such that $L = \lz{a_1x_1,\ldots, a_nx_n}$.  We may find all reflection-stable enlargements of $L$ by iterating over all matrices of the form
\begin{equation}\label{reflectiveEnlargement}
\left(\begin{array}{cccc}
d_1 &  \\
b_{1,2} & d_2  \\
\vdots & & \ddots\\
b_{1,n} & b_{2,n} & \ldots & d_n
\end{array}\right)
\end{equation}
where $(d_n)\supset (a_n)$, and $b_{i,n}$ iterates over the set of all distinct coset representatives of $\frako/d_i\frako$.  For each matrix \eqref{reflectiveEnlargement}, the vectors
$$d_ix_i+\sum_{j = 1}^{i-1}b_{i,j}x_j$$
$i = 1,\ldots, n$ generate a reflection-stable enlargement of $L$.

\subsection{Root norms}
We wish to list all the norms that roots of a SSF lattice $L$ may have up to equivalence in the sense of (N1) or (N2).  Let $r$ be a root of $L$, $M = \lz{r}\oplus r^{\perp}$, and $\pi_r$, $\pi_{r^{\perp}}$ projection onto the $F$-spans of $r$ and $r^{\perp}$ respectively.  Suppose that $g$ is a glue vector between $\lz{r}$ and $r^{\perp}$.  Since $r$ is a root and $L$ is integral, we have 
$$g\cdot r = \pi_r(g)\cdot r \in \frac{r^2}{2}\frako\cap \frako$$
which implies that
$$\pi_r(g)\in \lz{r}^{rh}\cap\lz{r}^*$$
Likewise, the fact that $L$ is integral means that for all $v\in r^{\perp}$,
$$g\cdot v = \pi_{r^{\perp}}(g)\cdot v \in \frako$$
and so we have
$$\pi_{r^{\perp}}(g)\in (r^{\perp})^*$$
The reflective hull of $\lz{r}$ is generated by $\frac{r}{2}$.  Thus 
$$\lz{r}^{rh}/\lz{r}\cong \frako/2\frako$$
By primitivity of $r$ and saturatedness of $r^{\perp}$, the projections $\pi_r,\pi_{r^{\perp}}$ descend to injective $\frako$-module homomorphisms
$$\overline{\pi}_r:L/M \rightarrow (\lz{r}^{rh}\cap \lz{r}^*)/\lz{r}$$
and
$$\overline{\pi}_{r^{\perp}}:L/M \rightarrow (r^{\perp})^*/r^{\perp} = \Delta(r^{\perp})$$
Thus we have established that $L/M$ is isomorphic to an $\frako$-submodule of $\frako/2\frako$, and also to a submodule of $\Delta(r^{\perp})$.  Since $\frako/2\frako$ has 3 submodules, there are $3$ cases.
\begin{enumerate}
\item If $L/M$ is trivial, then 
$$\det L = \det M = r^2\det(r^{\perp})$$
In this case $r^2$ divides $\det L$, and so it certainly also divides $2\det L$.

\item If $L/M\cong \frako/\sqrt{2}\frako$, then
$$2\det L = \det M = r^2\det(r^{\perp})$$
In this case $r^2$ divides $2\det L$.

\item If $L/M\cong \frako/2\frako$, then 
$$4\det L = \det M = r^2\det(r^{\perp})$$
Since $L/M$ injects into $\Delta(r^{\perp})$, $2$ divides $\det M$.  Thus $r^2$ divides $2\det(L)$.

\end{enumerate}

We see that in each case, $r^2$ divides $2\det L$.  To list all the possible norms of roots of $L$, we therefore list all of the positive divisors of $2\det L$ up to whichever equivalence (N1) or (N2) applies in $F$.  By arithmeticity, if $n$ is the norm of a root then $N_{F/\Q}(n)$ must be positive.  Thus when (N1) holds we throw out any norms that do not have $N_{F/\Q}(n) > 0$, and when (N2) holds we replace all $n$ with $N_{F/\Q}(n) < 0$ by $\alpha_0 n$.

\subsection{Vinberg's algorithm, Bugaenko's modification}
Vinberg's algorithm, first introduced in \cite{vinberg1972groups}, is a way of listing all simple roots of the fundamental domain of a hyperbolic reflection group starting from a known corner.  Fix a corner of the chamber in $\Lambda^n$, and let $p\in \mathfrak{C}^+$ be a primitive lattice vector pointing along the corner.  Let $r_1,\hdots r_n$ be simple roots for the stabilizer of $p$.  All the faces of the chamber not passing through $p$ correspond to roots having non-positive inner product with $p, r_1,\hdots r_n$.  

New edges of the chamber can only be at certain distances from the corner.  The list of these distances is discrete.  The algorithm iterates over this list in increasing order, and at each possible distance checks whether there is a root.  If the chamber has finitely many sides, then eventually the algorithm will find a system of simple roots.

If the chamber does not have finitely many sides, then there is always a finite stage of the algorithm at which it is possible to prove that it has infinitely many sides.  Vinberg has several methods for doing this, including Proposition 1 in \cite{vinberg1972groups} and Proposition 4.1 in \cite{vinberg1985hyperbolic}.  In \cite{vinberg1978groups}, Vinberg speeds up the process of finding roots by using symmetries of the polygon.  We use that idea here as Allcock did in \cite{allcock2012reflective}.  If the chamber has infinitely many sides, then it will have a symmetry of infinite order which can be found at some finite stage of the algorithm.

Vinberg's original proof of the algorithm uses the fact that $\Z$ is discrete in $\R$ to show that the list of distances from $p$ is discrete. In a field whose ring of integers is not discrete in the order topology on $\R$, there is some additional work required to get a discrete list of distances.  Bugaenko's innovation, which he uses in \cite{bugaenko1984groups}, \cite{bugaenko1990reflective}, and \cite{bugaenko1992arithmetic}, was to notice that arithmeticity gives a way of restricting the inner products that roots can have with $p$ to a discrete ordered set.  The Galois conjugate of the inner product of two vectors is bounded.  If we think of the elements of $F$ as living in the plane, with $a+b\sqrt{d}$ corresponding to the point $(a,b)\in\Q^2$, this bound describes an upward sloping strip in the plane.  The elements $a+b\sqrt{d}$ of $\frako$ such that the point $(a,b)$ is inside this region are a discrete subset of $\R$ under the identity embedding $\frako\rightarrow\R$, and they inherit an ordering from $\R$.

\subsection{Chamber angles over $\Q(\sqrt{2})$}

A version of the following lemma is true for any number field.  For example, the angles allowed for lattices defined over $\Q$ are $\frac{\pi}{n}$ with $n = 2,3,4,\text{ or }6$.  This is the $\Q(\sqrt{2})$ version.

\begin{lemma}
\label{angles-and-norms}
Suppose $r$ and $s$ are consecutive simple roots with $R$ and $S$ the corresponding sides of the fundamental polygon.  The angle between $R$ and $S$ is $\frac{\pi}{m}$ for $m=2,3,4,6,$ or $8$ and up to choosing a scale for $L$, and replacing $r$ or $s$ with equivalent roots in the sense of (N2), we have that
\begin{enumerate}
\item if $m = 3$, $r^2 = s^2$.
\item if $m = 4$, either $r^2 = s^2$, $r^2 = \frac{s^2}{2}$, or $r^2 = 2s^2$.
\item if $m = 6$, either $r^2 = \frac{s^2}{3}$, or $r^2 = 3s^2$.
\item if $m = 8$, either $r^2 = (2+\sqrt{2})s^2$, or $r^2 = \frac{s^2}{2+\sqrt{2}}$
\end{enumerate}
\end{lemma}

\begin{proof} We have
$$\cos\pz{\frac{\pi}{m}} = -\frac{r\cdot s}{\sqrt{r^2s^2}}$$
The square of the righthand side clearly lives in $\Q(\sqrt{2})$.  The only $m\in\N$ for which $\cos\left(\frac{\pi}{m}\right)$ lives in a degree 2 extension of $\Q(\sqrt{2})$ are $m=2,3,4,5,6,8$, so the angle between $R$ and $S$ is $\frac{\pi}{m}$ for one of these $m$.  The $m=5$ case does not occur, because $\cos^2\left(\frac{\pi}{5}\right)$ is not an element of $\Q(\sqrt{2})$ 

Since $r$ and $s$ are roots, we can write $r\cdot s$ in two ways
$$r\cdot s = -\frac{r^2\alpha}{2}\text{ ~ and ~ } r\cdot s= -\frac{s^2\beta}{2}$$
where $\alpha,\beta\in\Z[\sqrt{2}]$ and $\alpha,\beta >0$.  We also have that
$$\cos\left(\frac{\pi}{m}\right) = -\frac{r\cdot s}{\sqrt{r^2 s^2}}$$
so
$$\cos\left(\frac{\pi}{m}\right) = \frac{\alpha}{2}\sqrt{\frac{r^2}{s^2}} = \frac{\beta}{2}\sqrt{\frac{s^2}{r^2}}$$

Now, $\cos\left(\frac{\pi}{m}\right) = \frac{\sqrt{\gamma(m)}}{2}$, where $\gamma(3) = 1$, $\gamma(4) = 2$, $\gamma(6) = 3$, $\gamma(8) = 2+\sqrt{2}$ so 
$$\gamma(m) = 4\left(\frac{\sqrt{\gamma(m)}}{2}\right)^2 = 4\left(\frac{\alpha}{2}\sqrt{\frac{r^2}{s^2}}\right)\left(\frac{\beta}{2}\sqrt{\frac{s^2}{r^2}}\right) = \alpha\beta$$
The cases listed in the statement of the lemma come from the factorizations of the various values of $\gamma(m)$ as products of positive elements of $\frako$ with positive field norm $N_{F/\Q}$.

\end{proof}

\subsection{$A_2$ corners}

Let $L$ be a lattice of signature $(2,1)$, $\Gamma \leq\Aut(L)$ the reflection subgroup of the automorphism group of $L$, $C$ a fixed Weyl chamber for the action of $\Gamma$ on $\Lambda^2$, and $c$ a corner of $C$.  Fix an orientation on $V = L\otimes F$.  There are roots $r,s$ pointing outward from walls of $C$ that meet at $c$, and a primitive lattice vector $p\in V_r\cap V_s\cap\mathfrak{C}^+$ such that the basis $\{r,s,p\}$ is in the orientation on $V$.  These conditions uniquely determine $r,s,\text{ and }p$ up to multiplication by positive units.  We call this the corner basis at $c$, and we say that $p$ lies along the corner $c$.

The roots $r$ and $s$ at the corner are simple roots for a rank 2 positive definite sublattice of $L$.  The type of $c$ is the type of that positive definite sublattce.  
By Lemma \ref{angles-and-norms}, $\lz{r,s}$ has type $A_1^2$, $A_2$, $B_2$, $G_2$, or $I_2(8)$.

The next several lemmas are about the structure of $L$ at an $A_2$ corner.  By Lemma \ref{angles-and-norms}, we know that $r^2 = s^2u^2$ where $u > 0$ is a unit.  Since (U2) and (N2) hold when $F =\Q(\sqrt{2})$, we may replace $s$ by $su^{-1}$, and assume that $r^2 = s^2$.

\begin{lemma}\label{nmp}
Let $c$ be an $A_2$ corner of the chamber $C$ with $\{r,s,p\}$ a corner basis at $c$ with $r^2 = s^2$.  Then we have
$$L/\lz{r,s,p} \cong \frako/3\frako$$
Let bar denote passage to the quotient.  If $\overline{g}$ is a generator for $L/\lz{r,s,p}$ then 
$$L = \lz{r,s}\oplus_g\lz{p}$$
If $L$ is strongly squarefree, then up to multiplication by a square unit, $r^2 = s^2 = 2$, and $p^2 = 3\det(L)$.
\end{lemma}
Recall that both $p^2$ and $\det(L)$ are both defined up to multiplication by square units.  We will see that fixing a choice of square unit multiple for one of them also fixes a choice for the other.
\begin{proof}
Let $M = \lz{r,s}$.  Then $M^{rh}$, the largest superlattice of $M$ in which $r$ and $s$ are roots, has type $G_2$ and is generated over $\frako$ by $M$ and $a = \frac{2r+s}{3}$.  As an $\frako$-module, the quotient
$$M^{rh}/M \cong \frako/3\frako \cong \mathbb{F}_9$$
is isomorphic to a finite field with $9$ elements, $\mathbb{F}_9$.

Let $\pi$ be orthogonal projection onto $M\otimes F$.  Since $r$ and $s$ are roots of $L$, we have $M\subseteq\pi(L)\subseteq M^{rh}$.   If there were no glue between $M$ and $\lz{p}$, the reflection that negates $a$ and preserves $a^{\perp}$ would preserve $L$, and there would be a root of $L$ in the $F$-span of $a$.  But $r$ and $s$ are simple roots of $L$, so this is impossible.  Therefore there must be a nontrivial glue vector $g$ gluing $M$ to $\lz{p}$.  Since any nontrivial element of $\mathbb{F}_9$ generates it as an $\frako$-module, and $\pi(L) = M^{rh}$, we may assume that $\pi(g) = a$.

We have $g\in L$ and $3\pi(g)\in L$, thus $3(g-\pi(g))\in L\cap(\lz{p}\otimes F) = \lz{p}$.  Since $\pi(g)\notin L$, we have $g-\pi(g)\notin\lz{p}$.  Thus
$$g-\pi(g) = \frac{\alpha p}{3}$$
for some $\alpha\in\frako$ representing a nontrivial coset of $\frako/3\frako$.  By subtracting $\frako$-multiples of $p$, we may assume that $0 < \alpha < 3$.

We claim that $L = \lz{r,s,g}$.  To see that this is true, let $h\in L$.  If $\pi(h)\in M$ then $h-\pi(h)\in L\cap\lz{p}$, so  $h\in \lz{r,s,p}\subset \lz{r,s,g}$.  Suppose $\pi(h)\in M^{rh}\setminus M$.  Then $\pi(g)$ generates all of $M^{rh}/M$, as an $\frako$-module, so there is some $z\in\frako$ such that $\pi(zg) = z\pi(g) = \pi(h)$.  Thus $zg-h\in\lz{p}$, so $h\in \lz{r,s,g}$.  Thus $L = \lz{r,s,g}$ and 
$$L/\lz{r,s,p}\cong\frako/3\frako$$
is a finite $\frako$-module generated by $\overline{g}$. 

We compute the norm of $g$:
\begin{align*}
g^2 &= \pz{\frac{2r+s}{3}+\frac{\alpha p}{3}}^2\\
&= \frac{4r^2+4r\cdot s+s^2}{9} +\frac{\alpha^2 p^2}{9}\\
&= \frac{r^2}{3} + \frac{\alpha^2 p^2}{9}
\end{align*}

Recall from the proof of Lemma \ref{angles-and-norms} that 
$$-r\cdot s = -\frac{r^2\beta}{2} = -\frac{s^2\delta}{2}$$
where 
$$\beta\delta = \gamma(3) = 1$$
Thus $r^2$ is divsible by $2$.

Since $g, p\in L$, we have $g\cdot p\in \frako$.
$$g\cdot p = \frac{\alpha p^2}{3}$$
Since $\alpha$ represents a nontrivial coset of $\frako/3\frako$ and $3$ is prime in $\frako$, this implies that $p^2$ is divisible by $3$.

Since $L = \lz{r,s,g}$, their inner product matrix has determinant $\det(L)$.
\begin{align*}
\det(L) &= \left|\begin{array}{ccc}
r^2 & -\frac{r^2}{2} & \frac{r^2}{2}\\
-\frac{r^2}{2} & r^2 & 0\\
\frac{r^2}{2} & 0 & g^2
\end{array}\right|  \\
& = \pz{\frac{r^2}{2}}^2 \pz{3g^2-r^2}\\
& = \pz{\frac{r^2}{2}}^2 \pz{ 3\pz{\frac{r^2}{3} +\frac{\alpha^2 p^2}{9}} -r^2}\\
& = \pz{\frac{r^2}{2}}^2 \pz{ \frac{\alpha^2 p^2}{3}}
\end{align*}

Since $2$ divides $r^2$ and $3$ divides $p^2$, both $\frac{r^2}{2}$ and $\frac{\alpha p^2}{3}$ are elements of $\frako$.  Because $L$ is SSF, $\det(L)$ is squarefree.  Thus the squared factor
$$\pz{\frac{\alpha r^2}{2}}^2$$
must be a unit.  This means we must have 
$$\alpha r^2 = 2u$$
where $u\in U(F)$ is a unit.  Because we know that $2$ divides $r^2$ and $r^2,\overline{r}^2 >0$, we can conclude that up to scaling $r$ by a positive unit, $\alpha\in U(F)$ and $r^2 = 2$.  Here we are using the fact that (U2) holds in $\Q(\sqrt{2})$ to say that $r^2=2$ and not $2$ times a non-square unit.

Finally, by Lemma \ref{subDeterm}, we have
$$9\det{L} = \det(M\oplus \lz{p}) = \det(M)\det\lz{p} = 3p^2$$
so 
$$3\det(L) = p^2$$

\end{proof}

For the rest of this 
subsection
we will be working with an $A_2$ corner in an SSF lattice with corner basis $\{r,s,p\}$ where $r^2 = s^2 = 2$.

\begin{lemma}\label{A2Corners}
Let $c$ be an $A_2$ corner of $C$, and let 
\begin{equation}\label{gPlusMinus}
g_{\pm} = \frac{r+s}{2} \pm \frac{r-s}{6} + \frac{p}{3}
\end{equation}
Exactly one of $g_{+}$ or $g_{-}$ is an element of $L$, and 
$$L = \lz{r,s}\oplus_{g_{\pm}}\lz{p}$$
\end{lemma}
\begin{proof}
We have that $g_+$ and $g_-$ are not both elements of $L$, since 
$$g_++g_- = \frac{r+s}{2}\notin L$$

On the other hand, one of $g_{\pm}$ must be in $L$, by the following argument.  As we saw in the proof of Lemma \ref{nmp}, there is a glue vector $g$ such that
$$L = \lz{r,s}\oplus_g\lz{p}$$
with $\pi(g) = \frac{2r+s}{3} =: a$. 
Also recall from the proof of Lemma \ref{nmp} that
$$g = \pi(g) +\frac{\alpha p}{3} = a+\frac{\alpha p}{3}$$
where $0 < \alpha < 3$ represents some coset of $\frako/3\frako$.  Since $L = \lz{r,s,g}$, the inner product matrix of the generators has determinant $\det(L)$.
\begin{align*}
\det(L) &= \left|\begin{array}{ccc}
2 & -1 & 1\\
-1 & 2 & 0\\
1 & 0 & g^2
\end{array}\right|  \\
& = 3g^2-2\\
& = 3\left(a^2 +\alpha^2\frac{p^2}{9}\right) - 2\\
& = 3\left(\frac{2}{3} +\alpha^2\frac{p^2}{9}\right) - 2\\
& = -\alpha^2\frac{p^2}{3}
\end{align*}

Thus $\alpha = \pm 1$.  If $\alpha = 1$, then $g=g_+$ and 
\begin{equation}\label{plusGlue}
L = \lz{r,s}\oplus_{g_+}\lz{p}
\end{equation}

If $\alpha = -1$, then 
$$g_- = 2g + p - r\in L$$
Since $g_-$ projects to an element of $\lz{r,s}^{rh}\setminus \lz{r,s}$, it is a glue vector between $\lz{r,s}$ and $\lz{p}$, and
\begin{equation}\label{minusGlue}
L = \lz{r,s}\oplus_{g_-}\lz{p}
\end{equation}

\end{proof}

If $L$ has the form \eqref{plusGlue} (resp. \eqref{minusGlue}) at the $A_2$ corner $c$ of the chamber $C$, then we say that $C$ has positive (resp. negative) glue at $c$, or that the pair $(c,C)$ has positive (resp. negative) glue. Note that this depends on our fixed choice of $\mathfrak{C}^+$ and orientation on $V = L\otimes F$.

Let $\phi\in\Aut(L)$ be an automorphism of $L$ that preserves the future cone $\mathfrak{C}^+$.  Let $C$ be a copy of the chamber, $c$ a corner of $C$.  Then $\phi$ takes $C$ to another copy of the chamber $C'$ and $c$ to a corner $c'$ of $C'$.  Let $\{r,s,p\}$ and $\{r's',p'\}$ be corner bases at the corners $c$ and $c'$ respectively.  Then $\phi(p) = p'$ and $\phi(\{r,s\}) = \{r',s'\}$.  As we will see in this next lemma, $\phi$ may or may not preserve the orientation on $V$.

Recall that $\Aut(L)$ is the group of automorphisms of $L$ that preserve the future cone $\mathfrak{C}^+$.

\begin{lemma} \label{A2AutBoth}
Suppose $C$ and $C'$ are copies of the chamber for $L$, and $c$ and $c'$ are $A_2$ corners of $C$ and $C'$ respectively.  Then there is an automorphism $\phi\in\Aut(L)$ that takes the pair $(c,C)$ to $(c',C')$.  If $(c,C)$ and $(c',C')$ both have positive glue or both have negative glue, then $\phi$ preserves the orientation on $V = L\otimes F$.  If they have opposite glue, then $\phi$ reverses the orientation on $V$.
\end{lemma}
\begin{proof}
First suppose that $c$ and $c'$ both have positive glue.  Then the linear transformation defined by
$$\phi:(r,s,p)\mapsto (r',s',p')$$
preserves the gluing, meaning $\phi(g_+) = g'_+$ where $g'_+$ is defined as $g_+$ but with all non-primed things in \eqref{gPlusMinus} primed.  Since $L = \lz{r,s,g_+} = \lz{r',s',g'}$, $\phi$ is an automorphism of $L$.  The argument is the same if $c$ and $c'$ both have negative glue.

Now suppose that $c$ has positive glue and $c'$ has negative glue.  Let $\psi$ be the linear transformation defined by 
$$\psi:(r,s,p)\mapsto (s',r',p')$$
Then $\psi(g_+) = g'_-$, so $\psi\in\Aut(L)$.  The ordered basis $\{s',r',p'\}$ has the opposite orientation from $V$, so $\psi$ is orientation reversing.
\end{proof}

\begin{lemma}\label{A2Aut}
Let $C$ be a copy of the chamber for $L$.  If $C$ has two $A_2$ corners $c$ and $c'$ in the same boundary component of $C$, then either they both have positive glue, or both have negative glue.
\end{lemma}

\begin{proof}
By Lemma \ref{A2AutBoth}, there exists $\phi\in \Aut(L)$ that preserves the chamber $C$ and takes $c$ to $c'$.  Because $c$ and $c'$ are corners of the same boundary component of the chamber $C$, there is a chain of roots $\{r_1,\ldots,r_k\}$ with $r = r_1,s = r_2,r'=r_{k-1},\text{ and } s' = r_k$ such that for each $i$ with $1 \leq i \leq k-1$, $r_{i+1}$ is the next root after $r_i$.  

Suppose that $c$ and $c'$ have opposite glue.  Then $\phi$ is orientation reversing, and
$$\phi(r_1) = r_k, \phi(r_2) = r_{k-1}$$
Because each root $r_{i+1}$ is the next root after $r_i$, the automorphism $\phi$ takes the chain to itself, with $\phi(r_i) = r_{k+1-i}$.  Thus it is a reflection and not a glide reflection.  But then there must be a root $t$ such that $R_t = \phi$, and the reflecting plane $V_t$ cuts through the interior of $C$.  This is impossible since $C$ is a chamber.  Therefore $c$ and $c'$ cannot have opposite glue.

\end{proof}

\begin{figure}[ht]

\input figure3-1.tex
\input figure3-2.tex

On the left is a chamber with two $A_2$ corners with the same glue.  On the right, we see that if a chamber had 2 $A_2$ corners with opposite glue, a reflecting plane would cut through the chamber, making it not a chamber and giving a contradiction.  The $\bullet$'s show the glue vector and its orbit under order 3 rotation, projected into the plane spanned by the roots.  The $\times$'s show where the projection of the opposite glue vector and its orbit would be if it were present.
\caption{If a chamber has 2 $A_2$ corners, they must have the same glue.}
\label{A2-2ways}
\end{figure}

\begin{corollary}\label{allA2}
If a chamber $C$ for $\Gamma$ has consecutive $A_2$ corners, then all of the corners in the boundary component containing those two are $A_2$ corners.  
\end{corollary}

\begin{proof}
If $C$ has two consecutive $A_2$ corners $c_1$ and $c_2$, then by Lemma \ref{A2AutBoth} there is an automorphism $\phi$ that takes a corner basis at $c_1$ to a corner basis at $c_2$.  By Lemma \ref{A2Aut}, $\phi$ is orientation preserving.  Since $\phi$ preserves $C$, and $c_1$ and $c_2$ are in the same boundary component of $C$, $\phi$ also preserves the chain of roots for the boundary component of $C$ containing both $c_1$ and $c_2$.  Therefore $\phi$ preserves adjacency of roots in that chain.  By applying all integer powers of $\phi$ to the chain, we see that all corners in that boundary component are $A_2$ corners.

\end{proof}

\section{Getting a finite list}\label{finite}

The first step in the proof of Theorem \ref{main} is to generate a finite list of matrices, each the inner product matrix for roots in a chain of length 3, 4, or 5.  The construction of this list ensures that any reflective lattice contains some such chain. It will also contain a lot of non-reflective lattices, and some of the lattices listed will be redundant.  Those will all be sorted away in the second step.

\subsection{Hyperbolic polygons are thin}
We take the following definitions from Allcock, whose short edges in \cite{allcock2012reflective} were a version of Nikulin's thin parts of polygons in \cite{nikulin2000classification}.  Our discussion will now involve hyperbolic polygons, since the chamber of a reflective lattice has finitely many sides.

Let $P$ be a hyperbolic polygon.  At any corner of $P$, the angle bisector means the ray originating at the vertex that passes through the interior of $P$ and bisects the angle at the corner\footnote{Allcock gives a definition of angle bisector at an ideal vertex so that this definition extends to hyperbolic polygons with ideal vertices.  We do not need that here, since none of our polygons have any ideal vertices.}.  Similarly the perpendicular bisector of an edge of $P$ means a ray that passes through $P$ that originates at the midpoint of the edge and is orthogonal to it. 

An edge of $P$ is called a \emph{short edge} if the angle bisectors at its endpoints intersect.  Theorem 1 in \cite{allcock2012reflective} says that any finite sided polygon has a short edge in one of three configurations.  We restate the three possibilities, as we will refer to them often.
\begin{enumerate}
\item $P$ has a short edge orthogonal to at most one of its neighbors such that those neighbors are not orthogonal to each other.
\item $P$ has at least 5 edges, and \emph{short pair} $(S,T)$.  Here $S$ is a short edge orthogonal to both its neighbors, $T$ is one of these neighbors, and the perpendicular bisector of $S$ intersects the angle bisector emanating from the opposite end of $T$.
\item $P$ has at least 6 edges, and a \emph{close pair} of short edges $(S,S')$.  Here $S$ and $S'$ are both short edges with a common neighbor that is not a short edge, both $S$ and $S'$ are orthogonal to their neighbors, and their perpendicular bisectors intersect.
\end{enumerate}

We can be even more precise about how thin these thin parts are. There are bounds on the distances between the edges adjacent to a short edge, short pair, or close pair.  These bounds are given in Lemmas 3-5 in \cite{allcock2012reflective}.  Together with the following adaptation of Lemma 6 in \cite{allcock2012reflective}, these bounds are what goes into generating our list.

\begin{lemma}[Pair of Roots]\label{pairOfRoots}
Suppose $a,b\in L$ are roots whose associated unit vectors satisfy $-\hat{a}\cdot\hat{b} = k > 0$.  Then the inner product matrix of $a$ and $b$ is an $F$ multiple of 
\begin{equation}\label{pairOfRoots2x2Matrix}
\pz{\begin{array}{cc}
2u & -uv \\ -uv & 2v
\end{array}}
\end{equation}
for some positive $u,v\in\frako$ with $uv = 4k^2$.  Furthermore, if the quadratic form $q$ on $L$ is arithmetic, then for any non-identity element $\sigma\in\Gal\pz{F/\Q}$ we have $\sigma(u),\sigma(v) >0$ and $\sigma(uv) < 4$.
\end{lemma}
\begin{proof}
Because $a$ and $b$ are roots with $a\cdot b < 0$, we can write
\begin{equation}\label{defuandv}
a\cdot b = -\frac{va^2}{2}\text{ ~ and ~ } a\cdot b=-\frac{ub^2}{2}
\end{equation}
for some positive $u,v\in\frako$.  Thus we have
$$-a\cdot b = \sqrt{\pz{-a\cdot b}^2} = \sqrt{\pz{\frac{va^2}{2}}\pz{\frac{ub^2}{2}}} = \frac{\sqrt{a^2}\sqrt{b^2}}{2}\sqrt{uv}$$
Using the fact that 
$$-\hat{a}\cdot\hat{b} = k$$
we get that
$$-a\cdot b = \sqrt{a^2}\sqrt{b^2}k$$
Thus
$$uv = 4k^2$$
We then scale the inner product so that $a\cdot b = -uv$, and then by \eqref{defuandv} we have $a^2 = 2u$ and $b^2 = 2v$.

The matrix \eqref{pairOfRoots2x2Matrix} is the matrix for $q$ restricted to the $F$-span of $a$ and $b$.  If $q$ is arithmetic then for any non-identity element $\sigma\in\Gal\pz{F/\Q}$, the quadratic form $\sigma{q}$ is positive-definite on $\sigma FL$.  Then for any such $\sigma$, we have $\sigma(a^2),\sigma(b^2) >0$, which gives us 
\begin{equation}\label{posuv}
\sigma(u) > 0\text{ ~ and ~ }\sigma(v) > 0
\end{equation}
We also have that
\begin{equation}\label{normbd}
-1\leq -\sigma q\pz{\widehat{\sigma(a)},\widehat{\sigma(b)}} \leq 1
\end{equation}
The unit vectors $\hat{a}$ and $\hat{b}$ do not live in $L$, but they do live in the lattice $EL$ where $E$ is the splitting field over $F$ of $(x^2-a^2)(x^2-b^2)\in F[x]$.  Thus $\sigma$ extends to an automorphism of $E$.  If we fix a lift of $\sigma$ to $\Gal\pz{E/\Q}$, the following makes sense:
$$\sigma\pz{-\hat{a}\cdot \hat{b}} = \sigma\pz{\frac{-a\cdot b}{\sqrt{a^2}\sqrt{b^2}}} = \frac{\sigma(-a\cdot b)}{\sqrt{\sigma(a^2)}\sqrt{\sigma(b^2)}} = \frac{\sigma(uv)}{\sqrt{4\sigma(uv)}} = \frac{\sqrt{\sigma(uv)}}{2}$$
combining this with \eqref{posuv} and \eqref{normbd} we obtain
$$0 < \sigma(uv) < 4$$

\end{proof}

\subsection{Good factorizations of elements of $\frako$}\label{unitEquiv} 
We use Lemma \ref{pairOfRoots} repeatedly in the proofs of Lemmas \ref{shortEdge}-\ref{closePair}, which play the same role in our classification that Allcock's lemmas 7-9 do in his.  There are two ways in which our lemmas are different from Allcock's.  First, our version of Lemma \ref{pairOfRoots} involves a bound on the conjugate of the inner product of any pair of roots, and so our versions of Lemmas \ref{shortEdge}-\ref{closePair} do as well.  Second, if there are to be finitely many factorizations of an integer $z\in\frako$ as a product of two integers, then factorizations need to be defined up to some sort of equivalence in the sense of (N1) or (N2).  We make that precise now before we state the lemmas.  The following discussion applies to any $F = \Q(\sqrt{d})$ with $\frako_F$ a PID, not $\Q(\sqrt{2})$ exclusively.

Given a positive number $z\in\frako_F$, we wish to recover all inner product matrices for pairs of roots $a,b$ with $-a\cdot b = z$.  We will call a factorization $z = uv$ a good factorization if $u,v,\sigma(u),\sigma(v) >0$ and $\sigma(uv) < 4$ for all non-identity $\sigma\in\Gal(F/\Q)$.  If $z = uv$ is a good factorization and $\alpha\in U(F)$ is a unit, then $z = (u\alpha)(\alpha^{-1}v)$ is also a good factorization of $z$ provided $\alpha >0$ and $\sigma(\alpha) > 0$ for all $\sigma\in\Gal(F/\Q)$.  The group $U_1^+$ of positive units with norm $1$ acts on the set of good factorizations of $z$, and thus on the set of configurations of pairs of roots $a,b$ with $-a\cdot b = z$ by
\begin{equation}\label{pairOfRootsUnit}
\alpha.
\pz{\begin{array}{cc}
2u & -uv \\ -uv & 2v
\end{array}}
=
\pz{\begin{array}{cc}
2u\alpha & -uv \\ -uv & 2v\alpha^{-1}
\end{array}}
\end{equation}
We say that two good factorizations $(u,v)$ and $(u\alpha,v\alpha^{-1})$ are equivalent if $\alpha\in U_1^+$ is a square.  In other words, the equivalence classes are the orbits of the subgroup of $U_1^+$ consisting of square units.  If (U1) holds in $F$, then there are two equivalence classes of good factorizations for $z$, and if (U2) holds there is one equivalence class.

\subsection{The root configuration lemmas}
We follow the notation used in \cite{allcock2012reflective}.  $P$ is the fundamental domain in $\Lambda^2$ for an arithmetic reflection group acting discretely by isometries on $\R^{2,1}$.  We let $r,(s),t,(s'),r'$ be roots in a lattice $L$ corresponding to consecutive faces of $P$, called $R,(S),T,(S'),R'$ respectively.  When we have no $S$ and $S'$ the edge $T$ is short.  When we have $S$ but no $S'$, $(S,T)$ is a short pair.  When we have both $S$ and $S'$ then $(S,S')$ is a close pair.  We define $\mu:=-\hat{r}\cdot\hat{t}$, $\mu':=-\hat{r}'\cdot\hat{t}$, $\lambda = -\hat{r}\cdot\hat{r}'$, and $K = 1+\mu+\mu'+2\sqrt{1+\mu}\sqrt{1+\mu'}$.

The numerical inequalities given by Lemmas 3-5 in \cite{allcock2012reflective} all hold.  In fact, we can do slightly better for a trivial reason.  The arithmeticity condition means that lattices defined over nontrivial extensions of $\Q$ have no isotropic vectors, so the bounds on the inner products between unit normal vectors of adjacent sides are all strict.  The bounds are as follows.

\emph{Short edge:} $0 < \mu < 1$, $0\leq \mu' < 1$, and $\lambda < K < 7$ 

\emph{Short pair:} $1 < \mu < 3$, $0\leq \mu' < 1$, and $\lambda < K < 5+4\sqrt{2}$

\emph{Close pair:} $1 < \mu, \mu' < 3$, and $\lambda < K < 15$

We are working in a real quadratic extension $F = \Q(\sqrt{d})$ of $\Q$.  Bar denotes Galois conjugation $\overline{\cdot}:\sqrt{d}\mapsto -\sqrt{d}$.  

\begin{lemma}[short edge]\label{shortEdge}
Suppose $P$ has consecutive edges $R,T,R'$. Suppose also that $T$ is a short edge and $R\not\perp T,R'$.  Then, up to scale, there are finitely many possibilities for the inner product matrix of $r,t,r'$.

To list those possibilities, we let $(A,B)$ vary over a set of representatives for the equivalence classes of all good factorizations of elements $z$ of $\frako$ with $0 < z,\overline{z} < 4$. Let $(A',B')$ vary over all those same pairs and also $(0,0)$.

Given such an $A,B,A',B'$, let $(C,C')$ vary over a set of representatives for the equivalence classes of all good factorizations of all $z\in\frako$ with $0 < z < 4K^2$, $0 < \overline{z}< 4$.  If $A',B' >0$, we only consider pairs $(C,C')$ satisfying
\begin{equation}\label{betaCompatability}
\frac{AB'C'}{A'BC} = u^2
\end{equation}
where $u\in U^+_1$ is a square unit.  Let $\beta = A'BCu$.  For some such $A,B,A',B',C,C'$, the the inner product matrix of $r,t,r'$ is a $F$ multiple of
\begin{equation}\label{senoMatrix}
\pz{\begin{array}{ccc}
2AB' & -ABB' & -\beta\\
-ABB' & 2BB' & -A'B'B\\
-\beta & -A'B'B & 2A'B
\end{array}}
\text{ ~ if } A',B' > 0
\end{equation}
\begin{equation}\label{seoMatrix}
\pz{\begin{array}{ccc}
2AC & -ABC & -ACC'\\
-ABC & 2BC & 0\\
-ACC' & 0 & 2AC'
\end{array}}
\text{ ~ if } A',B' = 0
\end{equation}
which we keep only if the matrix has signature $(2,1)$ and its Galois conjugate is positive definite.
\end{lemma}

\begin{proof}
We know that $0 < \mu < 1$ since $R$ and $T$ intersect inside of $\Lambda^2$. By Lemma \ref{pairOfRoots}, we know that there exist $A,B\in\frako$ satisfying $A, B, \overline{A}, \overline{B} >0$, $AB < 4$, and $\overline{AB} < 4$ such that the inner product matrix of $r$ and $t$ is an $F$ multiple of 
\begin{equation}\label{ipRT}
\pz{\begin{array}{cc}
2A & -AB \\ -AB & 2B
\end{array}}
\end{equation}
If $T\not\perp R'$ then we apply Lemma \ref{pairOfRoots} again to get that there exist $A',B'\in\frako$ satisfying the same conditions as $A,B$ such that the inner product matrix of $r'$ and $t$ has the same form as \eqref{ipRT} with $A',B'$ in place of $A,B$.

We apply Lemma \ref{pairOfRoots} once again to $r$ and $r'$, using the fact that $\lambda < K^2$ to get that there exist $C,C'\in \frako$ satisfying $C,C',\overline{C},\overline{C'} >0$, $CC' < 4K^2$, and $\overline{CC'}<4$ such that the inner product matrix of $r$ and $r'$ is and $F$ multiple of 
\begin{equation}\label{ipRRp}
\pz{\begin{array}{cc}
2C & -CC' \\ -CC' & 2C'
\end{array}}
\end{equation}
We have the following:
$$\frac{2C'}{2C} = \frac{{r'}^2}{r^2} = \frac{{r'}^2t^2}{r^2t^2} = \frac{\pz{2A'}\pz{2B}}{\pz{2A}\pz{2B'}}$$
so in order to put these matrices together in a sensible way, we need
\begin{equation}\label{compatCheck}
AB'C' = A'BC
\end{equation}
Our choice of $C$ may $C'$ fail to satisfy \eqref{compatCheck} in two possible ways.  The ratio
\begin{equation}\label{ratio}
\frac{AB'C'}{A'BC}
\end{equation}
either is or is not a square unit.  If \eqref{ratio} is a square unit, then all hope is not lost.  As discussed in \ref{unitEquiv}, we get equivalent configurations of the roots $r$ and $r'$ if instead of $CC'$ we choose the factorization $\pz{Cu}\pz{u^{-1}C'}$ where $u$ is a square unit.

Making this substitution into \eqref{compatCheck}, we get \eqref{betaCompatability}.  So if a square unit can be found that makes \eqref{betaCompatability} hold, we can combine the 3 $2\times 2$ matrices into a $3\times 3$ matrix of the form \eqref{senoMatrix} by letting $\beta = A'BCu$ and choosing the scale at which $t^2 = 2BB'$.

In the case where $T\perp R'$, we take $A' = B' = 0$.  As before $r,t$ have inner product matrix an $F$ multiple of \eqref{ipRT} and $r,r'$ have inner product matrix an $F$ multiple of \eqref{ipRRp}.  The condition \eqref{betaCompatability} is trivially true since both sides of the equation are 0.  We put together the two matrices by choosing the scale at which $r^2 = 2AC$.  The resulting inner product matrix is \eqref{seoMatrix}.

\end{proof}

\begin{lemma}[short pair]\label{shortPair}
Suppose that $P$ has at least 5 edges, and $R,S,T,R'$ are consecutive edges with $(S,T)$ a short pair. Then the inner product matrix of $r,s,t,r'$ is one of finitely many possibilities up to scale. 

To list those possibilities, let $(A,B)$ vary over a set of representatives for the equivalence classes of all good factorizations of all $z\in \frako$ with $4 < z < 36$ and $\overline{z} < 4$.  Let $(A',B')$ vary over a set of representatives for the equivalence classes of all good factorizations of all $z\in\frako$ with $z,\overline{z} < 4$, and also $(0,0)$.  

For fixed $A,B,A',B'$, let $(C,C')$ vary over a set of representatives for the equivalence classes of all good factorizations of all $z\in\frako$ such that $4 < z < 4K^2$, $0 < \overline{z} < 4$, there exists a square unit $u$ such that \eqref{betaCompatability} holds if $A',B' >0$, and
\begin{equation}\label{N-compatability}
N:= 4+4\frac{CC'+\beta+A'B'}{AB-4}
\end{equation}
is an element of $\frako$.  As in Lemma \ref{shortEdge}, $\beta = AB'C'u^{-1}$.

The inner product matrix for $r,t,r'$ has the form \eqref{senoMatrix} or \eqref{seoMatrix}.  Since we only want arithmetic quadratic forms of signature $(2,1)$, we keep it on our list only if it has signature $(2,1)$ and its galois conjugate is positive definite.

Fixing $A,B,A',B',C,C'$, let $k$ vary over all positive elements of $\frako$ dividing $N$ (also up to equivalence in the sense of (U1) or (U2)).  For some such $A,B,A',B',C,C',k$, the inner product matrix of $r,s,t,r'$ is an $F$ multiple of
\begin{equation}\label{spnoMatrix}
\pz{\begin{array}{cccc}
2AB' & 0 & -ABB' & -\beta\\
0 & 2A'B\frac{N}{k^2} & 0 &-A'B\frac{N}{k}\\
-ABB' & 0 & 2BB' & -A'B'B\\
-\beta & -A'B\frac{N}{k} & -A'B'B & 2A'B
\end{array}}
\text{ ~ if } A',B' > 0
\end{equation}
\begin{equation}\label{spoMatrix}
\pz{\begin{array}{cccc}
2AC & 0 & -ABC & -ACC'\\
0 & 2AC'\frac{N}{k} & 0 & -AC'\frac{N}{k}\\
-ABC & 0 & 2BC & 0\\
-ACC' & -AC'\frac{N}{k} & 0 & 2AC'
\end{array}}
\text{ ~ if } A',B' = 0
\end{equation}
\end{lemma}

\begin{proof}
We apply the same argument as in the proof of Lemma \ref{shortEdge} with different bounds on $\mu$, and $\lambda$ to get that the inner product matrix of $r,t,r'$ is one of \eqref{senoMatrix} or \eqref{seoMatrix} up to scale.  What changes here is that $1 < \mu < 3$, so $4 < AB < 36$, and $\lambda > 1$ so $4 < CC' < 4k^2$.  Otherwise everything is identical.

Given a matrix of the form \eqref{senoMatrix} or \eqref{seoMatrix} we wish to build the possible $4\times 4$ inner product matrix for $r,s,t,r'$ of which it is a submatrix.  Consider the root $s$.  Let $s_0$ be the projection of $\hat{r}'$ onto the $F$-span of $s$, so that the projection of $r'$ is $\sqrt{{r'}^2}s_0$.  Since $s$ is a root, this lies in $\frac{s}{2}\frako$.  Therefore there exists some $k\in\frako$ with $k >0$ such that $\sqrt{{r'}^2}s_0 = -\frac{ks}{2}$.  This can be rearranged to get that
\begin{equation}\label{s}
s = -2\frac{\sqrt{{r'}^2}s_0}{k}
\end{equation}
We also have that $r'$ is a root, so $s\cdot r'\in\frac{{r'}^2}{2}\frako$, so there exists $M\in\frako$ with $M < 0$ such that $s\cdot r' = \frac{M{r'}^2}{2}$.  We have
\begin{align*}
M &= 2\frac{s\cdot r'}{{r'}^2}\\
&= 2\frac{\sqrt{{r'}^2} s\cdot s_0}{{r'}^2} \text{ ~ using the projection}\\
&= 4\frac{{r'}^2}{k{r'}^2}s_0^2\text{ ~ using \eqref{s}}\\
&= -\frac{4}{k}\pz{1 + \frac{\lambda^2+2\lambda\mu\mu'+{\mu'}^2}{\mu^2-1} } \text{ ~ using Lemma 4 from \cite{allcock2012reflective}}
\end{align*}
Writing $\lambda,\mu,\mu'$ in terms of $A,B,A',B',C,C'$ shows that $M = -\frac{N}{k}$ where $N$ is as in \eqref{N-compatability}.  Since $M,k\in\frako$, we have $N\in\frako$ with $k$ dividing $N$.  Now we can fill in the rest of the matrix.  Since $(S,T)$ is a short pair, we have $s\cdot r = s\cdot t = 0$.  We have
$$s\cdot r' = \frac{M{r'}^2}{2} = -\frac{N{r'}^2}{2k} $$
Finally, we can compute $s^2$:
$$s^2 =\pz{\frac{2}{k}\sqrt{{r'}^2}s_0}^2 = \frac{4{r'}^2}{k^2}s_0^2 = \frac{4{r'}^2}{k^2}\pz{-\frac{Mk}{4}} = -\frac{{r'}^2}{k}\pz{-\frac{N}{k}} - \frac{N{r'}^2}{k^2}$$

\end{proof}

\begin{lemma}[close pair]\label{closePair}
Suppose that $P$ has at least 6 edges, and $R,S,T,S',R'$ are consecutive edges with $(S,S')$ a close pair. Then the inner product matrix of $r,s,t,s',r'$ is one of finitely many possibilities up to scale. 

To list those possibilities, let $(A,B)$ vary over a set of representatives for the equivalence classes of all good factorizations of all $z\in\frako$ with $4 < z < 36$ and $0 < \overline{z} < 4$.  Let $(A',B')$ vary over the same set of pairs.

For fixed $A,B,A'B'$, let $(C,C')$ vary over a set of representatives for the equivalence classes of all good factorizations of all $z\in\frako$ such that $4 < z < 4K^2$, $\overline{z} < 4$, there exists a square unit $u$ such that \eqref{betaCompatability} holds, and $N,N'\in\frako$.  $N$ is defined as in \eqref{N-compatability} and $N'$ is similarly defined but with primed and non-primed letters swapped.  As in Lemma \ref{shortEdge}, $\beta = AB'C'u^{-1}$.

The inner product matrix for $r,t,r'$ has the form \eqref{senoMatrix} or \eqref{seoMatrix}.  Since we only want arithmetic quadratic forms of signature $(2,1)$, we keep only those matrices with $(2,1)$ where the galois conjugate is positive definite.

Fixing $A,B,A',B',C,C'$, let $k$ and $k'$ vary over all positive elements (up to equivalence in the sense of (U1) or (U2)) of $\frako$ dividing $N$ and $N'$ respectively such that 
\begin{equation}\label{gammaCompatability}
\frac{\gamma k^2}{A'BN}\text{ ~ and ~ } \frac{\gamma {k'}^2}{AB'N'}
\end{equation}
are both elements of $\frako$ where 
\begin{equation}\label{defineGamma}
\gamma = \frac{2\beta}{kk'}\pz{2+\frac{\beta}{CC'}-\frac{\pz{2CC'+\beta}^3}{\pz{AB-4}\pz{A'B'-4}\pz{CC'}^2}}
\end{equation}

For some such $A,B,A',B',C,C',k,k'$, the inner product matrix of $r,s,t,s',r'$ is an $F$ multiple of
\begin{equation}\label{cpMatrix}
\pz{\begin{array}{ccccc}
2AB' & 0 & -ABB' & -AB'\frac{N'}{k'} & -\beta\\
0 & 2A'B\frac{N}{k^2} & 0 & \gamma &-A'B\frac{N}{k}\\
-ABB' & 0 & 2BB' & 0 & -A'B'B\\
-AB'\frac{N'}{k'} & \gamma & 0 & 2AB'\frac{N'}{{k'}^2} & 0\\
-\beta & -A'B\frac{N}{k} & -A'B'B & 0 & 2A'B
\end{array}}
\end{equation}
\end{lemma}

\begin{proof}
We repeat the argument from Lemma \ref{shortPair} using the bounds $1 < \mu,\mu' < 3$ and $1 < \lambda < K$ to get $A,B,A',B',C,C',u$ satisfying \eqref{betaCompatability} so that the inner product matrix of $r,t,r'$ has the form \eqref{senoMatrix} up to scale.  The same argument in Lemma \ref{shortPair} that gets us $N$ and $k$ satisfying \eqref{N-compatability} gets us $N,k,N',k'$ here, by replacing primed letters by unprimed ones.  Thus the only entries remaining to be filled in to get the matrix \eqref{cpMatrix} are the $s\cdot s'$ ones.  As in Lemma \ref{shortPair} we use the fact that $s$ and $s'$ are both roots, so we may write them as \eqref{s} for $s$, and similarly for $s'$ with primed and unprimed letters switched.  We compute:
\begin{align*}
s\cdot s'_0 &= \frac{4}{kk'}\sqrt{r^2{r'}^2} s_0\cdot s'_0\\
&= \frac{4}{kk'}\sqrt{r^2{r'}^2}\left(\lambda+4-\frac{(\lambda+4)^3}{(\mu^2-2)(\mu'^2-1)}\right)\text{ ~ using Lemma 5 from \cite{allcock2012reflective}}
\end{align*}
Writing $\lambda,\mu,\mu'$ in terms of $A,B,A',B',C,C'$ shows that $s\cdot s' = \gamma$ as defined by \eqref{defineGamma}.  We get the condition \eqref{gammaCompatability} from the fact that $s$ and $s'$ are both roots, and so $\gamma = s\cdot s'$ lies in both $\frac{s^2}{2}\frako$ and $\frac{{s'}^2}{2}\frako$.
\end{proof}

\subsection{The box picture}

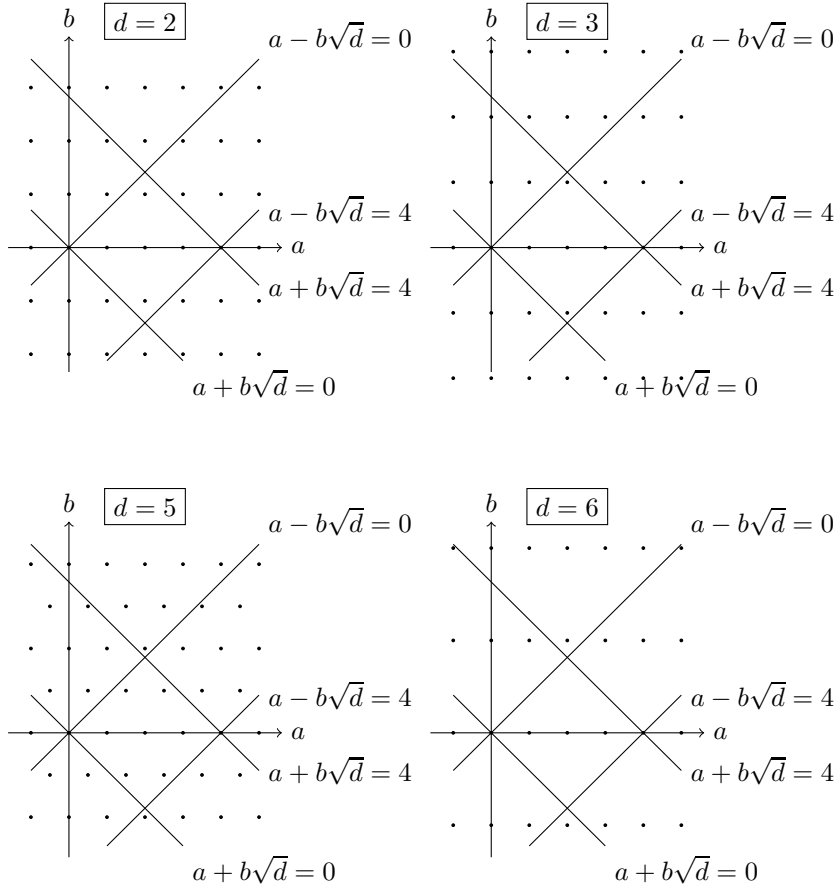
\begin{figure}[ht]
\begin{tikzpicture}[scale = .5]

\draw[->] (-1.6,0) -- (5.6,0) node[anchor = west]{$a$};
\draw[->] (0,-3.3) -- (0,5.6) node[anchor = south]{$b$};

  \foreach \x in { -1, 0, 1, 2, 3, 4, 5}
    \foreach \y in {-2.828, -1.414, 0, 1.414, 2.828, 4.243}
      \filldraw[shift={(\x,\y)}] (0,0) circle (1pt);

  \draw[domain=-1:3,variable=\x,black] plot ({\x},{-\x}) node[anchor = north west]{$a+b\sqrt{d} = 0$};
  \draw[domain=-1:5,variable=\x,black] plot ({\x},{(4-\x)}) node[anchor = west]{$a+b\sqrt{d} = 4$};
  \draw[domain=-1:5,variable=\x,black] plot ({\x},{\x}) node[anchor= south west]{$a-b\sqrt{d} = 0$};
  \draw[domain= 1:5,variable=\x,black] plot ({\x},{(\x-4)}) node[anchor = west]{$a-b\sqrt{d} = 4$};

\node[draw] at (2,6) {$d=2$};
\node[draw=none, fill=none] at (2,-6) {};

\end{tikzpicture}
\begin{tikzpicture}[scale = .5]
\draw[->] (-1.6,0) -- (5.6,0) node[anchor = west]{$a$};
\draw[->] (0,-3.3) -- (0,5.6) node[anchor = south]{$b$};

  \foreach \x in { -1, 0, 1, 2, 3, 4, 5}
    \foreach \y in {-3.464, -1.732, 0, 1.732, 3.464, 5.196}
      \filldraw[shift={(\x,\y)}] (0,0) circle (1pt);

  \draw[domain=-1:3,variable=\x,black] plot ({\x},{-\x}) node[anchor = north west]{$a+b\sqrt{d} = 0$};
  \draw[domain=-1:5,variable=\x,black] plot ({\x},{(4-\x)}) node[anchor = west]{$a+b\sqrt{d} = 4$};
  \draw[domain=-1:5,variable=\x,black] plot ({\x},{\x}) node[anchor= south west]{$a-b\sqrt{d} = 0$};
  \draw[domain= 1:5,variable=\x,black] plot ({\x},{(\x-4)}) node[anchor = west]{$a-b\sqrt{d} = 4$};

\node[draw] at (2,6) {$d=3$};
\node[draw=none, fill=none] at (2,-6) {};

\end{tikzpicture}
\begin{tikzpicture}[scale = .5]
\draw[->] (-1.6,0) -- (5.6,0) node[anchor = west]{$a$};
\draw[->] (0,-3.3) -- (0,5.6) node[anchor = south]{$b$};

  \foreach \x in { -1, 0, 1, 2, 3, 4, 5}
    \foreach \y in {-2.236, 0, 2.236, 4.472}
      \filldraw[shift={(\x,\y)}] (0,0) circle (1pt);
      
  \foreach \x in { -1/2, 1/2, 3/2, 5/2, 7/2, 9/2 }
    \foreach \y in {-1.118, 1.118, 3.354}
      \filldraw[shift={(\x,\y)}] (0,0) circle (1pt);

  \draw[domain=-1:3,variable=\x,black] plot ({\x},{-\x}) node[anchor = north west]{$a+b\sqrt{d} = 0$};
  \draw[domain=-1:5,variable=\x,black] plot ({\x},{(4-\x)}) node[anchor = west]{$a+b\sqrt{d} = 4$};
  \draw[domain=-1:5,variable=\x,black] plot ({\x},{\x}) node[anchor= south west]{$a-b\sqrt{d} = 0$};
  \draw[domain= 1:5,variable=\x,black] plot ({\x},{(\x-4)}) node[anchor = west]{$a-b\sqrt{d} = 4$};

\node[draw] at (2,6) {$d=5$};

\end{tikzpicture}
\begin{tikzpicture}[scale = .5]
\draw[->] (-1.6,0) -- (5.6,0) node[anchor = west]{$a$};
\draw[->] (0,-3.3) -- (0,5.6) node[anchor = south]{$b$};

  \foreach \x in { -1, 0, 1, 2, 3, 4, 5}
    \foreach \y in {-2.450, 0, 2.450, 4.899}
      \filldraw[shift={(\x,\y)}] (0,0) circle (1pt);

  \draw[domain=-1:3,variable=\x,black] plot ({\x},{-\x}) node[anchor = north west]{$a+b\sqrt{d} = 0$};
  \draw[domain=-1:5,variable=\x,black] plot ({\x},{(4-\x)}) node[anchor = west]{$a+b\sqrt{d} = 4$};
  \draw[domain=-1:5,variable=\x,black] plot ({\x},{\x}) node[anchor= south west]{$a-b\sqrt{d} = 0$};
  \draw[domain= 1:5,variable=\x,black] plot ({\x},{(\x-4)}) node[anchor = west]{$a-b\sqrt{d} = 4$};

\node[draw] at (2,6) {$d=6$};

\end{tikzpicture}

\caption{The box picture for bounds $0 <AB,\overline{AB} < 4$ with $d = 2$, $d = 3$, $d = 5$, $d = 6$.}
\label{boxPicture}
\end{figure}

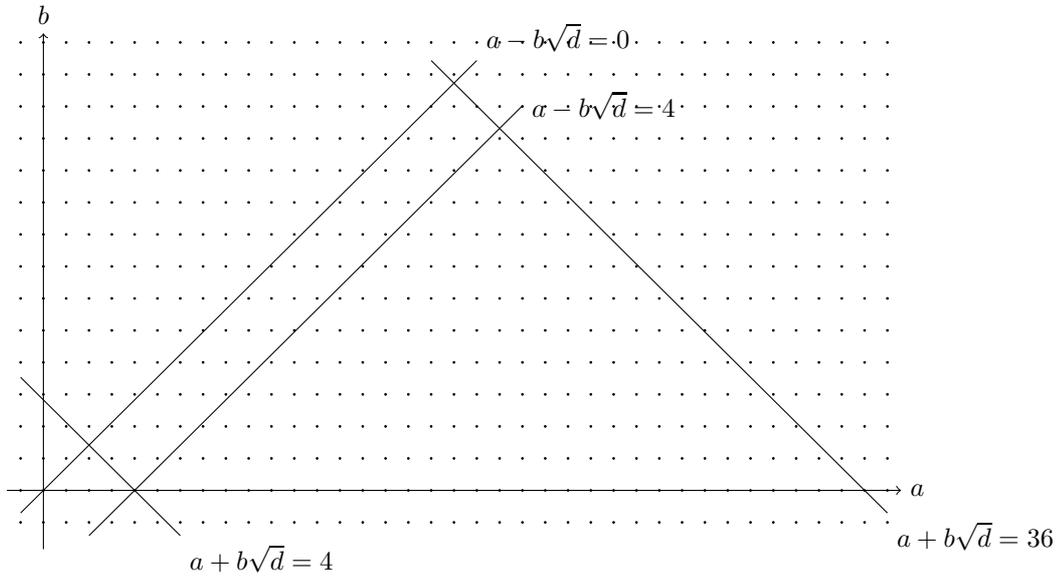
\begin{figure}[ht]
\begin{tikzpicture}[scale = .3]

\draw[->] (-1.6,0) -- (37.6,0) node[anchor = west]{$a$};
\draw[->] (0,-2.6) -- (0,20.2) node[anchor = south]{$b$};

  \foreach \x in {-1,0,1,2,3,4,5,6,7,8,9,10,11,12,13,14,15,16,17,18,19,20,21,22,23,24,25,26,27,28,29,30,31,32,33,34,35,36,37}
    \foreach \y in {-1.414,0,1.414,2.828,4.243,5.657,7.071,8.485,9.899,11.314,12.728,14.142,15.556,16.971,18.385,19.799}
      \filldraw[shift={(\x,\y)}] (0,0) circle (1pt);

  \draw[domain=17:37,variable=\x,black] plot ({\x},{(36-\x)}) node[anchor = north west]{$a+b\sqrt{d} = 36$};
  \draw[domain=-1:6,variable=\x,black] plot ({\x},{(4-\x)}) node[anchor = north west]{$a+b\sqrt{d} = 4$};
  \draw[domain=-1:19,variable=\x,black] plot ({\x},{\x}) node[anchor= south west]{$a-b\sqrt{d} = 0$};
  \draw[domain= 2:21,variable=\x,black] plot ({\x},{(\x-4)}) node[anchor = west]{$a-b\sqrt{d} = 4$};

\end{tikzpicture}
\caption{The box picture for the bounds $4 < AB < 36$, $0 < \overline{AB} < 4$ from Lemma \ref{shortPair} and \ref{closePair} shown with $d = 2$}
\label{boxPictureSPCP}
\end{figure}

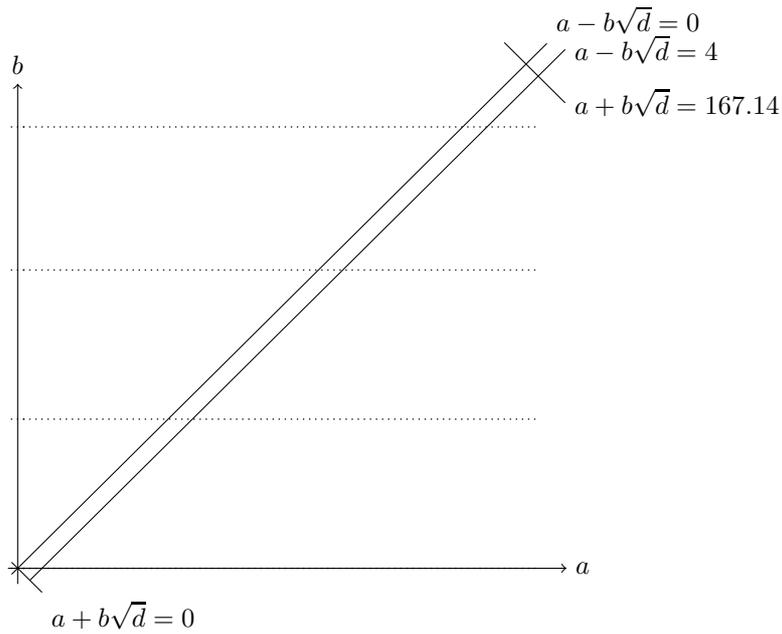
\begin{figure}[ht]
\begin{tikzpicture}[scale = .08]

\draw[->] (-1.6,0) -- (90.2,0) node[anchor = west]{$a$};
\draw[->] (0,-2.6) -- (0,80.3) node[anchor = south]{$b$};

  \foreach \x in {-1,0,1,2,3,4,5,6,7,8,9,10,11,12,13,14,15,16,17,18,19,20,21,22,23,24,25,26,27,28,29,30,31,32,33,34,35,36,37,38,39,40,41,42,43,44,45,46,47,48,49,50,51,52,53,54,55,56,57,58,59,60,61,62,63,64,65,66,67,68,69,70,71,72,73,74,75,76,77,78,79,80,81,82,83,84,85,}
    \foreach \y in {0,24.718,49.436,73.155}
      \filldraw[shift={(\x,\y)}] (0,0) circle (1pt);

  \draw[domain=80:90,variable=\x,black] plot ({\x},{(167.14-\x)}) node[anchor = west]{$a+b\sqrt{d} = 167.14$};
  \draw[domain=-1:4,variable=\x,black] plot ({\x},{-\x}) node[anchor = north west]{$a+b\sqrt{d} = 0$};
  \draw[domain=-1:87,variable=\x,black] plot ({\x},{\x}) node[anchor= south west]{$a-b\sqrt{d} = 0$};
  \draw[domain= 2:90,variable=\x,black] plot ({\x},{(\x-4)}) node[anchor = west]{$a-b\sqrt{d} = 4$};

\end{tikzpicture}

\caption{The box picture for the bounds $0 < CC' < 167.14$, $0 < \overline{CC'} < 4$ from Lemma \ref{shortEdge} shown with the large value of $d$, $d=611$}
\label{boxPictureSE}
\end{figure}

The bounds on $\mu,\mu',\text{ and }\lambda$ can be used to show that for $d$ large enough, there are no reflective arithmetic lattices of signature $(2,1)$ defined over $\Q(\sqrt{d})$.  The argument in Theorem \ref{discrBounds} could be adapted to get similar bounds on the discriminant for a field $F/\Q$ of any fixed degree.

\begin{theorem}[Discriminant bounds]\label{discrBounds}
Suppose $d>0$ is a square-free integer, and $F = \Q(\sqrt{d})$ with $\frako_F$ a PID.
\begin{enumerate}
\item If $d > 27935$, or if $d \equiv 2\text{ or }3 \mod 4$ and $d > 6983$ then 
there are no no polygons that are fundamental chambers for reflective lattices with ground field $F$, so there can be no such lattices.
\item If $d > 1296$, or if $d \equiv 2 \text{ or } 3 \mod 4$ and $d > 324$ then there are no polygons with a short pair or close pair that are fundamental chambers for reflective hyperbolic lattices with ground field $F$, so there can be no such lattices.
\end{enumerate}
\end{theorem}
\begin{proof}
The ring of integers $\frako$ in a quadratic extension of $\Q$ can be thought of as a lattice in $\R^2$.  Each element $z\in \frako$ can be written as 
$$z = a+b\sqrt{d}$$
If $d\equiv 2\text{ or }3\mod 4$ then $a,b\in \Z$.  If $d\equiv 1\mod 4$ then either $a,b\in\Z$ or $a,b\in\frac{1}{2}\Z\setminus\Z$.  The point in $\R^2$ corresponding to $a+b\sqrt{d}$ will be $(a,b\sqrt{d})$.

Let $D_i$, $i=0,1,2,3$ be the coset of $i\mod 4$.  Suppose $d > d'$, and $\frako,\frako'$ are the rings of integers in $\Q(\sqrt{d}),\Q(\sqrt{d'})$ respectively.  Notice that if $d,d'\in D_1$, then the vertical spacing between the lattice points of $\frako$ is greater than for $\frako'$.  This is also true for $d,d'\in D_2\cup D_3$.

The inequalities used to find $A,B,A',B',C,C'$ in Lemmas \ref{shortEdge}, \ref{shortPair}, and \ref{closePair} describe lines in the plane that specify regions where the lattice points corresponding to $AB,A'B',\text{ and } CC'$ must lie.  For example, the condition $0 < z < 4$ says that the lattice point corresponding to $z=a+b\sqrt{d}$ lies between two downward sloping lines given by the equations $a+b = 0$ and $a+b = 4$.  The condition $ 0 < \overline{z} < 4$ says that the lattice point corresponding to $z$ lies between two upward sloping lines given by $a-b = 0$ and $a-b = 4$.  In this way, upper and lower bounds on both $z$ and its conjugate $\overline{z}$ define a box in which the corresponding lattice point in $\R^2$ lives.  The bounds do not vary with $d$, since they come from the geometry of a hyperbolic polygon.  For large values of $d$, the boxes have very few or no points in them.  In the example in Figure \ref{boxPicture}, the only points the only points left in the box for $d >5$ will be $1,2,\text{ and }3$.

The box in which the point corresponding to $AB$ lives for both the short pair and close pair cases is shown in Figure \ref{boxPictureSPCP}, and is empty for the values of $d$ stated in (2).

The box in which the point corresponding to $CC'$ lives for the short edge case is never empty because it always contains the points corresponding to $1,2,\text{ and }3$.  However, for the values of $d$ specified in (1), these are the only 3 points in that box.  Having $CC' = 1,2,\text{ or }3$ means that $R$ and $R'$ intersect, so $P$ is a triangle.  All arithmetic triangles were described by Takeuchi in \cite{takeuchi1977arithmetic}.  For all the ones whose ground field is a quadratic extension of $\Q$, that ground field is one of $\Q(\sqrt{2})$, $\Q(\sqrt{3})$, $\Q(\sqrt{5})$, $\Q(\sqrt{6})$ and $2,3,5,\text{ and }6$ are all smaller than $27935$.

\end{proof}

\subsection{Results of this section}
We wrote a C program using the PARI library that iterates over all the points in the boxes in order to build all matrices of the form \eqref{senoMatrix}, \eqref{seoMatrix}, \eqref{spnoMatrix}, \eqref{spoMatrix}, and \eqref{cpMatrix} with entries in $\Q(\sqrt{2})$.  We removed all matrices that were non-arithmetic or did not have rank 3.  We then computed all reflection stable enlargements of each root configuration, and separated out the squarefree ones from the non-squarefree ones.  We also removed any where the root sequence was not a chain of roots due to not being locally simple.
The number of things on each list is summarized in Table \ref{listSummary}.  We did not take care at this point to prevent redundancies in our enumeration, and so the numbers in the table are large overestimates.  These computations took varying amounts of time to run on a personal laptop.  The close pair cases took about 10 days, the short pair non-orthogonal cases took about 2, and everything else took a few hours or even less.
\begin{table}[ht]
\begin{tabular}{cccc}
matrix type & root configs & enlargements & squarefree enlargements\\
\hline
\eqref{senoMatrix} & 282 & 1489 & 130\\
\eqref{seoMatrix} & 1736 & 7181 & 341\\
\eqref{spnoMatrix} & 5777 & 3653 & 233\\
\eqref{spoMatrix} & 88836 & 38871 & 1096\\
\eqref{cpMatrix} & 97526 & 13688 & 581
\end{tabular}
\caption{Summary of root configurations on our list}


\label{listSummary}
\end{table}

In Table \ref{listSummary}, the numbers in ``root configs'' column are how many root configurations there are of each type.  The types come from Lemmas \ref{shortEdge}, \ref{shortPair}, and \ref{closePair}.  The number in the ``enlargements'' column is how many reflection stable enlargements therre are for each type of root configuration became.  In the short pair and close pair cases, it is smaller than the total number of root configurations.  This is because reflection-stable enlargments of a subset of the roots were discarded when the remaining roots were not primitive.  The number in the ``squarefree enlargements'' column is the number of enlargements that are strongly squarefree.

\section{Determining Reflectivity}\label{determiningReflectivity}

The next step in the proof of Theorem \ref{main} is to determine which of the quadratic forms on the finite list from the previous section (summarized in Table \ref{listSummary}) are reflective and which are not.  The tool that is usually used for this sort of computation is Vinberg's algorithm.  We wrote and implemented a version of his algorithm for our problem, but found that it would be impossible to use it to do the entire calculation, since that would have taken an unreasonable amount of time.  We therefore came up with a new way to structure the search for new roots, which we think of as ``walking around the chamber.''  The key difference between Vinberg's algorithm and walking is the search space.  Whereas Vinberg's algorithm searches for new roots inside an $n$-dimensional polygonal cone, walking has a much more restricted search space.  The roots we seek satisfy additional inner product conditions that make the search space $1$-dimensional.  The idea is to build a chain of roots for a boundary component of the chamber starting from a single corner.  The new technique that makes walking fast enough to finish the computation was a method for finding the nearest translate of a corner of the chamber along an edge.  This technique also gave us a way of estimating computation times for searches done by walking.

With walking and finding the nearest translate, we were able to determine reflectivity for all but 48 of the strongly squarefree lattices on our list, and we were able to give time estimates for how long the 48 remaining cases would take.  These time estimates were all unreasonably long (as high as $10^{44}$ days).  The methods we used to resolve these cases were more technical, and only worked because a portion of the chamber had already been found by walking.  As one might reasonably suspect, none of the remaining 48 lattices were reflective.  A description of walking and the method for finding the nearest translate make up the bulk of this section.  The resolution of the final 48 cases comes at the end.

Because we now have a finite list of lattices, we know the following.

\begin{corollary}[Corollary of Lemma \ref{allA2}] All of the chambers for the lattices listed at the end of Section \ref{finite} enumerated in Table \ref{listSummary} have at least one boundary component with no consecutive $A_2$ corners.
\end{corollary}
\begin{proof}
The data for each of the lattices on our list is a quadratic form and a chain of $3,4,\text{ or }5$ (not necessarily simple) roots.  At least one corner in each chain is not an $A_2$ corner. By Lemma \ref{allA2}, consecutive $A_2$ corners will never occur in the boundary component of that chamber containing that corner.
\end{proof}

For the remainder of this section, $L$ will be a SSF lattice from the list tabulated in Table \ref{listSummary}.  We fix a choice of future cone $\mathfrak{C}^+$ and an orientation on $V = L\otimes F$.  For each lattice, we have a chain of $3,4,\text{ or }5$ roots.  This chain of roots usually does not bound a single chamber.  Since we wish to find a chain of roots that does bound a single copy of the chamber, we begin our searches with a consecutive pair of roots from our chain, which we call $r_1$ and $r_2$.   Since they are consecutive we know that they are part of a system of simple roots for $L$.

\subsection{Simple and non-simple roots at a corner}

Recall that the roots at a corner of the chamber are simple roots for a positive definite lattice of type $A_1^2, A_2, B_2, G_2,$ or $I_2(8)$.  A lattice of type $B_2$ also contains a lattice of type $A_1^2$.  The simple roots of the $A_1^2$ sublattice are roots of the $B_2$ lattice, though they are not simple roots of it.  Similarly, a lattices of type $G_2$ contain sublattices of type $A_1^2$ and $A_2$.  Lattices of type $I_2(8)$ contain sublattices of type $A_1^2$ and $B_2$.  

Sometimes we will need to work with roots at a corner that are not simple roots  of $L$, but are simple roots for a sublattice of $L$.  In particular, at $B_2,G_2,\text{ and }I_2(8)$ corners we would sometimes like to work with the simple roots of the $A_1^2$ sublattice.  

Conversely, if we have a pair of roots that span a sublattice of type $A_1^2$, we would like to have a way of determining whether they are contained in a sublattice of $L$ of type $B_2,G_2,\text{ or } I_2(8)$, or none of these.  That is, we would like to be able to say whether or not they are simple roots of $L$.

The following lemma gives us a way to tell whether roots whose mirrors intersect are simple roots of $L$, and a way to go back and forth between simple and non-simple roots at a corner.


\begin{lemma}\label{simpleRoots}Let $r$ and $s$ be roots of $L$ such that $V_r\cap V_s$ is negative definite.
\begin{enumerate}
\item Suppose $r$ and $s$ are simple roots such that the angle between the associated lines $R$ and $S$ in $\Lambda^2$ is $\frac{\pi}{m}$ with $m = 4,6,\text{ or }8$.  Then there is a root $r'$ such that $V_{r'}\cap V_s = V_r\cap V_s$ and the line associated to $r'$ in $\Lambda^2$ is orthogonal to $S$.

\item If $r$ and $s$ are roots that are orthogonal but possibly not simple, then there is a way to find the root $r'$ such that  $V_{r'}\cap V_s = V_r\cap V_s$ and $r'$ and $s$ are simple roots.
\end{enumerate}
\end{lemma}
\begin{proof} These constructions are, in some sense, opposites of each other.
\begin{enumerate}
\item For each $m$, there are a few reflections we can apply to $r$ and/or $s$ to obtain $r'$:

If $m = 4$, then $r' = R_r(s)$.  Note that in this case $r'^2 = s^2$

If $m = 6$, let $s' = R_r(s)$.  Then $r' = R_{s'}(r)$.  Note that in this case $r'^2 = r^2$, and by Lemma \ref{angles-and-norms}, $r^2$ is equal to either $3s^2$ or $\frac{s^2}{3}$.

If $m = 8$, let $r'' = R_r(s)$.  Then $r' = R_{r''}(s)$  Note that in this case $r'^2 = s^2$

\item By part (1), if $r$ and $s$ are not simple, then up to scaling $r$ or $s$ by a unit one of the following is true:
\begin{enumerate}
\item $r^2 = s^2$ and the angle between the lines $R'$ and $S$ is one of $\frac{\pi}{4}$ or $\frac{\pi}{8}$
\item $r^2$ is equal to either $3s^2$ or $\frac{s^2}{3}$ and the angle between the lines $R'$ and $S$ is $\frac{\pi}{6}$
\end{enumerate}

We do not know whether $r$ and $s$ are simple roots or not.  If their norms do not satisfy (a) or (b), they must be simple and there is nothing further to check.

If $r^2 = s^2$, we let $t$ be a primitive lattice vector in the span of $r-s$. We check whether $t$ is a root by checking whether the reflection $R_t$ preserves $L$.  If $t$ is not a root, then $r$ and $s$ were simple.  If $t$ is a root, we must check whether or not $s$ and $t$ are simple. If $t^2 \neq s^2$, then there is no mirror bisecting the angle between $t$'s mirror and $s$'s mirror, so $t$ is the desired root $r'$.  If $t^2 = s^2$, we let $t'$ be a primitive lattice vector in the span of $t-s$, and we check whether $t'$ is a root by checking whether the reflection $R_{t'}$ preserves $L$.  If it is a root, then $r' = t'$, and if not then $r' = t$.

If $r^2 = 3s^2$, then we let $t$ be a primitive lattice vector in the span of $\frac{r-3s}{2}$.  If $t$ is a root, then $t$ is the desired simple root.  Otherwise $r$ and $s$ are simple roots.  If $r^2 = \frac{s^2}{3}$ then we switch the labels on $r$ and $s$ and run the same argument.
\end{enumerate}

\end{proof}

\begin{lemma}\label{reflIsRoot}
Let $r$ and $s$ be simple roots at a corner of the chamber $C$.  Let $q$ be the projection of $r$ onto $V_s$
$$q = r-\frac{r\cdot q}{q^2}q$$
Then the image of $r$ under reflection with respect to $q$ is a root.
\end{lemma}
\begin{proof}
If the sublattice spanned by $r$ and $s$ has any type except $A_2$, then by Lemma \ref{simpleRoots} there is a root $t$ in the span of $q$.  The reflection $R_t$ is an automorphism of $L$, and so $R_q(r) = R_t(r)$ is a root.

If the sublattice spanned by $r$ and $s$ has type $A_2$, then there is no root in the span of $q$ since if there were the corner would have type $G_2$ and $r$ and $s$ would not be simple roots.  Recall that we may assume $r^2 = s^2$.  The composition of reflections
$$R_r\circ R_s$$
is an order $3$ rotation preserving the vector $p$ that lies along the corner, and
$$R_q(r) = R_r\circ R_s(r)$$
Thus $R_q(r)$ is a root, even though $R_q$ is not an element of $\Aut(L)$.
\end{proof}

\subsection{Finding the shortest translation along a line} \label{posTransDir}

Recall that the hyperbolic plane $\Lambda^2$ is tiled by chambers for the reflection subgroup $\Gamma$ of $\Aut(L)$.  If $r$ is a root of $L$, then the image of $V_r\cap\mathfrak{C}^+$ in $\Lambda^2$ is a line $R$ that contains an edge of a copy of the chamber.  Let $H_r\leq \Aut(L)$ be the subgroup consisting of all translations along $R$.  Fix a chamber $C$ with an edge contained in $R$.  We call the image of $C$ under an element of $H_r$ a translate of $C$ along $R$.  Since $\Aut(L)$ is discrete, if $H_r$ is nontrivial then $C$ has a nearest translate in either direction along $R$.  When $H_r$ is nontrivial, it is isomorphic to $\Z$.  We will give a method for computing a generator of $H_r$ in this case. 

Let $\{r_1,r_{2},p_1\}$ be a corner basis at a corner $c_1$ of $C$, suppose that $H_{r_2}$ is nontrivial, and let $\phi$ be a generator of $H_{r_2}$.  Then $\phi$ and $\phi^{-1}$ translate in opposite directions along $R_2$.  We will establish a convention by which one direction is positive and the other negative. Intuitively, the positive direction is to the side of $R_1$ where the next edge of the chamber $C$ would go.  To make this precise, consider $\phi(r_1)$.  Let $s$ be the projection of $\phi(r_1)$ onto $r_2^{\perp}$.  By Lemma \ref{reflIsRoot}, we know that 
$$r_1' = R_s(\phi(r_1))$$
is a root of $L$.  We say that $\phi$ is a translation in the positive direction if $\{r_1,r_2,r_3 := r_1'\}$ is a chain of roots. 

We will now describe the process by which we can write down a matrix for $\phi$.  We make use of the classical correspondence, due to Gauss and Dedekind, between quadratic forms defined over a field and ideals in quadratic extensions of that field.  For an exposition, see \cite{koch2000number} Chapter 9. Let
$$q = 2\pz{r_1-\frac{r_1\cdot r_2}{r_2^2}}$$
be twice the projection of $r_1$ onto $r_2^{\perp}$.  The lattice spanned by $q$ and $p$ is a rank 2 integral sublattice of $L$. If $D = -q^2p^2$ is squarefree in $F$, then $\lz{q,p}$ is isometric to the ring of integers $\frako_K$, which is a lattice in the quartic field
$$K = F(\sqrt{D}) = \Q(\sqrt{2},\sqrt{D})$$
with quadratic form given by the norm $N_{K/F}$.  Define $\varphi:\lz{p,q}_F\rightarrow K$ by 
\begin{equation}\label{translateIsom}
\varphi:\begin{array}{rl}
q & \mapsto 1\\
q^2p & \mapsto \sqrt{-p^2q^2}
\end{array}
\end{equation}
The inner product matrix for $q,q^2p\in L$ is $q^2$ times the one for $1, -p^2q^2\in K$.  

\begin{lemma}\label{nearest}
Let 
\begin{equation}\label{defG}
G = \{u = a+b\sqrt{D} \in U(K): a,b\in F, a >0, N_{K/F}(u) = 1\}
\end{equation}
(Note that $a$ and $b$ might not be elements of $\frako_F$, since $1$ and $\sqrt{D}$ may not be an integral basis for $K/F$.) Then $G$ is a rank $1$ free subgroup of $U(K)$, and $H_{r_2}$ is isomorphic to a (finite index) subgroup $H \leq G$ such that under the correspondence given by \eqref{translateIsom}, the translation taking the corner $c_1$ to its nearest translate along $R_2$ corresponds to a generator of $H$.
\end{lemma}

\begin{proof}
Let $\varphi$ be the map defined by \eqref{translateIsom}.  Suppose $m$ is an isometry of $K$.  Then $M = \varphi^{-1}\circ m \circ \varphi$ is an isometry of the subspace $r_2^{\perp}$ of $V = L\otimes F$.  $M$ can be extended to all of $V$ by declaring that it fixes the subspace spanned by $r_2$.  We are interested in the isometries of $K$ that induce isometries of $V$ preserving $L$, $\mathfrak{C}^+$, and the orientation on $V$.  These isometries descend to hyperbolic translations that preserve the line in $\Lambda^2$ containing $R_2$.

Let $m$ be an isometry of $K$ defined on the basis $1,\sqrt{D}$ for $K$ by 
\[\left\{\begin{array}{ccc}
1 & \mapsto & a+b\sqrt{D}\\
\sqrt{D} & \mapsto & c+d\sqrt{D}
\end{array}\right\}\]
Since $m$ is an isometry we have
\begin{equation}\label{imageNorms1}
1 = N_{K/F}(m(1)) = a^2-b^2D
\end{equation}
and
\begin{equation}\label{imageNorms2}
-D = N_{K/F}(m(\sqrt{D})) = c^2-d^2\sqrt{D}
\end{equation}
Alternative forms of \eqref{imageNorms1} and \eqref{imageNorms2} that will be useful for our calculations are
\begin{equation}\label{imageNorms}
b^2D = a^2-1\text{ and } c^2 = D(d^2-1)
\end{equation}
The fact that $m$ is an isometry also means that $m(1)$ and $m(\sqrt{D})$ must have inner product $0$.  We compute their inner product.
\begin{align*}
m(1)\cdot m(\sqrt{D}) &= \frac{N_{K/F}(m(1) +m(\sqrt{D})) -N_{K/F}(m(1)) -N_{K/F}(m(\sqrt{D}))}{2}\\
&= \frac{N_{K/F}(a+c +(b+d)\sqrt{D}) -1+D}{2}\\
&= \frac{(a+c)^2-(b+d)^2D -1+D}{2}\\
&= \frac{a^2+2ac+c^2-b^2D-2bdD-d^2D-1+D}{2}\\
&= \frac{a^2+2ac+D(d^2-1) - (a^2 - 1) -2bdD-d^2D -1 +D}{2} \text{ ~ using \eqref{imageNorms}}\\
&= ac-bdD
\end{align*}
Thus we need
\begin{equation}\label{ipZeroCond}
ac = bdD
\end{equation}
If we square both sides of \eqref{ipZeroCond} and make substitutions from \eqref{imageNorms}, we get that we need
\begin{align*}
a^2D(d^2-1) &= (a^2-1)Dd^2\\
a^2d^2-a^2 &= a^2d^2 -d^2\\
a^2 &= d^2\\
a &= \pm d
\end{align*}
By \eqref{ipZeroCond}, if $d = a$ then  $c = bD$, and if $d = -a$ then $c = -bD$.  Thus if $m$ is an isometry of $K$, then $m$ acts on $K$ either by scaling by $u$, or by scaling $1$ by $u$ and scaling $\sqrt{D}$ by $-u$ where $u = a+b\sqrt{d}\in U(K)$ has norm $1$.  We will show that if the map induced by $m$ on $V$ preserves the positive cone $\mathfrak{C}^+$ and the orientation on $V$ then only the first of these options is possible, and also we must have $a > 0$.

If the map $m$ induces on $V$ is to preserve the future cone $\mathfrak{C}^+$, then $m(\sqrt{D})$ needs to have negative inner product with $\sqrt{D}$.  We compute their inner product.
\begin{align*}
\sqrt{D}\cdot m_u(\sqrt{D}) &= \frac{N_{K/F}(\sqrt{D} + m_u(\sqrt{D})) - N_{K/F}(\sqrt{D}) -N_{K/F}(m_u(\sqrt{D}))}{2}\\
&= \frac{N_{K/F}(\sqrt{D} + c+d\sqrt{D}) + 2D}{2}\\
&= \frac{c^2 -(d+1)^2D +2D}{2}\\
&= \frac{c^2 -d^2D-2dD-D +2D}{2}\\
&= \frac{N_{K/F}(c+d\sqrt{D}) -2dD +D}{2}\\
&= \frac{-D -2dD +D}{2}\\
&= -dD
\end{align*}
Since $D$ is positive, $d$ must be positive in order for this to be negative.  If the map $m$ induces on $V$ also preserves the orientation on $V$, then we need $m(1)$ to have positive inner product with $1$.  A similar calculation to the one above shows that we must have $a >0 $.  Thus $a=d>0$, and $m$ is multiplication by $u = a+b\sqrt{D}$.  Together, $N_{F/K}(u) = 1$ and $a>0$ imply that $u, u^{-1} > 0$.  The maps induced by inverses are translations in opposite directions by the same amount.  

The element $u$ preserves a finitely generated $\frako_F$-submodule of $K$ that contains the $\frako_F$-module generated by $1$ and $\sqrt{D}$.  In particular, $u$ is contained in a finitely generated $\frako_F$ module, so $\frako_F[u]$ is finitely generated as an $\frako_F$-module.  Therefore $u\in \frako_K$, so $u$ is an element of the unit group $U(K)$.

We now show that 
$$G = \{u=a+b\sqrt{D}: N_{K/F}(u) = 1\text{ and }a>0\}$$
is a rank 1 subgroup of $U(K)$.  First we show that $G$ is a subgroup.  Let $v = a+b\sqrt{D}, w = a'+b'\sqrt{D}\in G$.  Since elements of $G$ have norm $1$, the inverse of $v$ is its conjugate, $a-b\sqrt{D}$, which is also an element of $G$.  We have
$$vw = (a+b\sqrt{D})(a'+b'\sqrt{D}) = aa'+bb'D + (ab'+ba')\sqrt{D}$$
Since $v$ and $w$ both have norm $1$, their product also has norm $1$.  We want to show 
\begin{equation}\label{weWant}
aa'+bb'D >0 
\end{equation}
If  $b$ and $b'$ have the same sign, then \eqref{weWant} is true since the lefthand side is a sum of positive numbers.  If $b$ and $b'$ have opposite signs, then we use the fact that
$$a^2-b^2D = 1 > 0$$
from which it follows that
$$a > |b|\sqrt{D}$$
and likewise
$$a' > |b'|\sqrt{D}$$
Thus,
$$aa' >|bb'|D$$
so we have
$$aa'-|bb'|D >0$$
Since $bb' <0$, $|bb'| = -bb'$, so \eqref{weWant} holds.

We now show that $G$ has rank 1.  Because the elements of $G$ have norm 1, they live in the kernel of the restriction of $N_{K/F}$ to $U(K)$.  Dirichlet's Unit Theorem says that the rank of the unit group of $K$ is $s_1+s_2-1$, where $s_1$ is the number of real embeddings of $K$ into $R$, and $s_2$ is the number of conjugate pairs of complex embeddings of $K$ into $\mathbb{C}$.  By arithmeticity, 
$$-p^2q^2 > 0\text{ and }-\overline{p^2q^2}<0,$$
so $K$ has exactly 2 real embeddings and one pair of complex embeddings.  Thus $U(K)$ has rank 2.  If $u\in U(F)\setminus\{\pm 1\}$, then $N_{K/F}(u) = u^2$.  Thus if we restrict $N_{K/F}$ to $U(K)$, its image in $U(F)$ is nontrivial.  Therefore the image has rank 1, and so the kernel also has rank 1.  The kernel consists of all units of norm $1$, and it isomorphic to $\Z\times\Z/2$.  The subgroup $G$ consisting of only those units $a+b\sqrt{D}$ with $a>0$ is isomorphic to $\Z$.

Let $H$ be the subgroup of $G$ whose induced maps on $V$ preserve $L$.  If $H$ is nontrivial, then it has finite index in $G$, and is generated by some power of $u$.

\end{proof}

\subsection{Walking}

We begin with a lemma that will give us the termination condition for the walking algorithm.

\begin{lemma}\label{pigeonhole}
Let $\Phi = \{r_i\}_{i\in I}$ be a chain of roots all of whose edges bound a single copy $C$ of the chamber for $L$.  Let $c_i$ be the corner of the chamber formed by the roots $r_i,r_{i+1}\in\Phi$.  If $\Phi$ is large enough, then there is a lattice automorphism $\psi$ taking a corner basis at $c_i$ to a corner basis at $c_j$ for some $i\neq j$.  By applying all powers of $\psi$ to the roots in $\Phi$ we obtain all of the roots whose mirrors make up a single boundary component of $C$.  In particular, if $C$ has a boundary component with at least 2 edges, then it has infinitely many edges if and only if it has an automorphism of infinite order.
\end{lemma}

\begin{proof}
By Lemma \ref{angles-and-norms} there are finitely many $m$ for which $\frac{\pi}{m}$ could be an angle of the chamber $C$.  Each corner $c_i$ has corner basis $\{r_i,r_{i+1},p_i\}$ with 
$$L/\lz{r_i,r_{i+1},p_i}$$
a finite $\frako$-module.  As we showed for the $A_2$ case in Lemma \ref{A2Corners}, there are finitely many non-isomorphic ways to glue $\lz{r_i,r_{i+1}}$ to $\lz{p_i}$.  The same is true for the other possible corner angles.  If $c_i$ and $c_j$ are two corners of the same type with the same gluing, then the linear transformation defined by $(r_i,r_{i+1},p_i)\mapsto (r_j,r_{j+1},p_j)$ is an automorphism of $L$ that preserves $\mathfrak{C}^+$ and the orientation on $L\otimes F$.  

If $I$ is large enough then by the pigeonhole principle there are two corners in the boundary component of $C$ that have the same type and the same glue.  Let $\psi$ be the automorphism taking one to the other.  Since the mirrors of the roots in $\Phi$ all bound a single copy of the chamber, $\psi$ preserves adjacency of roots.  If $\psi$ has finite order $k$, then if we apply $\psi,\psi^2, \ldots,\psi^{k-1}$ to the roots of $\Phi$ and adjoin them to $\Phi$ we get a closed chain of roots $\Phi$ that bound a chamber with finitely many sides.  

If $\psi$ has infinite order, then 
$$\bigcup_{m\in\Z}\psi^m(\Phi)$$
is an infinite chain of roots whose mirrors make up an entire boundary component of $C$.
\end{proof}

We are now ready to discuss walking.  Walking is a method for extending a bounded chain of roots by finding the next root if it exists.  There is no built in way of detecting whether or not there is always a next root.  However, in all of our lattices, we were able to prove that there is always a next root before even if it is hard to find it.

Like Vinberg's algorithm, the walking algorithm does a search for new roots in discrete batches ordered by a quantity called \emph{height}.  The height of a potential new root is directly related to the hyperbolic distance from the root's mirror to a known corner of the chamber.

The inputs to Vinberg's algorithm are a quadratic form $Q$ for a lattice $L$, and a pair of simple roots $r_1$ and $r_2$ at one corner of the chamber.  Recall that a choice of future cone $\mathfrak{C}^+$ and orientation on $V$ determines a corner basis $\{r_1,r_2,p_1\}$ that is in the orientation on $V$.  As described in section \ref{rrer}, we can use $Q$ to get a list of possible norms for roots.  The height of a vector $v$ with respect to $p_1$ is
$$\hite(v) = \frac{(v\cdot p_1)^2}{v^2}$$
The algorithm searches for roots in batches ordered by increasing height with respect to $p_1$.  We call this kind of search a \emph{batch search}.  If we require the roots to satisfy certain additional inner product conditions, we call it a restricted batch search.  Walking involves two kinds of restricted batch searches.  Suppose we have a chain of roots $\{r_i\}_{i\in I}$ with $\max(I) = k+1$.
\begin{enumerate}
\item If we are looking for an $A_2$ corner, so that $r_k^2 = 2$, we seek roots of norm 2 whose inner product with $r_{k+1}$ is equal to $-1$.  We call this a batch search of type I.
\item If we are looking for a non-$A_2$ corner, we look for roots whose inner product with $r_{k+1}$ is 0. (In this case we will be able to apply Lemma \ref{simpleRoots} to get the next simple root.)  We call this a batch search of type II.
\end{enumerate} 

Walking lets us take advantage of the things we know about a partial chamber to deduce information about the next root, which we need in order to justify restricting a batch search.  The next lemma tells us precisely how we extend a bounded chain of roots when the next root exists.

\begin{lemma}\label{nextRoot}
Let $\Phi = \{r_1,\dots,r_{k+1}\}$ be a bounded chain of roots of $L$ all of whose mirrors bound a single copy of the chamber, and assume the next root of $\Phi$ exists.  For any pair of consecutive roots $r_i$ and $r_{i+1}$, let $\frac{\pi}{m_i}$ be the angle between the corresponding edges $R_i$ and $R_{i+1}$. Then one of the following is true:
\begin{enumerate}
\item $m_{k}$ is even and $r_{k+1}^2\neq 2u^2$ for any unit $u\in U(F)$.
\item $m_{k}$ is even and $r_{k+1}^2 = 2u^2$ for some unit $u\in U(F)$
\item $m_{k} = 3$
\end{enumerate}
In each case we are able to find the next root $r_{k+2}$ by a combination of translations and restricted batch searches.
\end{lemma}
\begin{proof}
In all 3 cases, we may apply Lemma \ref{nearest} to find the nearest translate of $r_k$ along $R_{k+1}$.  Let $\phi$ be a generator for $H_{r_{k+1}}$, chosen so that $\phi$ is a translation the positive direction as defined \ref{posTransDir} (here $i = 1$).

Suppose either (1) or (2) holds.  Let $R_{k}'$ be the line in $\Lambda^2$ associated to the root $r_{k}' = \phi(r_{k})$.  Since $\phi$ is an isometry, the angle between $R_{k+1}$ and $R_{k}'$ is $\frac{\pi}{m_{k}}$.  Since (1) or (2) is true, by Lemma \ref{simpleRoots}, there is a root $s$ such that ${s}^{\perp}\cap r_{k+1}^{\perp} = {r_{k}'}^{\perp}\cap r_{k+1}^{\perp}$, and $s$ is orthogonal to $r_{k+1}$.  Let $r_{k}'' = R_s(r_{k}')$.  The composition $R_s\circ \phi$ is a lattice isometry that preseves $\mathfrak{C}^+$ and reverses the orientation on $V$.  We have that
$$R_s\circ\phi(r_k''-r_k) = R_s\circ\phi(r_k'') - R_s(r_k') = R_s\circ\phi(r_k'')-r_k''$$
But $R_s\circ\phi(r_k'') = r_k$, so $R_s\circ\phi$ negates the subspace spanned by $r_k''-r_k$.  Thus it must be a a reflection and not a glide reflection.  Let $t$ be a primitive lattice vector in the span of $r_k''-r_k$.  Then $t$ is a root orthogonal to $r_{k+1}$.  Let $c$ be the corner formed by $t$ and $r_{k+1}$.

The root $t$ has the property that if there is a root whose reflecting plane intersects $R_{k+1}$ between $c_{k}$ and $c$, it makes an $A_2$ corner with $r_{k+1}$.  If this were not the case, then $\phi$ would not take $C$ to its nearest translate along $R_{k+1}$.  In other words, $t$ is the nearest root to $c_{k}$  whose mirror makes an angle of $\frac{\pi}{2}$ with $R_{k}$. Using Lemma \ref{simpleRoots} we can find the simple root $t'$ satisfying $t'\cap r_k = t\cap r_k$

If (1) holds, then $t'$ is the next root.

If (2) holds, then we do a restricted batch search of type I to see if there is a root $r_{k+2}$ of height less that $\hite(t')$ that makes an $A_2$ corner with $r_{k+1}$.  If our restricted batch search finds a root before passing the height of $t'$, then that is the next root.  Otherwise $t'$ is the next root.

If (3) holds, then we know that $m_{k+1}$ will be even by Lemma \ref{allA2}.  Let $r_{k+2}$ be the next root.  By Lemma \ref{simpleRoots}, there is a root $r_{k+1}'$ such that $r_{k+1}'^{\perp}\cap r_k^{\perp} = r_{k+1}^{\perp}\cap r_k^{\perp}$, and $r_{k+1}'$ is orthogonal to $r_k$.  A restricted batch search of type II will always find $r_{k+1}'$.  Then we can find $r_{k+1}$ by part (2) of Lemma \ref{simpleRoots}.

\end{proof}

Now we have all that we need to state the walking theorem.  We assume from here on out that all of our bounded chains of roots have a next root and a previous root.

\begin{theorem}[Walking] \label{walking} 
Let $\Phi = \{r_1,\ldots,r_{k+1}\}$ be a bounded chain of roots all of whose mirrors bound a single copy of the chamber $C$, and let $c_i$ be the corner formed by $r_i$ and $r_{i+1}$.  $\Phi$ can be extended by repeatedly adjoining the next root.  After adjoining a finite number of roots, either $\Phi$ will become a closed chain of roots, or else there will be an orientation preserving automorphism $\psi$ that takes a basis at some corner $c_i$ to a basis at the last corner $c_{k}$.  After applying all powers of $\psi$ to the roots in $\Phi$ and adjoining all these roots to $\Phi$, $\Phi$ completely describes a boundary component of $C$.
\end{theorem}

\begin{proof}
From a corner $c_i$, the only inputs needed to either find the nearest translate or do a restricted batch search are the roots $r_i$ and $r_{i+1}$.  Thus if we know at least one corner, we can always find the next root by Lemma \ref{nextRoot}.  By Lemma \ref{pigeonhole}, if $\Phi$ never becomes a closed chain then after adjoining some finite number of roots to $\Phi$ there will be an automorphism $\psi$ taking a basis at some lower corner $c_i$ to a basis at the highest corner $c_k$.   In fact, we will always have $i=1$, because if $i>1$, then since $\psi$ is an orientation preserving automorphism that preserves adjacency of roots in the chain, it takes $c_1$ to a corner lower than $c_k$, so we would have already found $\psi$ before getting to $c_k$.

If $\psi$ has finite order, then applying all powers of $\psi$ to the roots in $\Phi$ and adjoining their images to $\Phi$ makes $\Phi$ a closed chain.  Otherwise applying all powers of $\psi$ to the roots of $\Phi$ and adjoining their images to $\Phi$ makes $\Phi$ into a chain of roots whose mirrors are all part of a single boundary component of $C$ with infinitely many sides.

\end{proof}

\subsection{An example}
This example demonstrates how walking works in practice.  The norm-angle sequence for the polygon that we build in this example is 
$$(2+\sqrt{2})_8(6+4\sqrt{2})_3(6+4\sqrt{2})_2(50+35\sqrt{2})_2$$

\begin{figure}[ht]
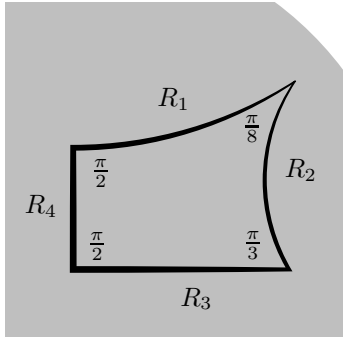


\begin{center}
\input figure10.tex
\end{center}

\caption{An 8,3,2,2 polygon}
\end{figure} 
Our code took less than a second to do these computations.  All the numbers involved in the computation are in Table \ref{allData}. In the table, elements $a+b\sqrt{2}\in \frako_F$ are represented as pairs $(a~b)$.  Vectors are represented as triples of pairs $(a_1~b_1 | a_2~b_2 | a_3~b_3)$.  Some vectors appear with their norms.  Elements $(a_1+b_1\sqrt{2})+(a_2+b_2\sqrt{2})\sqrt{D} \in \frako_K$ are represented as pairs of pairs $(a_1~b_1|a_2~b_2)$.

\begin{table}[ht]

\begin{center}
\begin{tabular}{|l|cc|cc|cc|cc|}
\hline
\multicolumn{9}{|c|}{matrices}\\
\hline
\multicolumn{9}{|c|}{\multirow{4}{*}{$Q = 
\pz{\begin{array}{cc|cc|cc}
60 & -5 & 20 & -20 & 10 & -10\\
20 & -20 & 18 & -12 & 9 & -6\\
10 & -10 & 9 & -6 & 6 & -4
\end{array}}$ }} \\
\multicolumn{9}{|c|}{}\\
\multicolumn{9}{|c|}{}\\
\multicolumn{9}{|c|}{}\\

\multicolumn{9}{|c|}{\multirow{4}{*}{$M_u = 
\pz{\begin{array}{cc|cc|cc}
-701 & -490 & 156 & 114 & 78 & 57\\
-8150 & -5770 & 1859 & 1310 & 929 & 655\\
0 & 0 & 0 & 0 & 1 & 0
\end{array}}$ }}\\
\multicolumn{9}{|c|}{}\\
\multicolumn{9}{|c|}{}\\
\multicolumn{9}{|c|}{}\\
\multicolumn{9}{|c|}{}\\
\hline

\multicolumn{7}{|c|}{vectors} & \multicolumn{2}{|c|}{norms}\\
\hline
$r_1$ & 1 & 0 & 0 & 1 & 3 & 2 & 2 & 1 \\
$r_2$ & 0 & 0 & 3 & 2 & -6 & -4 & 6 & 4 \\
$r_3$ & -22 & -16 & -100 & -70 & 3 & 2 & 6 & 4 \\
$r_4$ & -73 & -52 & -340 & -240 & 35 & 25 & 50 & 35 \\
$p_1$ & -11 & -8 & -50 & -35 & 0 & 0 & -120 & -85\\
\hline

\multicolumn{7}{|c|}{vectors} & \multicolumn{2}{|c|}{formulae}\\
\hline
$q$ & 2 & 0 & 3 & 4 & 0 & 0 & \multicolumn{2}{|c|}{$2\pi_{r_2^{\perp}}(r_1)$}\\
$q'$ & -22 & -14 & -243 & -174 & 0 & 0 & \multicolumn{2}{|c|}{$M_u(q)$}\\
$r_1'$ & 11 & 7 & 120 & 86 & 3 & 2 & \multicolumn{2}{|c|}{$M_u(r_1)$}\\
$t$ & -7 & -5 & -85 & -60 & 0 & 0 & \multicolumn{2}{|c|}{$\frac{1}{\sqrt{2}}R_q(r_1-r_1')$}\\
\hline

\multicolumn{9}{|c|}{numbers}\\

\hline
\multicolumn{9}{|c|}{$\begin{array}{|l|cc|cc|}
D & 240 & 170\\
u_0 & 1 & 1\\
\hline
u_1 & 1 & 1 & 0 & 0\\
u_2 & -19 & -13 & 1 &\frac{1}{2}\\
u & 579 & 410 & -26 & -19
\end{array}$} \\
\hline

\end{tabular}

\end{center}

\caption{Numbers involved in the computation for this example.}

\label{allData}
\end{table}

To start off, we know the quadratic form $Q$ and the roots $r_1$ and $r_2$ at the corner $c_1$ of the chamber $C$.  Their mirrors make an angle of $\frac{\pi}{8}$.  The vector $p_1\in\mathfrak{C}^+$ is a primitive lattice vector in $\lz{r_1,r_2}^{\perp}$ for which we compute coordinates.  Let $u_0=1+\sqrt{2}$ be a fundamental unit in $U(F)$.  Since $r_2^2 = 2u_0^2$, we are in situation (2) of the Lemma \ref{nextRoot}, so it is possible that the next corner will be an $A_2$.

As outlined in the proof of Lemma \ref{nextRoot} we will get a height bound for a type I restricted batch search.  To find the root $t$ whose mirror makes the next non-$A_2$ corner along $R_2$, we first find a generator for $H_{r_2}$.  Twice the projection of $r_1$ onto $r_2^{\perp}$ is given by 
$$q = 2\pz{r_1-\frac{r_1\cdot r_2}{r_2^2}r_2}$$
By Lemma \ref{translateIsom}
$$\lz{p,q}\cong K = \Q(\sqrt{2},\sqrt{D})$$
Fundamental units in $K$ are $u_1$ and $u_2$.  The group $G$ defined by \eqref{defG} is generated by $u = u_1^{-1}u_2^2$.  We let $M_u$ be the matrix for the translation corresponding to $u$.  It turns out that $M_u$ preserves $L$, so $H_{r_2}$ is nontrivial, and is in fact all of $G$.  

Since $G = \lz{u}$, we know that $M_u$ takes the chamber to its nearest translate in some direction.  The vector $M_u(q)$ is the projection of $M_u(r)$ onto $r_2^{\perp}$.  Let
$$r_1' = R_{M_u(q)}(M_u(r_1))$$
The root $t$ is a primitive lattice vector in the span of $r_1'-r_1$ whose inner product with $p$ is positive.  We compute $t$, and then we check that $\{r_1,r_2,r_3 = t\}$ is a chain of roots.  It is not, since $t$ has positive inner product with $r_1$.  Thus we replace $t$ with $R_q(t)$, which is what we would have gotten by translating in the opposite direction.  (This just means we chose the wrong generator for $G$.)  

Since $t$ and $r_2$ are simple roots at their corner, we know that $t$ is the nearest root whose mirror makes a non-$A_2$ corner with $R_2$, and so if it is not the next root of $\Phi$, the next root has height less than $\hite(t)$.

\begin{figure}[ht]
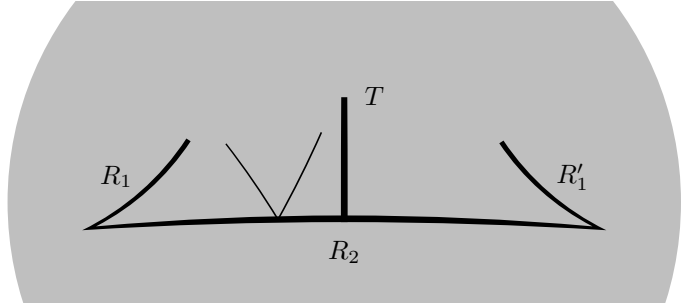


\begin{center}
\input figure8.tex
\end{center}

\caption{$R_1,R_2,R_1',R_1'',T$, and the still hypothetical $\frac{\pi}{3}$ corner.}
\end{figure}

The height of the next root is less than or equal to the height of $t$.  The height of $t$ is
$$\hite(t) = \frac{1}{{t}^2}\left(\frac{t\cdot p_1 }{p_1^2}\right)^2 \approx 2.41$$
We do a restricted batch search looking only for vectors of norm 2 whose inner product with $r_2$ is $-\frac{1}{2}$.  We find one such vector in the batch of height $\approx 0.086$ which extends $\Phi$.  We call it $r_3$, and adjoin it to $\Phi$.

By Lemma \ref{allA2}, we know that the next root $r_4$ makes a non-$A_2$ corner $r_3$.  We find it by doing a batch search for roots of any possible norm that satisfy 
$$r_3\cdot r_4 = 0$$
We know that the height will be less than the height of a $\frac{\pi}{3}$ corner further along $R_3$, which we find by taking the nearest translate.  We also know this bound is not very good (Figure \ref{badBound}), since there are several corners between an $A_2$ corner and its nearest translate.  We use this bound for time estimates only.
\begin{figure}[ht]
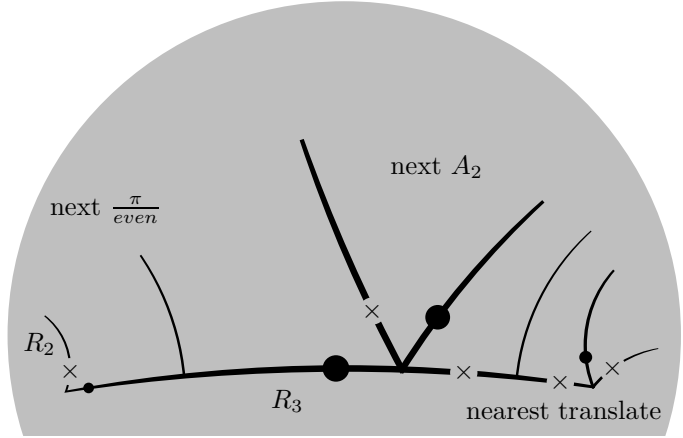


\begin{center}
\input figure9.tex
\end{center}

\caption{The bound coming from taking the nearest translate of an $A_2$ corner is not very good.}
\label{badBound}
\end{figure}
As a measure of just how bad this bound is, it tells us that we are looking for a root of height less than $12578.6$.  The actual next root that we find via batch search has height $\approx 0.059$.

At this point we check whether $R_1$ and $R_4$ intersect inside hyperbolic space by checking that $V_1\cap V_2$ is negative definite.  It is, and the angle between them is $\frac{\pi}{2}$.  Thus $\{r_1,r_2,r_3,r_4\}$ is a closed chain of roots for the fundamental chamber of a reflective lattice.

\subsection{Algorithmic descriptions}

Let $\Phi = \{r_i\}_{1 \leq i\leq k+1}$ be a chain of roots all of whose mirrors bound a single chamber $C$.   The line containing the edge of $C$ corresponding to $r_i$ is $R_i$.  The corner of $C$ formed by $r_i$ and $r_{i+1}$ is $c_i$.  The primitive lattice vector lying along the corner $c_i$ is $p_i$.  The point in $\Lambda^2$ corresponding to $p_i$ is $P_i$.  The angle at the corner $c_i$ is $\frac{\pi}{m_i}$. (Figure \ref{setup}).
\begin{figure}[ht]
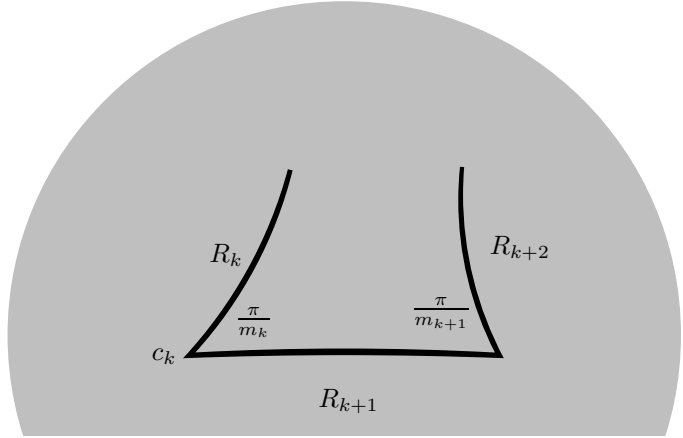


\begin{center}
\input figure11.tex
\end{center}

\caption{The next edge of $\Phi$ is $R_{k+2}$. }
\label{setup}
\end{figure}

Figure \ref{3cases} shows 3 examples of the possibilities outlined in Lemma \ref{nextRoot}.   On the left, $m_k$ is even and $r_{k+1}^2\neq 2u^2$, so $m_{k+1}$ must be even.  In the middle, $m_k$ is even and $r_{k+1}^2= 2u^2$, so $m_{k+1}$ could be 3, and there is a $\frac{\pi}{even}$ further along $R_{k+1}$.  On the right $m_k = 3$ so $m_{k+1}$ must be even.

\begin{figure}[ht]
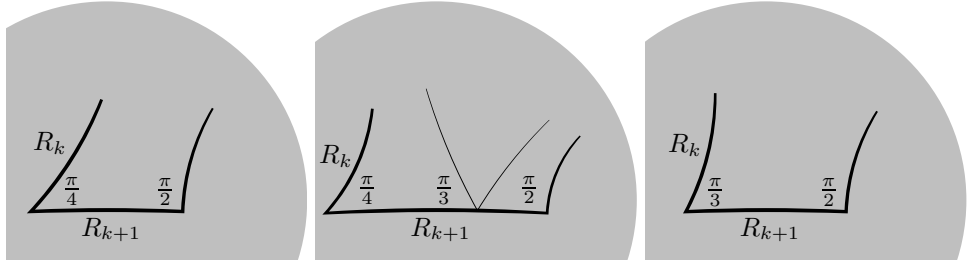


\begin{center}
\input figure12-1.tex
\input figure12-2.tex
\input figure12-3.tex
\end{center}

\caption{Three cases}
\label{3cases}
\end{figure}

\begin{algorithm}\label{nearestTranslateAlg}
This algorithm finds the positive generator for $H_{r_{k+1}}$.  Let 
$$q = 2\pz{r_{k}-\frac{r_k\cdot r_{k+1}}{r_{k+1}^2}r_{k+1}}$$
 be twice the projection of $r_k$ onto $V_{k+1}$.  Let $D = -q^2p^2$, $K = F(\sqrt{D})$.  Let $\varphi$ be the isometry from Lemma \ref{translateIsom}.
\begin{enumerate}
\item Recall from the proof of Lemma \ref{nearest} that $U(K)$ has two generators.  Call them $u_1$ and $u_2$.
\item The map $N_{K/F}:U(K)\rightarrow U(K)$ is given by a matrix which we can write in terms of the basis $u_1$, $u_2$, for the rank $2$ free subgroup of $U(K)$.  Compute the kernel of that matrix to get a generator $u$ for the free subroup of $U(K)$ consisting of units of norm $1$.  Let $u = a+b\sqrt{d}$.  If $a <0$, replace $u$ by $-u$.
\item Let $p' = \frac{1}{q^2}\varphi^{-1}(u\sqrt{D})$, $q' = \varphi^{-1}(u)$.  Let $M_u$ be the matrix defined by 
$$(p,q,r_{k+1})\mapsto (p',q',r_{k+1})$$
\item Check whether $M_u$ preserves $L$ by writing it in terms of a basis for $L$ and checking whether all its entries are in $\frako$.  Since $M_u$ has determinant $1$, we know that if it preserves $L$ its inverse will also preserve $L$ and so it is an automorphism.  Let $j$ be the smallest natural number such that $M_u^j$ preserves $L$.  Since we assumed $\Phi$ has a next root, $j > 0$ exists.  Let $\phi = M_u^j$.
\item Using Lemma \ref{nearest}, we know that $\phi$ is a generator for $H_{r_{k+1}}$.  The final step is to check whether it is the positive or negative generator with repect to the chain of roots $\Phi$.  Let
$$r_k' = R_{\phi(q)}(\phi(r_{k}))$$
if the intersection
$$V_k^-\cap V_{k'}\cap V_{k+1}$$
is nonempty then $\phi$ is a positive generator.  Otherwise $\phi$ was a negative generator, so $\phi^{-1}$ is a positive generator.
\end{enumerate}
\end{algorithm}

For the next algorithm, we suppose $m_k$ is even.

\begin{algorithm}\label{nextEvenAlg}
This algorithm tells us how to find the next root $t$ whose mirror $T$ intersects $R_{k+1}$ at a non-$A_2$ corner.  The root $t$ may or may not be the next root of $\Phi$.
\begin{enumerate}
\item Follow Algorithm \ref{nearest} to find $r_{k}'$ as defined in step (5).
\item Let $t' = \gamma(r_{k}'-r_{k})$ where $\gamma$ is a scalar that makes $t'$ a primitive lattice vector.
\item Let $t$ be a simple root at the corner formed by $t'$ and $s$ (Lemma \ref{simpleRoots} already gives an algorithmic description of how to find this).
\end{enumerate}
\end{algorithm}

The next algorithm is a full description of walking.

\begin{algorithm} \label{walkingAlg}
This algorithm tells us how to extend $\Phi$ to a chain of roots whose mirrors make up an entire boundary component of $C$.  The 3 cases we refer to are the ones from Lemma \ref{nextRoot}.
\begin{enumerate}
\item In cases (1) and (2), use find the root $t$ defined in Algorithm \ref{nextEvenAlg}.

\item In case (1), $r_{k+2} = t$, proceed to step (5).

\item In case (2), compute $h$, the height of $t$.  Do a restricted batch search of type I for roots of norm $2$ whose mirrors intersect $R_{k+1}$ in an angle of $\frac{\pi}{3}$.  If a root is found this way, it is the next root $r_{k+2}$.  If the height $h$ is passed and no root found, then $r_{k+1}=t$.  Proceed to step (5).

\item In case (3), do a restricted batch search search for roots whose mirrors intersect $R_{k+1}$ in an angle of $\frac{\pi}{2}$.  When this search finds a root $t$, use Lemma \ref{simpleRoots} to find the simple root at that corner. This is the next root $r_{k+2}$.

\item Let $\Phi' = \Phi\cup\{r_{k+2}\}$.  Check whether $\Phi'$ is a closed chain by checking whether $V_1\cap V_{r_{k+2}}$ is negative definite.  If $\Phi'$ is a closed chain then we are finished.

\item If $\Phi'$ is not a closed chain, we check to see whether the linear transformation defined by
$$\psi:(r_1,r_2,p_1)\mapsto (r_{i+1},r_{i+2},p_{i+1})$$
is a lattice automorphism.  If it is, then we are done.  If it is not, then repeat from step (1) with $\Phi = \Phi'$.
\end{enumerate}
\end{algorithm}

If we exit Algorithm \ref{walkingAlg} from step (5), then $L$ is a reflective lattice.  If we exit the algorithm from step (6), we have lattice automorphism $\psi$ which may have finite or infinite order.  If $\psi$ has finite order, then
\begin{equation}\label{union}
\bigcup_{j\in Z}\psi^j(\Phi)
\end{equation}
is a finite closed chain of roots, so $L$ is reflective.  If $\psi$ has infinite order, then \eqref{union} is an infinite chain of roots whose mirrors are an entire boundary component of $C$, and we therefore know $L$ is not reflective.

In all but 48 of the 2381 squarefree lattices on our list, these algorithms were enough to determine reflectivity in under two days (in fact, most of them took under a minute, and all but two of them took under a day).  In the 48 remaining cases, the algorithm had reached a corner from which, according to Algorithm \ref{walkingAlg}, the next corner would have to be found using a restricted batch search (either case (2) or (3) of Lemma \ref{nextRoot}).  The height bound given by translation along the side was so big that the estimated time it would take to do a batch search up to that height was absurdly long (in the worst case it was $10^{44}$ days, in the best it was about 100 days).

\subsection{The last 48 cases}

The remaining 48 cases fall into two categories, which we will call type I and type II.  The chains of roots of type I are those with a known $A_2$ corner.  The chains of roots of type II are those with no known $A_2$ corner.

\subsubsection{Partial simple systems of type I}

We consider a chain of roots $\Phi = \{r_1,\ldots r_k\}$ whose edges bound a chamber with one $A_2$ corner and several non-$A_2$ corners.  An example of such a chamber appears in Figure \ref{typeI}.  We may assume that the $A_2$ corner is the first one, formed by $r_1$ and $r_2$.  We say that $\Phi$ has type I if $k\geq 3$, and $r_{k}^2 = 2u^2$ for some unit $u\in U(F)$.

\begin{figure}[ht]
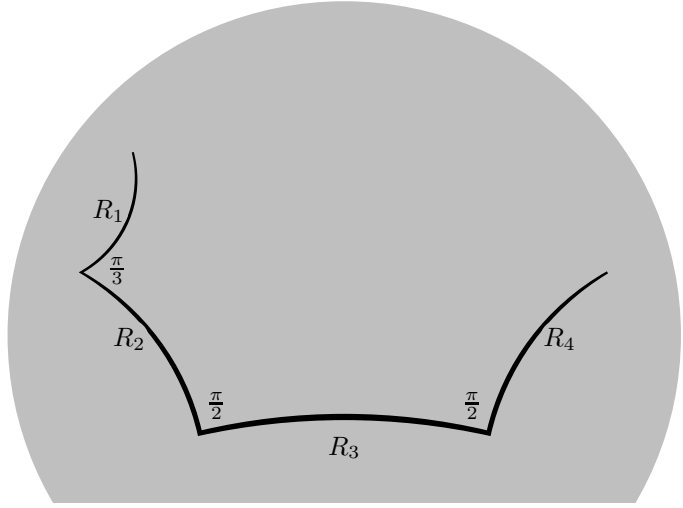


\begin{center}
\input figure13.tex
\end{center}

\caption{An $A_2$ corner followed by a bunch of non $A_2$'s}
\label{typeI}
\end{figure}

If we were walking around the polygon, we would first use Algorithm \ref{nextEvenAlg} to find the nearest root in the positive direction along $R_k$ whose mirror makes a non-$A_2$ corner with $R_k$.  We would then do a restricted batch search of type I to determine if $t$ is the next root, or if the next root is between $r_{k-1}$ and $t$ and makes an $A_2$ corner with $r_k$.  

Suppose that if the batch search were allowed to run for as much time as needed, it eventually would find that the next root $r_{k+1}$ makes an $A_2$ corner with $r_k$.  Then by Lemma \ref{A2Aut} the linear transformation $\psi$ defined by 
$$\psi:(r_1,r_2,p_1)\mapsto (r_k,r_{k+1},p_k)$$
is an automorphism of $L$.  The strategy here is to produce $\psi$ without first knowing $r_{k+1}$ and $p_k$.  Let $p$ be the primitive lattice vector that lies along the corner formed by the roots $t$ and $r_k$.  

\begin{lemma}\label{factor}
The automorphism $\psi$ of $L$ factors as the product of an order 3 rotation fixing $p_1$, and an automorphism of $L$ that takes the corner basis $\{r_2,r_3,p_2\}$ to the corner basis $\{r_k,t,p\}$.
\end{lemma}
\begin{proof}
First we will write down the order 3 rotation fixing $p_1$ as a product of 2 reflections.  Let $r_{2}' = R_{r_1}(r_{2})$.  The composition
$$\rho = R_{r_{2}}\circ R_{r_{2}'}$$
is a rotation by $\frac{2\pi}{3}$ fixing the vector $p_1$.

Let $\tau = \psi\circ\rho^{-1}$. Then $\tau$ is a lattice automorphism preserving a chamber for a sublattice of $L$ that takes the basis $(r_2,r_3,p_2)$ to $(r_k,t,p)$.  The desired factorization of $\psi$ is 
$$\psi = \tau\circ\rho$$
\end{proof}

\begin{figure}[ht]
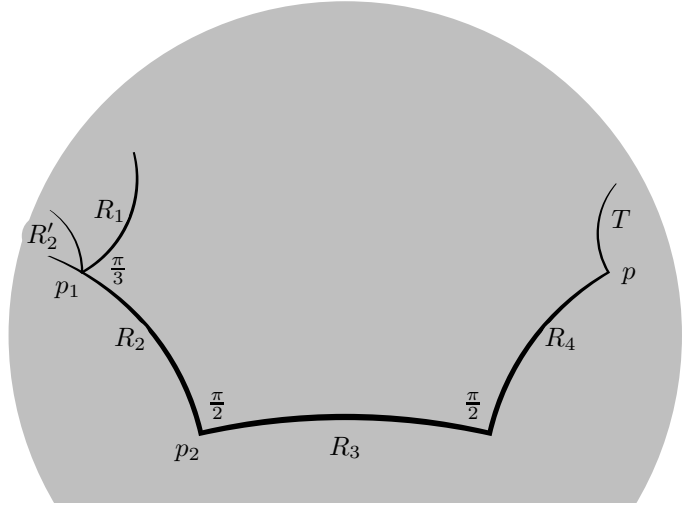


\begin{center}
\input figure13-p.tex
\end{center}

\caption{The polygon from Figure \ref{typeI}, now showing $R_2'$ and $T$ as well.}
\label{typeIAlg}
\end{figure}

In our application of Lemma \ref{factor}, we will not know $\psi$.  But if it exists, we will be able to find both $\tau$ and $\rho$.  Indeed, the next corner of $\Phi$ is an $A_2$ if and only if $\tau$ from this factorization exists, for if it does not then there can be no $\psi$.  The following algorithm describes what we do with chains of roots of type I.
\begin{algorithm} Let $\Phi = \{r_1,\ldots,r_k\}$ be a chain of roots of type I.  Let $\rho$ be the order $3$ rotation defined in Lemma \ref{factor}.
\begin{enumerate}
\item Follow Algorithm \ref{nextEvenAlg} to find the next root $t$ that makes a non-$A_2$ corner with $R_k$. Note that this is the same way we would begin if we were going to attempt to find $r_{k+1}$ by a restricted batch search, and $t$ would be the root giving the height bound.  Let $p$ be a primitive lattice vector such that $\{r_k,t,p\}$ is a corner basis.

\item Check for a lattice automorphism $\tau$ taking $(r_2,r_3,p_2)$ to $(r_k,t,p)$.  If there is no automorphism, then there can be no $A_2$ corner involving $r_k$.  Then the next root $r_{k+1}$ is $t$.  Let $\Phi' = \Phi\cup\{r_{k+1}\}$.  If $\Phi'$ is a closed chain, then we are done.  If $\Phi'$ is not a closed chain, we continue with walking (Algorithm \ref{walking}).

\item If in step (2), we do find an automorphism $\tau$, then let $\psi = \tau\circ \rho$.  Then $r_{k+1} = \psi(r_2)$ is the next root, and we are done.

\end{enumerate}
\end{algorithm}

If we exit the algorithm at step (3), we are finished, because \eqref{union} describes an entire boundary component of the chamber.  In all of our type I lattices, the algorithm exited at step (3) with an infinite order automorphism $\psi$, so they were all non-reflective.

\subsubsection{Partial simple systems of type II}

Now consider a chain of roots $\Phi = \{r_1,\ldots r_k\}$ that bound a chamber with no $A_2$ corners such that $k >3$ and $r_1^2 = 2u_1^2$, $r_k^2 = 2u_k^2$ for some units $u_1,u_k\in U(F)$.  If $\Phi$ satisfies all of these properties we will call $\Phi$ a chain of roots of type II.  An example of such a chamber with $k=4$ is shown in Figure \ref{typeII}.  

With the type I chains, we had a known $A_2$ corner, and were therefore able to write down an order 3 rotation about that corner.  Here we don't have that, but we will still be able to use Lemma \ref{factor}.  We will show that whether or not there is an $A_2$ corner, the chamber has a symmetry of infinite order.

There is an isomorphism between $\text{Isom}(\Lambda^2)$ and $SO(2,1)^+$.  In one direction, it is given by restricting the action of $SO(2,1)$ on $\R^{2,1}$ to the positive cone, and looking at the action on the quotient $\Lambda^2$.  In the other direction, it is given by the adjoint representation of $PSL_2\R \cong\text{Isom}(\Lambda^2)$ acting on the lie algebra $\mathfrak{sl}_2\R$ with the killing form.  The Lie algebra $\mathfrak{sl}_2\R$ has dimension $3$ and the killing form has signature $(2,1)$.  If $\theta\in SL_2\R$ has trace $\pm a$, then the corresponding element of $\text{ad}(PSL_2\R)$ has trace $|a|+1$.

We may think of elements of $\Aut(L)$ as isometries of $\Lambda^2$.  The isometry group $\text{Isom}(\Lambda^2)\cong PSL_2\R$ contains three types of orientation preserving isometries, called elliptic, parabolic, and hyperbolic.  These 3 types are characterized by the absolute values of their traces as elements of $SL_2\R$.  Let $\theta\in SL_2\R$, and let 
$$\pi:SL_2\R\rightarrow PSL_2\R$$
 be the standard projection.
\begin{enumerate}
\item If $|\tr(\theta)| < 2$, then $\pi(\theta)$ is elliptic and $\theta$ has a fixed point in the interior of $\Lambda^2$.
\item If $|\tr(\theta)| = 2$, then $\pi(\theta)$ is parabolic and $\theta$ has no fixed points in the interior of $\Lambda^2$ but a single fixed point on $\partial\Lambda^2$.
\item If $|\tr(\theta)| > 2$, then $\pi(\theta)$ is hyperbolic, and $\theta$ stabilizes a line in $\Lambda^2$ that joins two points of $\partial\Lambda^2$ that are fixed by $\theta$.
\end{enumerate}
Since $\Aut(L)$ is discrete, any elliptic element of $\Aut(L)$ must have finite order.  This means that if $\psi\in\Aut(L)$ is an orientation preserving isometry of infinite order, it must be a parabolic or hyperbolic element of $\text{Isom}(\Lambda^2)$.  Parabolic and hyperbolic elements of $PSL_2\R$ come from elements of $SL_2\R$ whose trace in absolute value is greater then or equal to $2$, so as matrices acting on $R^{n,1}$ they have trace greater than or equal to $3$.

\begin{figure}[ht]
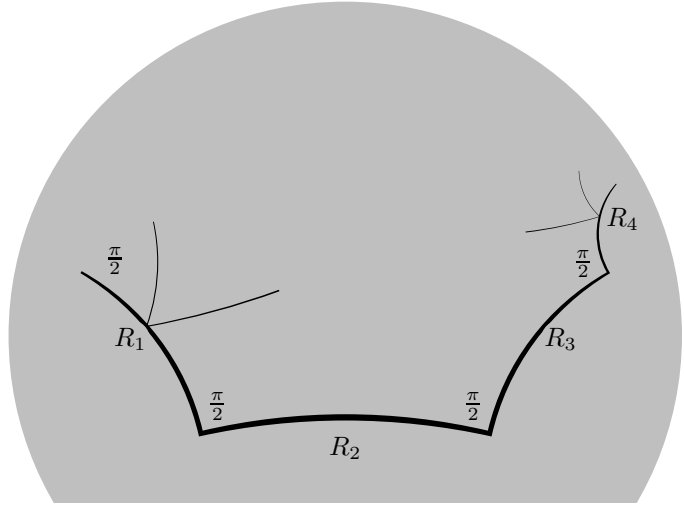


\begin{center}
\input figure14.tex
\end{center}

\caption{A sequence of non-$A_2$ corners, on either end the potential for an $A_2$ corner}
\label{typeII}
\end{figure}

Let $t$ be the next non-$A_2$ root along $R_k$, and let $q$ be the primitive lattice vector such that $\{r_k,t,q\}$ is the corner basis at the corner formed by $t$ and $r_k$. If both the previous corner and the next corner of $\Phi$ are $A_2$'s, then there is a lattice automorphism $\psi$ that takes the corner along one to the other.  In this scenario let $r_{k+1}$ be the next root, $r_0$ the previous root, and $\psi$ be given by
$$\psi:(r_k,r_{k+1},p_k)\mapsto (r_0,r_1,p_0)$$
By Lemma \ref{factor}, $\psi$ factors as the product of an order 3 rotation $\rho$ fixing $p_k$ and the lattice automorphism $\tau$ defined by
$$\tau:(r_1,r_2,p_1)\mapsto(r_k,t,q)$$

Introduce a parameter $x\in[0,1]$, and let $v(x) = (1-x)p_{k-1}+xq$.  Then we have
$$v(0) = p_{k-1}, v(1)=q$$
 and the image of $v(x)$ in $\Lambda^2$ is a point on the segment of $R_k$ joining the two corners.  Let $w(x)$ be the projection of $p_{k-1}-q$ onto $v^{\perp}$.  If $w(x)\cdot r_k >0$, we replace $w(x)$ by $-w(x)$, so that our directional conventions work out.  When coding this algorithm, we want to be precise and avoid rounding.  Thus instead of computing the square root, we introduce a function 
$$z(x) = \sqrt{\frac{6}{w(x)^2}}$$
Then $z(x)w(x)$ is a vector of norm 6 orthogonal to both $v(x)$ and $s$.

If there is an $A_2$ corner somewhere along $R_k$, then for some $x\in(0,1)$, $v(x)$ points along that corner, and order 3 rotation fixing $v(x)$ is a lattice automorphism.  Recall that $u_k$ is a unit such that $(u_k^{-1} r_k)^2 = 2$.  Order $3$ rotation fixing $v(x)$ is a function of $x$ and is given by
$$\rho(x):\left(u_k^{-1} r_k,z(x)w(x),v\right)\mapsto \pz{\frac{z(x)w(x)-u_k^{-1} r_k}{2},\frac{-z(x)w(x)-3u_k{-1} r_k}{2},v(x)}$$
The chamber symmetry $\psi$ taking the $A_2$ corner along $R_k$ to the one along $R_1$ would then be the composition
$$\psi = \tau^{-1}\circ\rho(x)$$
The trace of $\psi$ is then a function of $x$, and $\psi$ is parabolic or hyperbolic if and only if $\tr(\psi)(x) \geq 3$.  There are two possibilities.  Either the chamber has an $A_2$ corner, and $\psi(x)$ is a lattice automorphism for some $x\in (0,1)$, or else there are no $A_2$ corners and $\tau$ is a lattice automorphism.  If we establish that $\tau$ has infinite order and $\psi(x)$ has infinite order for all $x\in(0,1)$, then the chamber has an infinite order automorphism whether or not it has an $A_2$ corner.  Thus, to show that the chamber has infinitely many sides, it suffices to show that $\tau$ has infinite order, and $\tr(\psi(x)) > 3$ for all $x\in [0,1]$.

For each of the type II chains of roots, we computed $\tr(\psi(x))-3$ and $z(x)^2$ and found that they always had the same form.
\begin{equation} \label{system}
\begin{array}{rl}
\tr(\psi(x))-3 &= \frac{f_1(x)z(x)^2 + f_2(x)z(x) + f_3(x)}{f_4(x)z(x)}\\
z(x)^2 &= f_5(x)
\end{array}
\end{equation}
where $f_1,f_2,f_3,f_4,\text{ and }f_5$ are polynomials with coefficients in $\Q(\sqrt{2})$, $f_1$ has degree 1, $f_2,f_4,f_5$ have degree 2, and $f_3$ has degree 3.  

Sturm's theorem on polynomials in one variable gives an algorithm for determining how many real roots a polynomial has on a given closed interval.  A good explanation of Sturm's theorem can be found in Bartlett's notes \cite{bartlett2013finding}.  We apply it here to show that certain polynomials have no roots.  Once we know that a polynomial has no roots in $[0,1]$, we can say whether it is strictly positive or strictly negative by evaluating it at any point in the interval.  

Our goal is to show that $\tr(\psi(z,x)) >3$ for all $x,\in (0,1)$.  We check that the \eqref{system} is defined for all $x\in(0,1)$ by checking that both $f_5(x)$ and $f_4(x)$ have no roots in $[0,1]$.  Then we check and find that $f_5(x) > 0$ for all $x\in (0,1)$, so that $z(x)\in\R$.  If these conditions are met, then $\tr(\psi(x))-3$ has no solution in $(0,1)$ if and only if the numerator 
$$f_1(x)z(x)^2+f_2(x)z(x)+f_3(x) $$
has no zeros in the interval $(0,1)$.  We have: 
\begin{align*}
f_1(x)z^2+f_2(x)z+f_3(x) &= 0
\\ \Leftrightarrow~~~~~~~~~~~~~&\\
(f_1(x)f_5(x) +f_3(x))  + f_2(x)z &= 0 \\
\end{align*}
Let $h(x) = f_1(x)f_5(x)+f_3(x)$.  If $h(x)$ and $f_2(x)$ are both either strictly positive or strictly negative on the interval $(0,1)$, then since $z(x) >0$ we conclude that 
$$h(x) + f_2(x)z(x) = 0$$
has no solution with $x\in (0,1)$.  Therefore $\tr(\psi(x)) >3$ for all $x\in (0,1)$  

Here is an algorithmic description of what we do with type II chains of roots.
\begin{algorithm} \label{typeIIAlg}
Let $\Phi = \{r_1,\ldots,r_k\}$ be a chain of roots of type II.
\begin{enumerate}
\item Initialize an list $I = \{1\}$.  We will keep track in $I$ of the indices of all the edges that might span multiple chambers.

\item Find the the root $t$ that makes the next non-$A_2$ corner along $R_t$.  Let $q$ be a primitive vector pointing along the corner formed by $r_k$ and $t$, so that $\{r_k,t,q\}$ is a corner basis. For each $i\in I$, check whether the linear transformation $\tau_{ik}$ defined by
$$\tau_{ik}: (r_i,r_{i+1},p_i)\mapsto (r_k,t,q)$$
is a lattice automorphism.

\item If $\tau_{ik}$ is not an automorphism of $L$, then we store $k$ in the list $I$, and let $\Phi' = \{r_1,\ldots,r_k,r_{k+1} = t\}$. Note that at this point, the roots in $\Phi'$ may not all be part of a system of simple roots.  We continue to extend $\Phi'$ by walking as in Algorithm \ref{walkingAlg} until it once again true that the highest root $r_k$ satisfies $r_k^2 = 2u_k^2$ for some unit $u_k$, and then we return to step (1).

\item If $\tau_{ik}$ is an automorphism of $L$, we introduce the parameters $x$ and $z$, and compute the polynomials $f_1,f_2,f_3,f_4,f_5$ as defined above.  If all of the following hold for $x\in(0,1)$, then $C$ has an infinite order symmetry.
\begin{enumerate}
\item $\tau_{ik}$ has infinite order
\item $f_5(x) >0$
\item $|f_4(x)| >0$
\item $f_2(x)$ and $(f_1f_5+f_3)(x)$ both have no roots, and have the same sign. 
\end{enumerate}
\end{enumerate}
\end{algorithm}

For all of our lattices of type II, conditions (a)-(d) in step (4) all hold, so none of these lattices are reflective.  This finishes the proof of Theorem \ref{main}.

\appendix
\section{Description of the table}
The entries in the table each describe an $\frako$-lattice $L$.  A generator $\delta$ for the determinant ideal is given as an element of $\frako$, with $w=\sqrt{2}$.  We also give the norm $N_{F/\Q}(\delta)$ because those are easier to compare.  

Below that is a sequence of numbers, either undecorated or inside of a $(\cdot)^2$.  This is the sequence of angles between the closed chain of edges for the polygon.  For example, $842$ describes a triangle with angles $\frac{\pi}{8}, \frac{\pi}{4}, \frac{\pi}{2}$, and $(842)^2$ describes a hexagon with an order 2 rotations whose angles are $\frac{\pi}{8}, \frac{\pi}{4}, \frac{\pi}{2}, \frac{\pi}{8}, \frac{\pi}{4}, \frac{\pi}{2}$.  

Below the angle sequence are the quadratic form for $L$ and a list of roots along with their norms.  The quadratic form and the roots are both written with respect to the same basis for $L$.  The entries in the matrices and vectors are pairs $(a ~ b)$, standing for elements $a+b\sqrt{2}\in \frako$.

\include{appendix-table}

\bibliography{sources}{}
\bibliographystyle{plain}

\end{document}

%% file: figure3-1.tex
\begingroup
\psset{unit=80pt}
\begin{pspicture*}(-1.0000,-0.5000)(1.0000,0.8000)
\pscircle*[linecolor=lightgray](0.0000,0.0000){1.0000}
\pscustom[fillstyle=solid,fillcolor=black,linestyle=none]{%
\psarcn(-1.0431,0.0000){0.2773}{0.8545}{-0.8545}
\psarc(0.0000,49.9967){50.0067}{-90.8775}{-90.5295}
\psarc(-1.3130,0.0000){0.8509}{-0.5295}{0.5295}
\psarc(0.0000,-49.9967){50.0067}{90.5295}{90.8775}
\closepath
}%
\pscustom[fillstyle=solid,fillcolor=black,linestyle=none]{%
\psarcn(-1.0431,0.0000){0.2773}{0.8545}{-0.8545}
\psarc(0.0000,49.9967){50.0067}{-90.8775}{-90.0000}
\psline(0.0000,-0.0100)(0.0000,0.0100)
\psarc(0.0000,-49.9967){50.0067}{90.0000}{90.8775}
\closepath
}%
\pscustom[fillstyle=solid,fillcolor=black,linestyle=none]{%
\psarcn(-1.3593,-1.5258){1.7622}{59.7960}{59.2870}
\psarc(-1.2877,0.4818){0.9638}{-30.7360}{11.4290}
\psarcn(-0.4016,0.9623){0.2953}{-78.5710}{-80.2310}
\psarcn(-1.3394,0.5011){1.0025}{9.7690}{-30.1811}
\closepath
}%
\pscustom[fillstyle=solid,fillcolor=black,linestyle=none]{%
\psarc(-1.2698,1.4253){1.6461}{-60.7335}{-60.1836}
\psarc(-1.2877,-0.4818){0.9638}{29.7935}{49.7403}
\psarcn(-0.9454,0.4913){0.3676}{-40.2597}{-41.7907}
\psarcn(-1.3394,-0.5011){1.0025}{48.2093}{29.2894}
\closepath
}%
\pscustom[fillstyle=solid,fillcolor=black,linestyle=none]{%
\psarcn(-1.3593,1.5258){1.7622}{-59.2870}{-59.7960}
\psarcn(-1.3394,-0.5011){1.0025}{30.1811}{6.9891}
\psarcn(-0.2102,-1.4738){1.1029}{96.9891}{96.2168}
\psarc(-1.2877,-0.4818){0.9638}{6.2168}{30.7360}
\closepath
}%
\pscustom[fillstyle=solid,fillcolor=black,linestyle=none]{%
\psarc(-1.2698,-1.4253){1.6461}{60.1836}{60.7335}
\psarcn(-1.3394,0.5011){1.0025}{-29.2894}{-48.2093}
\psarcn(-0.9454,-0.4913){0.3676}{41.7907}{40.2597}
\psarc(-1.2877,0.4818){0.9638}{-49.7403}{-29.7935}
\closepath
}%
\pscustom[fillstyle=solid,fillcolor=black,linestyle=none]{%
\psarcn(-1.3358,0.0000){0.8658}{0.5156}{-0.5156}
\psarc(0.0000,49.9967){50.0067}{-90.5385}{-89.4705}
\psarcn(1.3130,0.0000){0.8509}{-179.4705}{179.4705}
\psarc(0.0000,-49.9967){50.0067}{89.4705}{90.5385}
\closepath
}%
\pscustom[fillstyle=solid,fillcolor=black,linestyle=none]{%
\psarcn(-1.3358,0.0000){0.8658}{0.5156}{-0.5156}
\psarc(0.0000,49.9967){50.0067}{-90.5385}{-89.1273}
\psarcn(1.0373,0.0000){0.2757}{-179.1273}{179.1273}
\psarc(0.0000,-49.9968){50.0068}{89.1273}{90.5385}
\closepath
}%
\pscustom[fillstyle=solid,fillcolor=black,linestyle=none]{%
\psarc(1.2698,1.4253){1.6461}{-119.8164}{-119.2665}
\psarcn(1.3394,-0.5011){1.0025}{150.7106}{131.7907}
\psarcn(0.9454,0.4913){0.3676}{-138.2093}{-139.7403}
\psarc(1.2877,-0.4818){0.9638}{130.2597}{150.2065}
\closepath
}%
\pscustom[fillstyle=solid,fillcolor=black,linestyle=none]{%
\psarcn(1.3593,-1.5258){1.7622}{120.7130}{120.2040}
\psarcn(1.3394,0.5011){1.0025}{-149.8189}{170.2310}
\psarcn(0.4016,0.9623){0.2953}{-99.7690}{-101.4290}
\psarc(1.2877,0.4818){0.9638}{168.5710}{-149.2640}
\closepath
}%
\pscustom[fillstyle=solid,fillcolor=black,linestyle=none]{%
\psarc(1.2698,-1.4253){1.6461}{119.2665}{119.8164}
\psarc(1.2877,0.4818){0.9638}{-150.2065}{-130.2597}
\psarcn(0.9454,-0.4913){0.3676}{139.7403}{138.2093}
\psarcn(1.3394,0.5011){1.0025}{-131.7907}{-150.7106}
\closepath
}%
\pscustom[fillstyle=solid,fillcolor=black,linestyle=none]{%
\psarcn(1.3593,1.5258){1.7622}{-120.2040}{-120.7130}
\psarc(1.2877,-0.4818){0.9638}{149.2640}{173.7832}
\psarcn(0.2102,-1.4738){1.1029}{83.7832}{83.0109}
\psarcn(1.3394,-0.5011){1.0025}{173.0109}{149.8189}
\closepath
}%
\rput{0.0000}(-0.2449,0.0000){\pscirclebox*[fillcolor=lightgray,framesep=0pt]{$\times$}}
\rput{0.0000}(-0.5670,0.1481){\pscirclebox*[fillcolor=lightgray,framesep=0pt]{$\times$}}
\rput{0.0000}(-0.5670,-0.1481){\pscirclebox*[fillcolor=lightgray,framesep=0pt]{$\times$}}
\pscircle*(-0.3788,0.1852){0.0329}
\pscircle*(-0.6345,0.0000){0.0239}
\pscircle*(-0.3788,-0.1852){0.0329}
\rput{0.0000}(0.6351,0.0000){\pscirclebox*[fillcolor=lightgray,framesep=0pt]{$\times$}}
\rput{0.0000}(0.3793,0.1854){\pscirclebox*[fillcolor=lightgray,framesep=0pt]{$\times$}}
\rput{0.0000}(0.3793,-0.1854){\pscirclebox*[fillcolor=lightgray,framesep=0pt]{$\times$}}
\pscircle*(0.5664,0.1480){0.0263}
\pscircle*(0.2446,0.0000){0.0376}
\pscircle*(0.5664,-0.1480){0.0263}
\end{pspicture*}
\endgroup

%% file: figure3-2.tex
\begingroup
\psset{unit=80pt}
\begin{pspicture*}(-1.0000,-0.5000)(1.0000,0.8000)
\pscircle*[linecolor=lightgray](0.0000,0.0000){1.0000}
\pscustom[fillstyle=solid,fillcolor=black,linestyle=none]{%
\psarcn(-1.0431,0.0000){0.2773}{0.8545}{-0.8545}
\psarc(0.0000,49.9967){50.0067}{-90.8775}{-90.5295}
\psarc(-1.3130,0.0000){0.8509}{-0.5295}{0.5295}
\psarc(0.0000,-49.9967){50.0067}{90.5295}{90.8775}
\closepath
}%
\pscustom[fillstyle=solid,fillcolor=black,linestyle=none]{%
\psarcn(-1.0431,0.0000){0.2773}{0.8545}{-0.8545}
\psarc(0.0000,49.9967){50.0067}{-90.8775}{-90.0000}
\psline(0.0000,-0.0100)(0.0000,0.0100)
\psarc(0.0000,-49.9967){50.0067}{90.0000}{90.8775}
\closepath
}%
\pscustom[fillstyle=solid,fillcolor=black,linestyle=none]{%
\psarcn(-1.3593,-1.5258){1.7622}{59.7960}{59.2870}
\psarc(-1.2877,0.4818){0.9638}{-30.7360}{11.4290}
\psarcn(-0.4016,0.9623){0.2953}{-78.5710}{-80.2310}
\psarcn(-1.3394,0.5011){1.0025}{9.7690}{-30.1811}
\closepath
}%
\pscustom[fillstyle=solid,fillcolor=black,linestyle=none]{%
\psarc(-1.2698,1.4253){1.6461}{-60.7335}{-60.1836}
\psarc(-1.2877,-0.4818){0.9638}{29.7935}{49.7403}
\psarcn(-0.9454,0.4913){0.3676}{-40.2597}{-41.7907}
\psarcn(-1.3394,-0.5011){1.0025}{48.2093}{29.2894}
\closepath
}%
\pscustom[fillstyle=solid,fillcolor=black,linestyle=none]{%
\psarcn(-1.3593,1.5258){1.7622}{-59.2870}{-59.7960}
\psarcn(-1.3394,-0.5011){1.0025}{30.1811}{6.9891}
\psarcn(-0.2102,-1.4738){1.1029}{96.9891}{96.2168}
\psarc(-1.2877,-0.4818){0.9638}{6.2168}{30.7360}
\closepath
}%
\pscustom[fillstyle=solid,fillcolor=black,linestyle=none]{%
\psarc(-1.2698,-1.4253){1.6461}{60.1836}{60.7335}
\psarcn(-1.3394,0.5011){1.0025}{-29.2894}{-48.2093}
\psarcn(-0.9454,-0.4913){0.3676}{41.7907}{40.2597}
\psarc(-1.2877,0.4818){0.9638}{-49.7403}{-29.7935}
\closepath
}%
\pscustom[fillstyle=solid,fillcolor=black,linestyle=none]{%
\psarcn(-1.3358,0.0000){0.8658}{0.5156}{-0.5156}
\psarc(0.0000,49.9967){50.0067}{-90.5385}{-89.4705}
\psarcn(1.3130,0.0000){0.8509}{-179.4705}{179.4705}
\psarc(0.0000,-49.9967){50.0067}{89.4705}{90.5385}
\closepath
}%
\pscustom[fillstyle=solid,fillcolor=black,linestyle=none]{%
\psarcn(-1.3358,0.0000){0.8658}{0.5156}{-0.5156}
\psarc(0.0000,49.9967){50.0067}{-90.5385}{-89.1273}
\psarcn(1.0373,0.0000){0.2757}{-179.1273}{179.1273}
\psarc(0.0000,-49.9968){50.0068}{89.1273}{90.5385}
\closepath
}%
\pscustom[fillstyle=solid,fillcolor=black,linestyle=none]{%
\psarc(1.2698,1.4253){1.6461}{-119.8164}{-119.2665}
\psarcn(1.3394,-0.5011){1.0025}{150.7106}{131.7907}
\psarcn(0.9454,0.4913){0.3676}{-138.2093}{-139.7403}
\psarc(1.2877,-0.4818){0.9638}{130.2597}{150.2065}
\closepath
}%
\pscustom[fillstyle=solid,fillcolor=black,linestyle=none]{%
\psarcn(1.3593,-1.5258){1.7622}{120.7130}{120.2040}
\psarcn(1.3394,0.5011){1.0025}{-149.8189}{170.2310}
\psarcn(0.4016,0.9623){0.2953}{-99.7690}{-101.4290}
\psarc(1.2877,0.4818){0.9638}{168.5710}{-149.2640}
\closepath
}%
\pscustom[fillstyle=solid,fillcolor=black,linestyle=none]{%
\psarc(1.2698,-1.4253){1.6461}{119.2665}{119.8164}
\psarc(1.2877,0.4818){0.9638}{-150.2065}{-130.2597}
\psarcn(0.9454,-0.4913){0.3676}{139.7403}{138.2093}
\psarcn(1.3394,0.5011){1.0025}{-131.7907}{-150.7106}
\closepath
}%
\pscustom[fillstyle=solid,fillcolor=black,linestyle=none]{%
\psarcn(1.3593,1.5258){1.7622}{-120.2040}{-120.7130}
\psarc(1.2877,-0.4818){0.9638}{149.2640}{173.7832}
\psarcn(0.2102,-1.4738){1.1029}{83.7832}{83.0109}
\psarcn(1.3394,-0.5011){1.0025}{173.0109}{149.8189}
\closepath
}%
\pscustom[fillstyle=solid,fillcolor=black,linestyle=none]{%
\psline(-0.0033,-0.0033)(0.0033,-0.0033)
\psline(0.0033,-0.0033)(0.0006,0.9052)
\psarcn(0.0000,1.0050){0.0998}{-89.6543}{-90.3457}
\psline(-0.0006,0.9052)(-0.0033,-0.0033)
\closepath
}%
\pscustom[fillstyle=solid,fillcolor=black,linestyle=none]{%
\psline(0.0033,0.0033)(-0.0033,0.0033)
\psline(-0.0033,0.0033)(-0.0014,-0.7616)
\psarcn(0.0000,-1.0373){0.2757}{90.2909}{89.7091}
\psline(0.0014,-0.7616)(0.0033,0.0033)
\closepath
}%
\rput{0.0000}(-0.2449,0.0000){\pscirclebox*[fillcolor=lightgray,framesep=0pt]{$\times$}}
\rput{0.0000}(-0.5670,0.1481){\pscirclebox*[fillcolor=lightgray,framesep=0pt]{$\times$}}
\rput{0.0000}(-0.5670,-0.1481){\pscirclebox*[fillcolor=lightgray,framesep=0pt]{$\times$}}
\pscircle*(-0.3788,0.1852){0.0329}
\pscircle*(-0.6345,0.0000){0.0239}
\pscircle*(-0.3788,-0.1852){0.0329}
\rput{0.0000}(0.5670,0.1481){\pscirclebox*[fillcolor=lightgray,framesep=0pt]{$\times$}}
\rput{0.0000}(0.2449,0.0000){\pscirclebox*[fillcolor=lightgray,framesep=0pt]{$\times$}}
\rput{0.0000}(0.5670,-0.1481){\pscirclebox*[fillcolor=lightgray,framesep=0pt]{$\times$}}
\pscircle*(0.6345,0.0000){0.0239}
\pscircle*(0.3788,0.1852){0.0329}
\pscircle*(0.3788,-0.1852){0.0329}
\end{pspicture*}
\endgroup

%% file: figure10.tex
\begingroup
\psset{unit=126pt}
\begin{pspicture*}(-0.2000,-0.2000)(0.8000,0.8000)
\pscircle*[linecolor=lightgray](0.0000,0.0000){1.0000}
\pscustom[fillstyle=solid,fillcolor=black,linestyle=none]{%
\psarcn(0.7813,0.6453){0.1448}{-145.2917}{-147.2242}
\psarcn(0.0000,1.5921){1.2190}{-57.2471}{-90.4046}
\psarc(49.9967,0.0000){50.0067}{179.5725}{-179.9885}
\psarc(0.0000,49.9967){50.0067}{-90.0115}{-89.2538}
\psarcn(1.1123,0.2705){0.5375}{-149.0740}{147.1363}
\psarcn(0.7416,0.6871){0.1488}{-122.8637}{-124.7768}
\psarc(1.0889,0.2648){0.5262}{145.2232}{-150.5429}
\psarc(0.0000,-49.9967){50.0067}{89.2773}{89.9885}
\psarc(-49.9967,0.0000){50.0067}{0.0115}{0.4076}
\psarc(0.0000,1.5181){1.1623}{-89.5695}{-55.2688}
\closepath
}%
\rput{0.0000}(0.5304,0.4221){\pscirclebox*[fillcolor=lightgray,framesep=0pt]{$\frac{\pi}{8}$}}
\rput{0.0000}(0.0813,0.2884){\pscirclebox*[fillcolor=lightgray,framesep=0pt]{$\frac{\pi}{2}$}}
\rput{0.0000}(0.0705,0.0705){\pscirclebox*[fillcolor=lightgray,framesep=0pt]{$\frac{\pi}{2}$}}
\rput{0.0000}(0.5332,0.0729){\pscirclebox*[fillcolor=lightgray,framesep=0pt]{$\frac{\pi}{3}$}}
\rput{0.0000}(0.2978,0.5098){\pscirclebox*[fillcolor=lightgray,framesep=0pt]{$R_1$}}
\rput{0.0000}(0.6750,0.2976){\pscirclebox*[fillcolor=lightgray,framesep=0pt]{$R_2$}}
\rput{0.0000}(0.3658,-0.0864){\pscirclebox*[fillcolor=lightgray,framesep=0pt]{$R_3$}}
\rput{0.0000}(-0.0961,0.1905){\pscirclebox*[fillcolor=lightgray,framesep=0pt]{$R_4$}}
\end{pspicture*}
\endgroup

%% file: figure8.tex
\begingroup
\psset{unit=126pt}
\begin{pspicture*}(-1.0000,-0.3000)(1.0000,0.6000)
\pscircle*[linecolor=lightgray](0.0000,0.0000){1.0000}
\pscustom[fillstyle=solid,fillcolor=black,linestyle=none]{%
\psarc(-2.2204,-2.5058){3.2152}{56.7223}{56.9901}
\psarcn(-1.1227,0.6157){0.7800}{-33.0328}{-63.8079}
\psarcn(0.0000,-12.5366){12.4767}{93.5772}{86.4228}
\psarcn(1.1227,0.6157){0.7800}{-116.1921}{-146.4212}
\psarc(2.3730,-2.6779){3.4355}{123.5788}{123.8288}
\psarc(1.0888,0.5970){0.7564}{-146.1712}{-117.6964}
\psarc(0.0000,-8.3633){8.3233}{84.9185}{95.0815}
\psarc(-1.0888,0.5970){0.7564}{-62.3036}{-33.2548}
\closepath
}%
\pscustom[fillstyle=solid,fillcolor=black,linestyle=none]{%
\psarcn(0.0000,-12.5366){12.4767}{90.0458}{89.9542}
\psarc(-49.9967,0.0000){50.0067}{-0.0687}{0.3598}
\psarcn(0.0000,1.7493){1.4352}{-89.6402}{-90.3598}
\psarc(49.9967,0.0000){50.0067}{179.6402}{-179.9313}
\closepath
}%
\pscustom[fillstyle=solid,fillcolor=black,linestyle=none]{%
\psarcn(-1.8759,-3.0919){3.4704}{61.1272}{61.0482}
\psarc(-2.9967,1.4944){3.2009}{-28.9533}{-23.6982}
\psarcn(1.1019,2.8682){2.9053}{-113.6982}{-113.7921}
\psarcn(-3.0958,1.5438){3.3067}{-23.7921}{-28.8713}
\closepath
}%
\pscustom[fillstyle=solid,fillcolor=black,linestyle=none]{%
\psarc(-4.3362,6.8170){8.0221}{-58.9432}{-58.9089}
\psarc(-2.2692,-1.3043){2.4238}{31.0896}{37.6381}
\psarcn(-1.0825,1.1258){1.1997}{-52.3619}{-52.5639}
\psarcn(-2.3256,-1.3367){2.4840}{37.4361}{31.0582}
\closepath
}%
\rput{0.0000}(-0.6771,0.0763){\pscirclebox*[fillcolor=lightgray,framesep=0pt]{$R_1$}}
\rput{0.0000}(0.0000,-0.1489){\pscirclebox*[fillcolor=lightgray,framesep=0pt]{$R_2$}}
\rput{0.0000}(0.6771,0.0763){\pscirclebox*[fillcolor=lightgray,framesep=0pt]{$R_1'$}}
\rput{0.0000}(0.0898,0.3168){\pscirclebox*[fillcolor=lightgray,framesep=0pt]{$T$}}
\end{pspicture*}
\endgroup

%% file: figure9.tex
\begingroup
\psset{unit=126pt}
\begin{pspicture*}(-1.0000,-0.3000)(1.0000,1.0000)
\pscircle*[linecolor=lightgray](0.0000,0.0000){1.0000}
\rput{0.0000}(0.2608,0.5032){\pscirclebox*[fillcolor=lightgray,framesep=0pt]{next $A_2$}}
\rput{0.0000}(0.6379,-0.2248){\pscirclebox*[fillcolor=lightgray,framesep=0pt]{nearest translate}}
\pscustom[fillstyle=solid,fillcolor=black,linestyle=none]{%
\psarcn(-1.0023,0.1478){0.1440}{-37.8079}{-39.4380}
\psarcn(-1.0192,-0.0993){0.2017}{50.5391}{-21.2915}
\psarcn(0.0000,-5.6253){5.5158}{98.6688}{82.3037}
\psarcn(1.0208,-0.1974){0.2847}{172.3037}{170.5754}
\psarc(0.0000,-4.6087){4.5189}{80.5754}{100.4869}
\psarc(-1.0111,-0.0985){0.2000}{-19.4734}{52.2150}
\closepath
}%
\pscustom[fillstyle=solid,fillcolor=black,linestyle=none]{%
\psarcn(0.0000,-5.3313){5.2267}{95.2592}{95.1761}
\psarc(-1.2442,-0.1958){0.7758}{5.1703}{44.9072}
\psarcn(-0.8825,0.5403){0.2660}{-45.0928}{-45.9383}
\psarcn(-1.2638,-0.1989){0.7881}{44.0617}{5.2649}
\closepath
}%
\pscustom[fillstyle=solid,fillcolor=black,linestyle=none]{%
\psarc(19.0862,30.4640){35.9552}{-121.7647}{-121.7341}
\psarcn(2.3210,-1.4445){2.5244}{148.2430}{133.1487}
\psarcn(0.8472,0.6341){0.3463}{-136.8513}{-138.4717}
\psarc(2.1082,-1.3120){2.2929}{131.5283}{148.2582}
\closepath
}%
\pscustom[fillstyle=solid,fillcolor=black,linestyle=none]{%
\psarcn(1.5776,-2.7318){2.9720}{118.3243}{117.9554}
\psarcn(4.7190,2.2977){5.1326}{-152.0675}{-160.5407}
\psarcn(-0.2997,1.0955){0.5384}{-70.5407}{-71.9048}
\psarc(3.9153,1.9063){4.2584}{-161.9048}{-151.6528}
\closepath
}%
\pscustom[fillstyle=solid,fillcolor=black,linestyle=none]{%
\psarcn(0.0000,-5.3313){5.2267}{84.4165}{84.3375}
\psarcn(1.2071,-0.1987){0.6948}{174.3318}{132.8573}
\psarcn(0.9125,0.4757){0.2428}{-137.1427}{-138.0041}
\psarc(1.1906,-0.1960){0.6853}{131.9959}{174.4222}
\closepath
}%
\pscustom[fillstyle=solid,fillcolor=black,linestyle=none]{%
\psarc(1.1192,0.3351){0.6244}{-128.1837}{-127.3904}
\psarcn(0.9945,-0.3556){0.3203}{142.5867}{101.1051}
\psarcn(1.0021,-0.0277){0.0707}{-168.8949}{-170.9831}
\psarc(0.9819,-0.3511){0.3163}{99.0169}{141.8393}
\closepath
}%
\pscustom[fillstyle=solid,fillcolor=black,linestyle=none]{%
\psarcn(0.9384,-0.6438){0.5235}{112.6295}{111.6974}
\psarcn(1.0630,-0.0309){0.3424}{-158.3255}{139.3123}
\psarcn(0.9580,0.3722){0.2372}{-130.6878}{-132.2370}
\psarc(1.0486,-0.0305){0.3378}{137.7630}{-157.3476}
\closepath
}%
\rput{0.0000}(-0.9054,-0.0259){\pscirclebox*[fillcolor=lightgray,framesep=0pt]{$R_2$}}
\rput{0.0000}(-0.1698,-0.1974){\pscirclebox*[fillcolor=lightgray,framesep=0pt]{$R_3$}}
\rput{0.0000}(-0.7222,0.3759){\pscirclebox*[fillcolor=lightgray,framesep=0pt]{next $\frac{\pi}{even}$}}
\rput{0.0000}(-0.8146,-0.1087){\pscirclebox*[fillcolor=lightgray,framesep=0pt]{$\times$}}
\pscircle*(-0.7592,-0.1580){0.0159}
\rput{0.0000}(0.3543,-0.1123){\pscirclebox*[fillcolor=lightgray,framesep=0pt]{$\times$}}
\rput{0.0000}(0.0830,0.0724){\pscirclebox*[fillcolor=lightgray,framesep=0pt]{$\times$}}
\pscircle*(-0.0247,-0.0996){0.0396}
\pscircle*(0.2775,0.0537){0.0368}
\rput{0.0000}(0.6407,-0.1412){\pscirclebox*[fillcolor=lightgray,framesep=0pt]{$\times$}}
\rput{0.0000}(0.7957,-0.0999){\pscirclebox*[fillcolor=lightgray,framesep=0pt]{$\times$}}
\pscircle*(0.7171,-0.0661){0.0192}
\end{pspicture*}
\endgroup

%% file: figure11.tex
\begingroup
\psset{unit=126pt}
\begin{pspicture*}(-1.0000,-0.3000)(1.0000,1.0000)
\pscircle*[linecolor=lightgray](0.0000,0.0000){1.0000}
\rput{0.0000}(-0.3507,0.2408){\pscirclebox*[fillcolor=lightgray,framesep=0pt]{$R_k$}}
\rput{0.0000}(-0.5353,-0.0643){\pscirclebox*[fillcolor=lightgray,framesep=0pt]{$c_k$}}
\rput{0.0000}(0.0000,-0.1974){\pscirclebox*[fillcolor=lightgray,framesep=0pt]{$R_{k+1}$}}
\rput{0.0000}(-0.2696,0.0383){\pscirclebox*[fillcolor=lightgray,framesep=0pt]{$\frac{\pi}{m_k}$}}
\rput{0.0000}(0.5080,0.2609){\pscirclebox*[fillcolor=lightgray,framesep=0pt]{$R_{k+2}$}}
\rput{0.0000}(0.2872,0.0642){\pscirclebox*[fillcolor=lightgray,framesep=0pt]{$\frac{\pi}{m_{k+1}}$}}
\pscustom[fillstyle=solid,fillcolor=black,linestyle=none]{%
\psarcn(0.0560,1.3368){0.8691}{-103.9379}{-104.8993}
\psarcn(-1.4659,0.8429){1.3438}{-14.9222}{-42.7442}
\psarcn(0.0000,-12.5366){12.4767}{92.2006}{87.8235}
\psarcn(1.4026,0.4204){1.0498}{-152.2162}{175.4766}
\psarcn(0.4022,1.0866){0.5852}{-94.5234}{-95.7450}
\psarc(1.3461,0.4035){1.0075}{174.2550}{-153.1212}
\psarc(0.0000,-8.3633){8.3233}{86.9184}{93.0440}
\psarc(-1.3912,0.7999){1.2752}{-41.9007}{-13.9150}
\closepath
}%
\end{pspicture*}
\endgroup

%% file: figure12-1.tex
\begingroup
\psset{unit=75pt}
\begin{pspicture*}(-0.5000,-0.3000)(1.0000,1.0000)
\pscircle*[linecolor=lightgray](0.0000,0.0000){1.0000}
\rput{0.0000}(-0.2837,0.2853){\pscirclebox*[fillcolor=lightgray,framesep=0pt]{$R_k$}}
\rput{0.0000}(0.0000,-0.1489){\pscirclebox*[fillcolor=lightgray,framesep=0pt]{$R_{k+1}$}}
\rput{0.0000}(-0.1731,0.0446){\pscirclebox*[fillcolor=lightgray,framesep=0pt]{$\frac{\pi}{4}$}}
\rput{0.0000}(0.2922,0.0436){\pscirclebox*[fillcolor=lightgray,framesep=0pt]{$\frac{\pi}{2}$}}
\pscustom[fillstyle=solid,fillcolor=black,linestyle=none]{%
\psarcn(0.2802,1.2887){0.8400}{-110.8415}{-111.8551}
\psarcn(-1.7503,1.1989){1.8512}{-21.8780}{-43.1124}
\psarcn(0.0000,-12.5366){12.4767}{91.8323}{88.2222}
\psarcn(1.5328,-0.1020){1.1463}{178.1993}{150.6656}
\psarcn(0.7200,0.7916){0.3808}{-119.3344}{-120.8587}
\psarc(1.4656,-0.0975){1.0960}{149.1413}{177.4247}
\psarc(0.0000,-8.3633){8.3233}{87.4476}{92.4682}
\psarc(-1.6297,1.1162){1.7236}{-42.4765}{-20.8186}
\closepath
}%
\end{pspicture*}
\endgroup

%% file: figure12-2.tex
\begingroup
\psset{unit=75pt}
\begin{pspicture*}(-0.6000,-0.3000)(1.0000,1.0000)
\pscircle*[linecolor=lightgray](0.0000,0.0000){1.0000}
\rput{0.0000}(-0.4999,0.2222){\pscirclebox*[fillcolor=lightgray,framesep=0pt]{$R_k$}}
\rput{0.0000}(0.0000,-0.1489){\pscirclebox*[fillcolor=lightgray,framesep=0pt]{$R_{k+1}$}}
\rput{0.0000}(-0.3460,0.0505){\pscirclebox*[fillcolor=lightgray,framesep=0pt]{$\frac{\pi}{4}$}}
\rput{0.0000}(0.4631,0.0524){\pscirclebox*[fillcolor=lightgray,framesep=0pt]{$\frac{\pi}{2}$}}
\rput{0.0000}(0.0381,0.0517){\pscirclebox*[fillcolor=lightgray,framesep=0pt]{$\frac{\pi}{3}$}}
\pscustom[fillstyle=solid,fillcolor=black,linestyle=none]{%
\psarcn(-0.2293,1.3054){0.8501}{-95.6208}{-96.5458}
\psarcn(-1.2747,0.5701){0.9548}{-6.5687}{-42.3315}
\psarcn(0.0000,-12.5366){12.4767}{92.6132}{87.4307}
\psarcn(1.1877,-0.1009){0.6291}{177.4078}{137.9766}
\psarcn(0.8963,0.5154){0.2628}{-132.0234}{-133.6842}
\psarc(1.1586,-0.0984){0.6136}{136.3158}{176.2139}
\psarc(0.0000,-8.3633){8.3233}{86.2368}{93.6947}
\psarc(-1.2278,0.5491){0.9197}{-41.2500}{-5.5979}
\closepath
}%
\pscustom[fillstyle=solid,fillcolor=black,linestyle=none]{%
\psarc(4.0678,6.3473){7.4773}{-121.1360}{-121.0994}
\psarcn(2.2526,-1.2902){2.3906}{148.8991}{134.9279}
\psarcn(0.8334,0.6707){0.3800}{-135.0721}{-135.4640}
\psarc(2.2000,-1.2600){2.3347}{134.5360}{148.8654}
\closepath
}%
\pscustom[fillstyle=solid,fillcolor=black,linestyle=none]{%
\psarcn(1.8303,-3.0005){3.3644}{118.9092}{118.8279}
\psarcn(2.9618,1.4625){3.1432}{-151.1735}{-163.3386}
\psarcn(-0.2267,1.1536){0.6183}{-73.3386}{-73.6550}
\psarc(2.8715,1.4180){3.0474}{-163.6550}{-151.0894}
\closepath
}%
\end{pspicture*}
\endgroup

%% file: figure12-3.tex
\begingroup
\psset{unit=75pt}
\begin{pspicture*}(-0.6000,-0.3000)(1.0000,1.0000)
\pscircle*[linecolor=lightgray](0.0000,0.0000){1.0000}
\rput{0.0000}(-0.4043,0.2771){\pscirclebox*[fillcolor=lightgray,framesep=0pt]{$R_k$}}
\rput{0.0000}(0.0000,-0.1489){\pscirclebox*[fillcolor=lightgray,framesep=0pt]{$R_{k+1}$}}
\rput{0.0000}(-0.2529,0.0423){\pscirclebox*[fillcolor=lightgray,framesep=0pt]{$\frac{\pi}{3}$}}
\rput{0.0000}(0.3153,0.0414){\pscirclebox*[fillcolor=lightgray,framesep=0pt]{$\frac{\pi}{2}$}}
\pscustom[fillstyle=solid,fillcolor=black,linestyle=none]{%
\psarcn(-0.2543,1.1610){0.6226}{-89.3196}{-90.5084}
\psarcn(-1.5722,0.5505){1.3125}{-0.5313}{-28.0586}
\psarcn(0.0000,-12.5366){12.4767}{91.9017}{88.1261}
\psarcn(1.4704,-0.1018){1.0630}{178.1032}{149.0771}
\psarcn(0.7471,0.7593){0.3670}{-120.9229}{-122.4620}
\psarc(1.4105,-0.0976){1.0197}{147.5380}{177.2780}
\psarc(0.0000,-8.3633){8.3233}{87.3010}{92.6564}
\psarc(-1.4938,0.5231){1.2470}{-27.3039}{0.7033}
\closepath
}%
\end{pspicture*}
\endgroup

%% file: figure13.tex
\begingroup
\psset{unit=126pt}
\begin{pspicture*}(-1.0000,-0.5000)(1.0000,1.0000)
\pscircle*[linecolor=lightgray](0.0000,0.0000){1.0000}
\pscustom[fillstyle=solid,fillcolor=black,linestyle=none]{%
\psarcn(-0.6847,0.7699){0.2290}{-74.9497}{-76.4760}
\psarcn(-0.9464,0.4715){0.3242}{13.5011}{-60.8698}
\psarcn(-1.1823,-0.4692){0.7663}{59.0905}{12.5921}
\psarcn(0.0000,-2.2503){1.9960}{102.5692}{77.4308}
\psarcn(1.1823,-0.4692){0.7663}{167.4079}{121.3591}
\psarcn(0.9798,0.3048){0.2299}{-148.6409}{-150.4035}
\psarc(1.1471,-0.4552){0.7435}{119.5965}{166.7354}
\psarc(0.0000,-2.0840){1.8485}{76.7583}{103.2417}
\psarc(-1.1471,-0.4552){0.7435}{13.2646}{59.9231}
\psarc(-0.9343,0.4655){0.3201}{-60.0372}{15.0732}
\closepath
}%
\rput{0.0000}(-0.6998,0.3716){\pscirclebox*[fillcolor=lightgray,framesep=0pt]{$R_1$}}
\rput{0.0000}(-0.6390,-0.0110){\pscirclebox*[fillcolor=lightgray,framesep=0pt]{$R_2$}}
\rput{0.0000}(0.0000,-0.3364){\pscirclebox*[fillcolor=lightgray,framesep=0pt]{$R_3$}}
\rput{0.0000}(0.6390,-0.0110){\pscirclebox*[fillcolor=lightgray,framesep=0pt]{$R_4$}}
\rput{0.0000}(-0.6748,0.1932){\pscirclebox*[fillcolor=lightgray,framesep=0pt]{$\frac{\pi}{3}$}}
\rput{0.0000}(-0.3799,-0.2104){\pscirclebox*[fillcolor=lightgray,framesep=0pt]{$\frac{\pi}{2}$}}
\rput{0.0000}(0.3799,-0.2104){\pscirclebox*[fillcolor=lightgray,framesep=0pt]{$\frac{\pi}{2}$}}
\end{pspicture*}
\endgroup

%% file: figure13-p.tex
\begingroup
\psset{unit=126pt}
\begin{pspicture*}(-1.0000,-0.5000)(1.0000,1.0000)
\pscircle*[linecolor=lightgray](0.0000,0.0000){1.0000}
\pscustom[fillstyle=solid,fillcolor=black,linestyle=none]{%
\psarcn(-0.6847,0.7699){0.2290}{-74.9497}{-76.4760}
\psarcn(-0.9464,0.4715){0.3242}{13.5011}{-60.8698}
\psarcn(-1.1823,-0.4692){0.7663}{59.0905}{12.5921}
\psarcn(0.0000,-2.2503){1.9960}{102.5692}{77.4308}
\psarcn(1.1823,-0.4692){0.7663}{167.4079}{121.0980}
\psarcn(0.9843,0.3062){0.2310}{-148.9249}{140.4966}
\psarcn(0.8595,0.5180){0.0840}{-129.5034}{-131.4927}
\psarc(0.9753,0.3034){0.2289}{138.5073}{-150.1506}
\psarc(1.1471,-0.4552){0.7435}{119.8723}{166.7354}
\psarc(0.0000,-2.0840){1.8485}{76.7583}{103.2417}
\psarc(-1.1471,-0.4552){0.7435}{13.2646}{59.9231}
\psarc(-0.9343,0.4655){0.3201}{-60.0372}{15.0732}
\closepath
}%
\pscustom[fillstyle=solid,fillcolor=black,linestyle=none]{%
\psarc(-0.7587,1.1392){0.9548}{-91.5902}{-91.1643}
\psarc(-1.0029,0.1892){0.2248}{-1.1873}{55.8997}
\psarcn(-0.9167,0.4024){0.0482}{-34.1003}{-36.2384}
\psarcn(-1.0119,0.1910){0.2268}{53.7616}{-1.5673}
\closepath
}%
\pscustom[fillstyle=solid,fillcolor=black,linestyle=none]{%
\psarc(-0.9753,0.3034){0.2289}{-31.6471}{-29.8494}
\psarc(-1.1471,-0.4552){0.7435}{60.1277}{73.4715}
\psarcn(-0.9644,0.2661){0.0301}{-16.5285}{-18.7506}
\psarcn(-1.1823,-0.4692){0.7663}{71.2494}{58.3758}
\closepath
}%
\rput{0.0000}(-0.6998,0.3716){\pscirclebox*[fillcolor=lightgray,framesep=0pt]{$R_1$}}
\rput{0.0000}(-0.6390,-0.0110){\pscirclebox*[fillcolor=lightgray,framesep=0pt]{$R_2$}}
\rput{0.0000}(0.0000,-0.3364){\pscirclebox*[fillcolor=lightgray,framesep=0pt]{$R_3$}}
\rput{0.0000}(0.6390,-0.0110){\pscirclebox*[fillcolor=lightgray,framesep=0pt]{$R_4$}}
\rput{0.0000}(0.8186,0.3490){\pscirclebox*[fillcolor=lightgray,framesep=0pt]{$T$}}
\rput{0.0000}(-0.8986,0.2970){\pscirclebox*[fillcolor=lightgray,framesep=0pt]{$R_2'$}}
\rput{0.0000}(-0.6748,0.1932){\pscirclebox*[fillcolor=lightgray,framesep=0pt]{$\frac{\pi}{3}$}}
\rput{0.0000}(-0.3799,-0.2104){\pscirclebox*[fillcolor=lightgray,framesep=0pt]{$\frac{\pi}{2}$}}
\rput{0.0000}(0.3799,-0.2104){\pscirclebox*[fillcolor=lightgray,framesep=0pt]{$\frac{\pi}{2}$}}
\rput{0.0000}(-0.8230,0.1367){\pscirclebox*[fillcolor=lightgray,framesep=0pt]{$p_1$}}
\rput{0.0000}(-0.4669,-0.3516){\pscirclebox*[fillcolor=lightgray,framesep=0pt]{$p_2$}}
\rput{0.0000}(0.8417,0.1768){\pscirclebox*[fillcolor=lightgray,framesep=0pt]{$p$}}
\end{pspicture*}
\endgroup

%% file: figure14.tex
\begingroup
\psset{unit=126pt}
\begin{pspicture*}(-1.0000,-0.5000)(1.0000,1.0000)
\pscircle*[linecolor=lightgray](0.0000,0.0000){1.0000}
\pscustom[fillstyle=solid,fillcolor=black,linestyle=none]{%
\psarcn(-0.9843,0.3062){0.2310}{-29.3478}{-31.0751}
\psarcn(-1.1823,-0.4692){0.7663}{58.9020}{12.5921}
\psarcn(0.0000,-2.2503){1.9960}{102.5692}{77.4308}
\psarcn(1.1823,-0.4692){0.7663}{167.4079}{121.0980}
\psarcn(0.9843,0.3062){0.2310}{-148.9249}{140.4966}
\psarcn(0.8595,0.5180){0.0840}{-129.5034}{-131.4927}
\psarc(0.9753,0.3034){0.2289}{138.5073}{-150.1506}
\psarc(1.1471,-0.4552){0.7435}{119.8723}{166.7354}
\psarc(0.0000,-2.0840){1.8485}{76.7583}{103.2417}
\psarc(-1.1471,-0.4552){0.7435}{13.2646}{60.6751}
\closepath
}%
\pscustom[fillstyle=solid,fillcolor=black,linestyle=none]{%
\psarcn(-1.1497,-0.0753){0.5672}{10.5764}{10.2486}
\psarc(-1.0182,2.3857){2.3984}{-79.7528}{-69.9399}
\psarc(-3.4069,-1.0398){3.4188}{20.0601}{20.1392}
\psarcn(-1.0432,2.4443){2.4573}{-69.8608}{-79.4221}
\closepath
}%
\pscustom[fillstyle=solid,fillcolor=black,linestyle=none]{%
\psarcn(-1.2366,-1.7767){1.9148}{70.3376}{70.2403}
\psarc(-1.1294,0.2195){0.5740}{-19.7611}{12.2654}
\psarcn(-0.6995,0.9438){0.6165}{-77.7346}{-77.9943}
\psarcn(-1.1359,0.2207){0.5773}{12.0057}{-19.6610}
\closepath
}%
\pscustom[fillstyle=solid,fillcolor=black,linestyle=none]{%
\psarcn(1.1147,0.0246){0.4881}{137.3292}{137.1523}
\psarcn(0.8892,0.4992){0.1946}{-132.8491}{-178.2697}
\psarcn(0.6859,0.7835){0.2903}{-88.2697}{-88.5411}
\psarc(0.8875,0.4982){0.1942}{-178.5411}{-132.6694}
\closepath
}%
\pscustom[fillstyle=solid,fillcolor=black,linestyle=none]{%
\psarcn(0.9311,0.4102){0.1826}{-162.7116}{-163.1830}
\psarcn(0.3891,1.5723){1.2692}{-73.1844}{-83.3715}
\psarc(1.0743,0.3742){0.5423}{-173.3715}{-173.0456}
\psarc(0.3842,1.5526){1.2532}{-83.0456}{-72.7101}
\closepath
}%
\rput{0.0000}(-0.6390,-0.0110){\pscirclebox*[fillcolor=lightgray,framesep=0pt]{$R_1$}}
\rput{0.0000}(0.0000,-0.3364){\pscirclebox*[fillcolor=lightgray,framesep=0pt]{$R_2$}}
\rput{0.0000}(0.6390,-0.0110){\pscirclebox*[fillcolor=lightgray,framesep=0pt]{$R_3$}}
\rput{0.0000}(0.8237,0.3490){\pscirclebox*[fillcolor=lightgray,framesep=0pt]{$R_4$}}
\rput{0.0000}(-0.6844,0.2254){\pscirclebox*[fillcolor=lightgray,framesep=0pt]{$\frac{\pi}{2}$}}
\rput{0.0000}(-0.3799,-0.2104){\pscirclebox*[fillcolor=lightgray,framesep=0pt]{$\frac{\pi}{2}$}}
\rput{0.0000}(0.3799,-0.2104){\pscirclebox*[fillcolor=lightgray,framesep=0pt]{$\frac{\pi}{2}$}}
\rput{0.0000}(0.7054,0.2152){\pscirclebox*[fillcolor=lightgray,framesep=0pt]{$\frac{\pi}{2}$}}
\end{pspicture*}
\endgroup

%% file: appendix-table.tex
\section{Tables}

\input{triangles}
\pagebreak
\input{quadrilaterals}
\pagebreak
\input{pentagons}
\pagebreak

\input{hexagons}
\pagebreak
\input{septagons}
\pagebreak

\input{octagons}
\pagebreak

\input{decagons}
\pagebreak
\input{undecagons}
\pagebreak
\input{dodecagons}
\input{tetradecagons}
\input{hexadecagons}
\pagebreak
\input{octadecagons}
\pagebreak
\input{icosagons}
\pagebreak
\input{24-gons}

%% file: triangles.tex
\begin{center}
\begin{longtable}{cc}
\caption{Triangles}\\
\endfirsthead

\multicolumn{2}{c}
{{ \tablename\ \thetable{}, cont.}} \\
\hline\\
\endhead

\endfoot

\endlastfoot

\hline\hline
det = $1 - w$ &  det norm = $-1$ \\*
\multicolumn{2}{c}{$824$}\\*
\hline\\*
\multicolumn{2}{c}{quadratic form}\\*
\multicolumn{2}{c}{
$\left(\begin{array}{cc|cc|cc}
1 & -1 & -2 & 2 & -1 & 0 \\*
-2 & 2 & 6 & -4 & 1 & -1 \\*
-1 & 0 & 1 & -1 & 1 & 0 \\*
\end{array}\right)$}
\\\\
\multicolumn{2}{c}{root list}\\*
\multicolumn{2}{c}{
$\begin{array}{|l|cc|cc|cc||cc|}
\hline
\multicolumn{7}{|c||}{\text{roots}} & \multicolumn{2}{|c|}{\text{norms}}\\*\hline
r_{1} & 0 & 0 & -1 & -1 & -2 & 0 & 2 & 0 \\*
r_{2} & -1 & 1 & 0 & 1 & 1 & 0 & 2 & -1 \\*
r_{3} & 0 & 0 & 0 & 0 & -1 & 1 & 3 & -2 \\*
\hline
\end{array}$}
\\
\end{longtable}
\end{center}

%% file: quadrilaterals.tex
\begin{center}
\begin{longtable}{cc}
\caption{Quadrilaterals}\\
\endfirsthead

\multicolumn{2}{c}
{{ \tablename\ \thetable{}, cont.}} \\
\hline\\
\endhead

\endfoot

\endlastfoot

\hline\hline
det = $-24 - 17w$ &  det norm = $-2$ \\*
\multicolumn{2}{c}{$2224$}\\*
\hline\\*
\multicolumn{2}{c}{quadratic form}\\*
\multicolumn{2}{c}{
$\left(\begin{array}{cc|cc|cc}
2 & -1 & -6 & 4 & -2 & 1 \\*
-6 & 4 & 16 & -17 & -1 & -9 \\*
-2 & 1 & -1 & -9 & -5 & -6 \\*
\end{array}\right)$}
\\\\
\multicolumn{2}{c}{root list}\\*
\multicolumn{2}{c}{
$\begin{array}{|l|cc|cc|cc||cc|}
\hline
\multicolumn{7}{|c||}{\text{roots}} & \multicolumn{2}{|c|}{\text{norms}}\\*\hline
r_{1} & 4 & 2 & 0 & 0 & 1 & 0 & 3 & 2 \\*
r_{2} & 9 & 5 & 2 & 1 & 0 & 0 & 10 & 7 \\*
r_{3} & 0 & 0 & 1 & 1 & 0 & -1 & 10 & 7 \\*
r_{4} & -7 & -6 & -1 & -1 & 0 & 0 & 34 & 24 \\*
\hline
\end{array}$}
\\
\end{longtable}
\end{center}

%% file: pentagons.tex
\begin{center}
\begin{longtable}{cc}
\caption{Pentagons}\\
\endfirsthead

\multicolumn{2}{c}
{{ \tablename\ \thetable{}, cont.}} \\
\hline\\
\endhead

\endfoot

\endlastfoot

\hline\hline
det = $-11 - 8w$ &  det norm = $-7$ \\*
\multicolumn{2}{c}{$22222$}\\*
\hline\\*
\multicolumn{2}{c}{quadratic form}\\*
\multicolumn{2}{c}{
$\left(\begin{array}{cc|cc|cc}
11 & -10 & -6 & 5 & -3 & 10 \\*
-6 & 5 & 4 & -2 & -1 & -7 \\*
-3 & 10 & -1 & -7 & 5 & -5 \\*
\end{array}\right)$}
\\\\
\multicolumn{2}{c}{root list}\\*
\multicolumn{2}{c}{
$\begin{array}{|l|cc|cc|cc||cc|}
\hline
\multicolumn{7}{|c||}{\text{roots}} & \multicolumn{2}{|c|}{\text{norms}}\\*\hline
r_{1} & 3 & 1 & 3 & 1 & 1 & 0 & 2 & 1 \\*
r_{2} & -1 & 1 & -2 & 2 & 0 & 0 & 1 & 0 \\*
r_{3} & -6 & -3 & -5 & -3 & -2 & 0 & 10 & 6 \\*
r_{4} & -1 & -1 & -1 & -1 & 0 & 0 & 1 & 0 \\*
r_{5} & 2 & 0 & 3 & -1 & 1 & 1 & 5 & 3 \\*
\hline
\end{array}$}
\\
\end{longtable}
\end{center}

%% file: hexagons.tex
\begin{center}
\begin{longtable}{cc}
\caption{Hexagons}\\
\endfirsthead

\multicolumn{2}{c}
{{ \tablename\ \thetable{}, cont.}} \\
\hline\\
\endhead

\endfoot

\endlastfoot

\hline\hline
det = $-2 - 3w$ &  det norm = $-14$ \\*
\multicolumn{2}{c}{$222222$}\\*
\hline\\*
\multicolumn{2}{c}{quadratic form}\\*
\multicolumn{2}{c}{
$\left(\begin{array}{cc|cc|cc}
13 & 9 & 0 & 0 & 0 & 0 \\*
0 & 0 & 4 & -3 & 0 & 0 \\*
0 & 0 & 0 & 0 & 1 & 0 \\*
\end{array}\right)$}
\\\\
\multicolumn{2}{c}{root list}\\*
\multicolumn{2}{c}{
$\begin{array}{|l|cc|cc|cc||cc|}
\hline
\multicolumn{7}{|c||}{\text{roots}} & \multicolumn{2}{|c|}{\text{norms}}\\*\hline
r_{1} & 1 & 0 & 7 & 5 & -2 & -2 & 1 & 0 \\*
r_{2} & 0 & 1 & 13 & 9 & -5 & -4 & 3 & 1 \\*
r_{3} & 0 & -2 & -21 & -15 & 8 & 5 & 2 & -1 \\*
r_{4} & -3 & -3 & -57 & -40 & 21 & 15 & 2 & 0 \\*
r_{5} & -15 & -9 & -225 & -159 & 85 & 61 & 6 & 2 \\*
r_{6} & -4 & -1 & -45 & -32 & 18 & 12 & 2 & -1 \\*
\hline
\end{array}$}
\\
\end{longtable}
\end{center}

%% file: septagons.tex
\begin{center}
\begin{longtable}{cc}
\caption{Septagons}\\
\endfirsthead

\multicolumn{2}{c}
{{ \tablename\ \thetable{}, cont.}} \\
\hline\\
\endhead

\endfoot

\endlastfoot

\hline\hline
det = $-22387 - 15830w$ &  det norm = $-31$ \\*
\multicolumn{2}{c}{$2228222$}\\*
\hline\\*
\multicolumn{2}{c}{quadratic form}\\*
\multicolumn{2}{c}{
$\left(\begin{array}{cc|cc|cc}
17037 & 12038 & 2307 & 1633 & -152 & -112 \\*
2307 & 1633 & 313 & 221 & -21 & -14 \\*
-152 & -112 & -21 & -14 & 2 & -1 \\*
\end{array}\right)$}
\\\\
\multicolumn{2}{c}{root list}\\*
\multicolumn{2}{c}{
$\begin{array}{|l|cc|cc|cc||cc|}
\hline
\multicolumn{7}{|c||}{\text{roots}} & \multicolumn{2}{|c|}{\text{norms}}\\*\hline
r_{1} & 3 & 1 & -13 & -14 & -2 & -1 & 58 & 41 \\*
r_{2} & 1 & 1 & -9 & -5 & 0 & 0 & 932 & 659 \\*
r_{3} & -1 & 0 & 2 & 4 & 2 & 1 & 17 & 12 \\*
r_{4} & 0 & 0 & 0 & 0 & 7 & 5 & 58 & 41 \\*
r_{5} & 5 & 3 & -33 & -26 & 7 & 5 & 198 & 140 \\*
r_{6} & 8 & 6 & -63 & -43 & 3 & 2 & 99 & 70 \\*
r_{7} & 56 & 40 & -423 & -297 & 0 & 0 & 1591 & 1125 \\*
\hline
\end{array}$}
\\
\end{longtable}
\end{center}

%% file: octagons.tex
\begin{center}
\begin{longtable}{cc}
\caption{Octagons}\\
\endfirsthead

\multicolumn{2}{c}
{{ \tablename\ \thetable{}, cont.}} \\
\hline\\
\endhead

\endfoot

\endlastfoot

\hline\hline
det = $-32 - 23w$ &  det norm = $-34$ \\*
\multicolumn{2}{c}{$22222222$}\\*
\hline\\*
\multicolumn{2}{c}{quadratic form}\\*
\multicolumn{2}{c}{
$\left(\begin{array}{cc|cc|cc}
46 & 32 & 0 & 0 & -9 & -7 \\*
0 & 0 & 0 & -1 & 2 & 1 \\*
-9 & -7 & 2 & 1 & -1 & -2 \\*
\end{array}\right)$}
\\\\
\multicolumn{2}{c}{root list}\\*
\multicolumn{2}{c}{
$\begin{array}{|l|cc|cc|cc||cc|}
\hline
\multicolumn{7}{|c||}{\text{roots}} & \multicolumn{2}{|c|}{\text{norms}}\\*\hline
r_{1} & -1 & 0 & -6 & -4 & -3 & -3 & 1 & 0 \\*
r_{2} & -6 & -3 & -55 & -39 & -32 & -23 & 14 & 9 \\*
r_{3} & 3 & 2 & 33 & 23 & 20 & 14 & 2 & 1 \\*
r_{4} & 155 & 110 & 1706 & 1206 & 1018 & 720 & 46 & 32 \\*
r_{5} & 77 & 54 & 840 & 594 & 500 & 354 & 6 & 4 \\*
r_{6} & 599 & 424 & 6553 & 4634 & 3898 & 2756 & 78 & 55 \\*
r_{7} & 68 & 48 & 741 & 524 & 440 & 311 & 10 & 7 \\*
r_{8} & 150 & 106 & 1628 & 1151 & 963 & 681 & 133 & 94 \\*
\hline
\end{array}$}
\\
\end{longtable}
\end{center}

%% file: decagons.tex
\begin{center}
\begin{longtable}{cc}
\caption{Decagons}\\
\endfirsthead

\multicolumn{2}{c}
{{ \tablename\ \thetable{}, cont.}} \\
\hline\\
\endhead

\endfoot

\endlastfoot

\hline\hline
det = $-194040 - 137207w$ &  det norm = $-98$ \\*
\multicolumn{2}{c}{$2222222222$}\\*
\hline\\*
\multicolumn{2}{c}{quadratic form}\\*
\multicolumn{2}{c}{
$\left(\begin{array}{cc|cc|cc}
-1055 & -746 & 0 & 0 & 0 & 0 \\*
0 & 0 & 10 & 7 & 0 & 0 \\*
0 & 0 & 0 & 0 & 5 & 3 \\*
\end{array}\right)$}
\\\\
\multicolumn{2}{c}{root list}\\*
\multicolumn{2}{c}{
$\begin{array}{|l|cc|cc|cc||cc|}
\hline
\multicolumn{7}{|c||}{\text{roots}} & \multicolumn{2}{|c|}{\text{norms}}\\*\hline
r_{1} & 0 & 1 & -3 & -2 & -10 & -7 & 58 & 41 \\*
r_{2} & 1 & 2 & 0 & 0 & -31 & -22 & 2366 & 1673 \\*
r_{3} & -1 & 0 & 3 & 2 & 7 & 5 & 198 & 140 \\*
r_{4} & 3 & 3 & -31 & -22 & -31 & -22 & 874 & 618 \\*
r_{5} & 21 & 14 & -157 & -111 & -205 & -145 & 1830 & 1294 \\*
r_{6} & 11 & 7 & -78 & -55 & -109 & -77 & 58 & 41 \\*
r_{7} & 118 & 82 & -861 & -609 & -1236 & -874 & 2366 & 1673 \\*
r_{8} & 3 & 1 & -16 & -11 & -24 & -17 & 17 & 12 \\*
r_{9} & 4 & 1 & -18 & -13 & -31 & -22 & 75 & 53 \\*
r_{10} & 1 & 2 & -11 & -8 & -24 & -17 & 157 & 111 \\*
\hline
\end{array}$}
\\
\end{longtable}
\end{center}

%% file: undecagons.tex
\begin{center}
\begin{longtable}{cc}
\caption{Undecagons}\\
\endfirsthead

\multicolumn{2}{c}
{{ \tablename\ \thetable{}, cont.}} \\
\hline\\
\endhead

\endfoot

\endlastfoot

\hline\hline
det = $-13114 - 9273w$ &  det norm = $-62$ \\*
\multicolumn{2}{c}{$22242222222$}\\*
\hline\\*
\multicolumn{2}{c}{quadratic form}\\*
\multicolumn{2}{c}{
$\left(\begin{array}{cc|cc|cc}
94 & 66 & 38 & 28 & -19 & -14 \\*
38 & 28 & 14 & 7 & -2 & 0 \\*
-19 & -14 & -2 & 0 & 1 & 0 \\*
\end{array}\right)$}
\\\\
\multicolumn{2}{c}{root list}\\*
\multicolumn{2}{c}{
$\begin{array}{|l|cc|cc|cc||cc|}
\hline
\multicolumn{7}{|c||}{\text{roots}} & \multicolumn{2}{|c|}{\text{norms}}\\*\hline
r_{1} & 1 & 0 & -2 & -1 & -2 & 0 & 34 & 24 \\*
r_{2} & 1 & 1 & 0 & 0 & 0 & 0 & 546 & 386 \\*
r_{3} & 0 & 0 & 1 & 0 & 2 & 0 & 10 & 7 \\*
r_{4} & 0 & 0 & 0 & 0 & 1 & 1 & 3 & 2 \\*
r_{5} & -1 & 1 & -2 & -1 & 4 & 2 & 6 & 4 \\*
r_{6} & 2 & -1 & -3 & -1 & 0 & 4 & 2 & 1 \\*
r_{7} & 4 & 1 & -19 & -14 & 19 & 14 & 47 & 33 \\*
r_{8} & 1 & 0 & -4 & -2 & 1 & 3 & 3 & 2 \\*
r_{9} & 6 & 6 & -47 & -33 & 28 & 19 & 160 & 113 \\*
r_{10} & 1 & 0 & -3 & -2 & 0 & 1 & 10 & 7 \\*
r_{11} & 5 & 2 & -19 & -14 & 0 & 0 & 160 & 113 \\*
\hline
\end{array}$}
\\
\end{longtable}
\end{center}

%% file: dodecagons.tex
\begin{center}
\begin{longtable}{cc}
\caption{Dodecagons}\\
\endfirsthead

\multicolumn{2}{c}
{{ \tablename\ \thetable{}, cont.}} \\
\hline\\
\endhead

\endfoot

\endlastfoot

\hline\hline
det = $-461 - 326w$ &  det norm = $-31$ \\*
\multicolumn{2}{c}{$\left(222222\right)^2$}\\*
\hline\\*
\multicolumn{2}{c}{quadratic form}\\*
\multicolumn{2}{c}{
$\left(\begin{array}{cc|cc|cc}
8 & -1 & 8 & -1 & -13 & -10 \\*
8 & -1 & 7 & -2 & -3 & -3 \\*
-13 & -10 & -3 & -3 & 3 & 2 \\*
\end{array}\right)$}
\\\\
\multicolumn{2}{c}{root list}\\*
\multicolumn{2}{c}{
$\begin{array}{|l|cc|cc|cc||cc|}
\hline
\multicolumn{7}{|c||}{\text{roots}} & \multicolumn{2}{|c|}{\text{norms}}\\*\hline
r_{1} & 1 & 0 & -2 & 0 & 4 & -3 & 2 & 1 \\*
r_{2} & 3 & 2 & 0 & 0 & 0 & 0 & 112 & 79 \\*
r_{3} & 0 & 0 & 1 & 1 & 2 & -1 & 3 & 2 \\*
r_{4} & -2 & -2 & 13 & 10 & 7 & 3 & 382 & 270 \\*
r_{5} & 0 & 0 & 0 & 0 & 1 & 0 & 3 & 2 \\*
r_{6} & 15 & 11 & -46 & -33 & 13 & 10 & 191 & 135 \\*
\hline
\end{array}$}
\\
\end{longtable}
\end{center}

%% file: tetradecagons.tex
\begin{center}
\begin{longtable}{cc}
\caption{Tetradecagons}\\
\endfirsthead

\multicolumn{2}{c}
{{ \tablename\ \thetable{}, cont.}} \\
\hline\\
\endhead

\endfoot

\endlastfoot

\hline\hline
det = $-2011 - 1422w$ &  det norm = $-47$ \\*
\multicolumn{2}{c}{$\left(8222222\right)^2$}\\*
\hline\\*
\multicolumn{2}{c}{quadratic form}\\*
\multicolumn{2}{c}{
$\left(\begin{array}{cc|cc|cc}
447 & 316 & -160 & -114 & 84 & 59 \\*
-160 & -114 & 66 & 36 & -27 & -24 \\*
84 & 59 & -27 & -24 & 16 & 9 \\*
\end{array}\right)$}
\\\\
\multicolumn{2}{c}{root list}\\*
\multicolumn{2}{c}{
$\begin{array}{|l|cc|cc|cc||cc|}
\hline
\multicolumn{7}{|c||}{\text{roots}} & \multicolumn{2}{|c|}{\text{norms}}\\*\hline
r_{1} & 0 & 0 & -1 & -1 & -4 & -2 & 6 & 4 \\*
r_{2} & 2 & -1 & 0 & 1 & 1 & -1 & 2 & 1 \\*
r_{3} & -68 & -50 & 0 & 0 & 521 & 369 & 84 & 59 \\*
r_{4} & -13 & -3 & 0 & -1 & 62 & 45 & 1 & 0 \\*
r_{5} & -40 & -31 & -5 & -4 & 299 & 213 & 2 & 1 \\*
r_{6} & -751 & -537 & -109 & -76 & 5373 & 3798 & 25 & 17 \\*
r_{7} & -61 & -46 & -9 & -7 & 446 & 317 & 1 & 0 \\*
\hline
\end{array}$}
\\
\end{longtable}
\end{center}

%% file: hexadecagons.tex
\begin{center}
\begin{longtable}{cc}
\caption{Hexadecagons}\\
\endfirsthead

\multicolumn{2}{c}
{{ \tablename\ \thetable{}, cont.}} \\
\hline\\
\endhead

\endfoot

\endlastfoot

\hline\hline
det = $-5712 - 4039w$ &  det norm = $-98$ \\*
\multicolumn{2}{c}{$\left(22222222\right)^2$}\\*
\hline\\*
\multicolumn{2}{c}{quadratic form}\\*
\multicolumn{2}{c}{
$\left(\begin{array}{cc|cc|cc}
58 & 41 & -44 & -31 & 10 & 7 \\*
-44 & -31 & 32 & 20 & -3 & -1 \\*
10 & 7 & -3 & -1 & 1 & 0 \\*
\end{array}\right)$}
\\\\
\multicolumn{2}{c}{root list}\\*
\multicolumn{2}{c}{
$\begin{array}{|l|cc|cc|cc||cc|}
\hline
\multicolumn{7}{|c||}{\text{roots}} & \multicolumn{2}{|c|}{\text{norms}}\\*\hline
r_{1} & -1 & -2 & -3 & -2 & -2 & -1 & 92 & 65 \\*
r_{2} & 1 & -1 & 0 & 0 & 0 & 0 & 10 & 7 \\*
r_{3} & 14 & 7 & 17 & 12 & 5 & 4 & 119 & 84 \\*
r_{4} & 2 & 3 & 4 & 3 & 1 & 1 & 3 & 2 \\*
r_{5} & 55 & 37 & 68 & 48 & 18 & 13 & 44 & 31 \\*
r_{6} & 22 & 13 & 25 & 18 & 6 & 5 & 6 & 4 \\*
r_{7} & 420 & 301 & 523 & 370 & 132 & 93 & 238 & 168 \\*
r_{8} & 35 & 29 & 47 & 33 & 12 & 8 & 2 & 1 \\*
\hline
\end{array}$}
\\
\end{longtable}
\end{center}

%% file: octadecagons.tex
\begin{center}
\begin{longtable}{cc}
\caption{Octadecagons}\\
\endfirsthead

\multicolumn{2}{c}
{{ \tablename\ \thetable{}, cont.}} \\
\hline\\
\endhead

\endfoot

\endlastfoot

\hline\hline
det = $-417 - 295w$ &  det norm = $-161$ \\*
\multicolumn{2}{c}{$\left(223222222\right)^2$}\\*
\hline\\*
\multicolumn{2}{c}{quadratic form}\\*
\multicolumn{2}{c}{
$\left(\begin{array}{cc|cc|cc}
9253 & 6515 & 160 & 124 & 80 & 62 \\*
160 & 124 & 6 & 0 & 3 & 0 \\*
80 & 62 & 3 & 0 & 2 & 0 \\*
\end{array}\right)$}
\\\\
\multicolumn{2}{c}{root list}\\*
\multicolumn{2}{c}{
$\begin{array}{|l|cc|cc|cc||cc|}
\hline
\multicolumn{7}{|c||}{\text{roots}} & \multicolumn{2}{|c|}{\text{norms}}\\*\hline
r_{1} & 2 & 0 & -63 & -48 & 6 & 4 & 6 & 4 \\*
r_{2} & 8 & 6 & -519 & -366 & 0 & 0 & 2014 & 1424 \\*
r_{3} & 0 & 0 & 1 & 1 & -2 & -2 & 6 & 4 \\*
r_{4} & 0 & 0 & 0 & 0 & 1 & 1 & 6 & 4 \\*
r_{5} & -1 & 1 & -18 & -8 & 5 & 3 & 5 & 3 \\*
r_{6} & 1 & 2 & -133 & -91 & 26 & 19 & 19 & 13 \\*
r_{7} & 4 & 2 & -229 & -164 & 38 & 27 & 54 & 38 \\*
r_{8} & 108 & 76 & -7191 & -5086 & 988 & 699 & 218 & 154 \\*
r_{9} & 167 & 119 & -11184 & -7906 & 1516 & 1072 & 27 & 19 \\*
\hline
\end{array}$}
\\
\end{longtable}
\end{center}

%% file: icosagons.tex
\begin{center}
\begin{longtable}{cc}
\caption{Icosagons}\\
\endfirsthead

\multicolumn{2}{c}
{{ \tablename\ \thetable{}, cont.}} \\
\hline\\
\endhead

\endfoot

\endlastfoot

\hline\hline
det = $-666 - 471w$ &  det norm = $-126$ \\*
\multicolumn{2}{c}{$\left(2222222222\right)^2$}\\*
\hline\\*
\multicolumn{2}{c}{quadratic form}\\*
\multicolumn{2}{c}{
$\left(\begin{array}{cc|cc|cc}
174 & 123 & 12 & 9 & 0 & 0 \\*
12 & 9 & 20 & -13 & 3 & -1 \\*
0 & 0 & 3 & -1 & 5 & 3 \\*
\end{array}\right)$}
\\\\
\multicolumn{2}{c}{root list}\\*
\multicolumn{2}{c}{
$\begin{array}{|l|cc|cc|cc||cc|}
\hline
\multicolumn{7}{|c||}{\text{roots}} & \multicolumn{2}{|c|}{\text{norms}}\\*\hline
r_{1} & 1 & 2 & -20 & -14 & 1 & 0 & 3 & 2 \\*
r_{2} & 3 & 1 & -24 & -17 & 0 & 0 & 10 & 7 \\*
r_{3} & 1 & 2 & -27 & -19 & -2 & -2 & 314 & 222 \\*
r_{4} & -1 & 0 & 0 & 0 & 0 & 0 & 174 & 123 \\*
r_{5} & 0 & 0 & 0 & 0 & 3 & 2 & 157 & 111 \\*
r_{6} & 3 & 3 & -34 & -24 & 6 & 4 & 58 & 41 \\*
r_{7} & 13 & 9 & -123 & -87 & 14 & 10 & 198 & 140 \\*
r_{8} & 682 & 482 & -6627 & -4686 & 594 & 420 & 9372 & 6627 \\*
r_{9} & 88 & 62 & -857 & -606 & 72 & 51 & 338 & 239 \\*
r_{10} & 1121 & 793 & -10980 & -7764 & 870 & 615 & 9372 & 6627 \\*
\hline
\end{array}$}
\\
\end{longtable}
\end{center}

%% file: 24-gons.tex
\begin{center}
\begin{longtable}{cc}
\caption{24-gons}\\
\endfirsthead

\multicolumn{2}{c}
{{ \tablename\ \thetable{}, cont.}} \\
\hline\\
\endhead

\endfoot

\endlastfoot

\hline\hline
det = $-10369 - 7332w$ &  det norm = $-287$ \\*
\multicolumn{2}{c}{$\left(422222222222\right)^2$}\\*
\hline\\*
\multicolumn{2}{c}{quadratic form}\\*
\multicolumn{2}{c}{
$\left(\begin{array}{cc|cc|cc}
43879 & 31022 & -4894 & -3460 & -204 & -152 \\*
-4894 & -3460 & 546 & 386 & 23 & 17 \\*
-204 & -152 & 23 & 17 & 10 & -6 \\*
\end{array}\right)$}
\\\\
\multicolumn{2}{c}{root list}\\*
\multicolumn{2}{c}{
$\begin{array}{|l|cc|cc|cc||cc|}
\hline
\multicolumn{7}{|c||}{\text{roots}} & \multicolumn{2}{|c|}{\text{norms}}\\*\hline
r_{1} & 2 & 0 & 13 & 4 & -22 & -15 & 34 & 24 \\*
r_{2} & -1 & 0 & -6 & -2 & -1 & -1 & 17 & 12 \\*
r_{3} & -26 & -18 & -254 & -178 & 1157 & 818 & 874 & 618 \\*
r_{4} & -175 & -124 & -1750 & -1238 & 9229 & 6526 & 973 & 688 \\*
r_{5} & -21 & -14 & -207 & -143 & 1112 & 786 & 17 & 12 \\*
r_{6} & -522 & -370 & -5261 & -3723 & 29466 & 20836 & 8590 & 6074 \\*
r_{7} & -19 & -13 & -190 & -133 & 1092 & 772 & 17 & 12 \\*
r_{8} & -208 & -146 & -2109 & -1487 & 12580 & 8895 & 1946 & 1376 \\*
r_{9} & -34 & -24 & -347 & -245 & 2134 & 1509 & 34 & 24 \\*
r_{10} & -889 & -628 & -9098 & -6430 & 57044 & 40336 & 167 & 118 \\*
r_{11} & -196 & -140 & -2015 & -1430 & 12747 & 9014 & 150 & 106 \\*
r_{12} & -15 & -12 & -160 & -119 & 1050 & 743 & 3 & 2 \\*
\hline
\end{array}$}
\\
\end{longtable}
\end{center}